\documentclass[final,oneside,notitlepage,10pt]{article}

\usepackage{amsthm}
\usepackage{thmtools}
\usepackage{graphicx}
\usepackage[margin=2cm,font=normalsize,labelfont=bf]{caption}
\usepackage{verbatim}
\usepackage{enumerate}
\usepackage{fancyhdr}
\usepackage[]{geometry}
\usepackage{xstring}
\usepackage{needspace}
\usepackage{color}
\usepackage{amstext}
\usepackage{hyperref}
\usepackage{mathdots}
\usepackage{soul}

\makeatletter
\newcounter{denseversion}
\newcounter{comments}
\newcounter{authorcounter}
\newcounter{adresscounter}

\def\title#1{\gdef\@title{#1}}
\def\@title{}
\def\subtitle#1{\gdef\@subtitle{#1}}
\def\@subtitle{}

\def\authortagsused{0}
\def\adresstag#1{\if!#1!\else$^{\;#1\;}$\fi}

\renewcommand{\author}[2][]{
  \stepcounter{authorcounter}
  \if!#1!\else\gdef\authortagsused{1}\fi
  \ifnum\value{authorcounter}=1
    \def\@authorstringa{#2\adresstag{#1}}
    \def\@authorstringb{#2}
    \def\@authorstringc{#2\adresstag{#1}}
  \else
    \g@addto@macro\@authorstringa{\ and #2\adresstag{#1}}
    \g@addto@macro\@authorstringb{\ and #2}
    \g@addto@macro\@authorstringc{\\#2\adresstag{#1}}
  \fi}
\def\@author{\ifnum\value{denseversion}=0\@authorstringa\else\@authorstringb\fi}

\def\@adressstringa{}
\def\@adressstringb{}
\newcommand{\adress}[2][]{
  \stepcounter{adresscounter}
  \ifnum\value{adresscounter}=1
    \g@addto@macro\@adressstringa{\ifnum\authortagsused=0\def\br{\\}\else\def\br{, }\fi\adresstag{#1}#2}
    \g@addto@macro\@adressstringb{\def\br{\\}\adresstag{#1}\parbox[t]{14cm}{#2}}
  \else
    \g@addto@macro\@adressstringa{\\[\bigskipamount]\adresstag{#1}#2}
    \g@addto@macro\@adressstringb{\\[\medskipamount]\adresstag{#1}\parbox[t]{14cm}{#2}}
  \fi}

\def\preprint#1{\gdef\@preprint{#1}}
\def\@preprint{}
\def\keywords#1{\gdef\@keywords{#1}}
\def\@keywords{}
\def\msc#1{\gdef\@msc{#1}}
\def\@msc{}
\def\email#1{
   \gdef\@email{#1}
   \g@addto@macro\@authorstringc{ {\it (#1)}}}
\def\@email{}
\def\dedication#1{\gdef\@dedication{#1}}
\def\@dedication{}

\def\mybaselinestretch#1{\gdef\@mybaselinestretch{#1}}
\def\@mybaselinestretch{}

\def\myparskip#1{\gdef\@myparskip{#1}}
\def\@myparskip{}

\binoppenalty=\maxdimen
\relpenalty=\maxdimen

\mybaselinestretch{1.2}
\myparskip{1ex}
\renewcommand{\baselinestretch}{\@mybaselinestretch}
\def\denseversion{
  \setcounter{denseversion}{1}
  \newgeometry{left=3cm,right=3cm,top=3cm}
  \mybaselinestretch{1.1}
  \myparskip{0.5ex}
  \renewcommand{\baselinestretch}{\@mybaselinestretch}
  \setlength{\parskip}{\@myparskip}
  \normalfont
  \def\possiblelinebreak{}
  \fancyfoot[C]{\itshape{--$\,\,$\thepage$\,\,$--}}}

\def\possiblelinebreak{\\}

\renewcommand{\emph}[1]{\def\reserved@a{it}\ifx\f@shape\reserved@a\ul{#1}\else\textit{#1}\fi}

\newcommand{\mytableofcontents}{
   \ifnum\value{denseversion}=0
     \tableofcontents
   \else
     \renewcommand{\baselinestretch}{1.1}
     \setlength{\parskip}{0ex}
     \normalfont
     \begingroup
     \def\addvspace##1{\vskip0.4em}
     \tableofcontents
     \endgroup
     \renewcommand{\baselinestretch}{\@mybaselinestretch}
     \setlength{\parskip}{\@myparskip}
     \normalfont
   \fi}

\newlength{\zeilenlaenge}
\def\putindent#1{
  \settowidth{\zeilenlaenge}{#1}
  \ifnum\zeilenlaenge>\textwidth
    #1
  \else
    \noindent #1
  \fi
}

\def\href#1#2{#2}

\def\pdfdaten{
  \hypersetup{
    linktocpage = true,
    pdftitle = {\@title},
    pdfauthor = {\@author},
    pdfkeywords = {\@keywords},    
    bookmarksopen = true,
    bookmarksopenlevel = 1
  }}  

\def\showkeywords{\begin{flushleft}\footnotesize\textbf{Keywords}: \@keywords\end{flushleft}}
\def\showmsc{\begin{flushleft}\footnotesize\textbf{MSC 2010}: \@msc\end{flushleft}}

\def\tocsection#1{\section*{#1}\addcontentsline{toc}{section}{#1}}

\def\mytitle{}
\def\zmptitle{
  \begin{tabular}{cc}
    \begin{minipage}[c]{0.4\textwidth}
      \begin{flushleft}
        \includegraphics[width=110pt]{../../tex/zmp}
      \end{flushleft}  
    \end{minipage}&
    \begin{minipage}[c]{0.55\textwidth}
      \begin{flushright}
      {\small\sf\@preprint}
      \end{flushright}
    \end{minipage}
  \end{tabular}
  \vskip 2cm}

\def\maketitle{
  \pdfdaten
  \newpage
  \noindent
  \mytitle
  \begin{center}
    \LARGE\@title\\
    \if!\@subtitle!\else\smallskip\LARGE\@subtitle\\\fi
    \bigskip
    \if!\@author!\else\bigskip\large\@author\\\fi
    \ifnum\value{denseversion}=0
      \if!\@adressstringa!\else\bigskip\normalsize\@adressstringa\\\fi
      \if!\@email!\else\ifnum\value{authorcounter}=1\bigskip\normalsize\textit{\@email}\\\else\fi\fi
    \else
    \fi
    \if!\@dedication!\else\bigskip\normalsize{\@dedication}\\\fi
  \end{center}
  \ifnum\value{denseversion}=0\vskip 1.5cm\else\vskip0.5cm\fi
  \if!\@draft!\else\thispagestyle{empty}\fi}

\def\kobib#1{
  \begin{raggedright}
  \ifnum\value{denseversion}=0\else\small\fi
  \Oldbibliography{#1/kobib}
  \bibliographystyle{#1/kobib}
  \end{raggedright}
  \ifnum\value{denseversion}=0\else
      \noindent
      \if!\@authorstringc!\else
        \ifnum\authortagsused=0\ifnum\value{authorcounter}>1\normalsize\@authorstringc\\[\medskipamount]\else\fi\else\normalsize\@authorstringc\\[\medskipamount]\fi
      \fi
      \if!\@adressstringb!\else\normalsize\@adressstringb\\{}\fi
      \ifnum\authortagsused=0
        \ifnum\value{authorcounter}=1
          \if!\@email!\else\linebreak\normalsize\textit{\@email}\\{}\fi
        \else
        \fi
      \else
      \fi
  \fi}

\let\Oldbibliography\bibliography
\def\bibliography#1{
  \begin{raggedright}
  \ifnum\value{denseversion}=0\else\small\fi
  \Oldbibliography{#1}
  \end{raggedright}
  \ifnum\value{denseversion}=0\else
      \medskip
      \noindent
      \if!\@authorstringc!\else
        \ifnum\authortagsused=0\ifnum\value{authorcounter}>1\normalsize\@authorstringc\\[\medskipamount]\else\fi\else\normalsize\@authorstringc\\[\medskipamount]\fi
      \fi
      \if!\@adressstringb!\else\normalsize\@adressstringb\\{}\fi
      \ifnum\authortagsused=0
        \ifnum\value{authorcounter}=1
          \if!\@email!\else\linebreak\normalsize\textit{\@email}\\{}\fi
        \else
        \fi
      \else
      \fi
  \fi
}

\newenvironment{commentfigure}{\begin{comment}}{\end{comment}}
\newenvironment{sidewayscommentfigure}{\begin{minipage}}{\end{minipage}}

\def\comments{
  \setcounter{comments}{1}
  \renewenvironment{comment}{\begin{list}{}{\rightmargin=1cm\leftmargin=1cm}\item\sf\begin{small}}{\end{small}\end{list}}
  
  }

\def\draftstamp#1#2#3{
  \ifnum\value{comments}=0
    \gdef\@draft{DRAFT - Version #1 - Last edited on #2 by #3 - Comments are not displayed}
  \else
      \gdef\@draft{DRAFT - Version #1 - Last edited on #2 by #3 - Comments are displayed}
  \fi
  \fancyhead[C]{\footnotesize\tt\textcolor{red}{\@draft}}}

\def\draft#1#2#3#4{
  \ifnum#4=1\comments\else\fi
  \draftstamp}
\def\@draft{}

\makeatother

\input{kobib}
\usepackage{amssymb}
\usepackage{amsmath}
\usepackage{amstext}
\usepackage{mathrsfs}
\usepackage{euscript}
\usepackage[all]{xy}
\usepackage{xstring}
\usepackage{slashed}
\usepackage{tensor}

\newcounter{marke}
\stepcounter{marke}

\def\N {\mathbb{N}}
\def\Z {\mathbb{Z}}

\def\R {\mathbb{R}}

\def\id{\mathrm{id}}

\def\hc#1{\mathrm{h}_{#1}}
\def\h {\mathrm{H}}

\def\subset{\subseteq}

\def\sep{\;|\;}
\def\maps{\colon}
\def\df{:=}

\renewcommand{\varepsilon}{\epsilon}
\renewcommand{\to}{\!\xymatrix@R=0cm@C=1.4em{\ar[r] &}}
\renewcommand{\mapsto}{\!\xymatrix@R=0cm@C=1.4em{\ar@{|->}[r] &}\!}
\renewcommand{\Rightarrow}{\!\xymatrix@R=0cm@C=1.4em{\ar@{=>}[r] &}\!}
\renewcommand{\Leftarrow}{\!\xymatrix@R=0cm@C=1.4em{\ar@{<=}[r] &}\!}
\newcommand{\incl}{\!\xymatrix@R=0cm@C=1.4em{\ar@{^(->}[r] &}\!}

\renewcommand\Leftrightarrow{\!\xymatrix@R=0cm@C=1.4em{\ar@{<=>}[r] &}\!}

\def\gdw{\Leftrightarrow}

\makeatletter
\renewenvironment{proof}[1][\nameProof]
  {\par\pushQED{\qed}%
   \normalfont \topsep6\p@\@plus6\p@\relax
   \trivlist
   \item[\hskip\labelsep
         \itshape
         #1\@addpunct{.}]
  \leavevmode}
  {\popQED\endtrivlist\@endpefalse}
\makeatother

\def\notebox#1#2{\begin{minipage}[b]{#1}\sloppy\renewcommand{\baselinestretch}{0.8}\footnotesize \begin{center}#2\end{center}\end{minipage}}

\def\mquad{\hspace{-2em}}

\def\erf#1{(\ref{#1})}

\def\stackref#1#2{\stackrel{\text{\ref{#1}}}{#2}}
\def\eqref#1{\stackref{#1}{=}}

\newlength{\myeqt} % Darin wird die Länge des übergebenen Textes abgespeichert
\newlength{\myeqs} % Darin wird die Länge des übergebenen Symbolds abgespeichert
\newlength{\myeqm} % Wieviel "klein" über die Breite des Symbolds hinausgehen darf
\newlength{\myeqn} % Die Standardbreite für große Boxen
\newcommand\eqtext[2][\myeqn]{\symtext[#1]{#2}{=}}
\newcommand\symtext[3][\myeqn]{
  \settowidth{\myeqt}{#2}
  \settowidth{\myeqs}{$#3$}
  \addtolength{\myeqs}{\the\myeqm}
  \ifdim\myeqt>\myeqs
    \stackrel{\hspace{-#1}\notebox{#1}{\medskip #2 \\ $\downarrow$\smallskip}\hspace{-#1}}{#3}
  \else
    \stackrel{\text{#2}}{#3}
  \fi}
\newcommand\eqcref[2][\myeqn]{\symcref[#1]{#2}{=}}
\newcommand\symcref[3][\myeqn]{\symtext[#1]{\cref{#2}}{#3}}

\def\brackets#1{\IfStrEq{#1}{-}{}{(#1)}}
\def\subindex#1{\IfStrEq{#1}{-}{}{_{#1}}}

\newcommand{\alxydim}[2]{\begin{aligned}\xymatrix#1{#2}\end{aligned}}



\newlength{\myl}

\def\ddt#1#2#3{\left.\frac{\mathrm{d}^{\IfStrEq{#1}{1}{}{#1}}}{\mathrm{d}#2}\IfStrEq{#2}{#3}{\right.}{\right|_{#3}}}

\newcommand{\ueins}{{\mathrm{U}}(1)}

\def\af{\mathcal{A}na^{\infty}}

\def\fun{\mathcal{F}un}
\def\two{2\text{-}}

\def\hom{\mathcal{H}\!om}

\def\act#1#2{#1/\!\!/#2}
\def\idmorph#1{#1_{dis}}

\def\pr{{\mathrm{pr}}}

\newlength{\widthtmp}
\def\length#1{\settowidth{\widthtmp}{#1}\the\widthtmp}

\def\lli#1{\tensor{_#1}{}}
\def\ttimes#1#2{\hspace{-0.15em}\tensor[_{#1}]{\times}{_{#2}}}

\def\buntech#1#2{\mathcal{B}\hspace{-0.01em}un_{\hspace{0.05em}#1}^{#2}}

\def\bun#1#2{\buntech{#1}{}\brackets{#2}}
\def\buncon#1#2{\buntech{#1}{\nabla}\hspace{-0.05em}\brackets{#2}}
\def\bunconmod#1#2#3{\buntech{#1}{\nabla_{\!#3}}\hspace{-0.05em}\brackets{#2}}

\def\zweibun#1#2{\two\bun{#1}{#2}}
\def\zweibuncon#1#2{\two\buncon{#1}{#2}}

\usepackage[latin1]{inputenc}
\usepackage[english]{babel}
\usepackage[english]{cleveref}

\def\refname{References}

\def\quot#1{``#1''}

\def\quand{\quad\text{ and }\quad}
\def\quomma{\quad\text{, }\quad}

\def\nameTheorem{Theorem}
\def\nameTheorems{Theorems}
\def\nameDefinition{Definition}
\def\nameDefinitions{Definitions}
\def\nameCorollary{Corollary}
\def\nameCorollaries{Corollaries}
\def\nameProposition{Proposition}
\def\namePropositions{Propositions}
\def\nameConstruction{Construction}
\def\nameConstructions{Constructions}
\def\nameLemma{Lemma}
\def\nameLemmas{Lemmas}
\def\nameRemark{Remark}
\def\nameRemarks{Remarks}
\def\nameProblem{Problem}
\def\nameProblems{Problems}
\def\nameExample{Example}
\def\nameExamples{Examples}
\def\nameExercise{Exercise}
\def\nameExercises{Exercises}

\def\nameProof{Proof}
\def\nameEquation{Eq.}
\def\nameEquations{Eqs.}
\def\nameSection{Section}
\def\nameSections{Sections}
\def\nameAppendix{Appendix}
\def\nameAppendixs{Appendices}

\input{theorems}

\def\splitting{transition span}
\def\Splitting{Transition span}

\def\ff{f\!f}
\def\flat{f\!lat}
\def\gen{g\hspace{-0.08em}e\hspace{-0.06em}n}
\def\con#1#2{\mathcal{C}\!on_{#1}\brackets{#2}}

\def\conff#1#2{\mathcal{C}\!on^{{f\!f}}_{#1}\brackets{#2}}

\def\inf#1{\EuScript{#1}}

\def\zweibunconreg#1#2{\two\bunconmod{#1}{#2}{r\!e\!g}}
\def\zweibunconff#1#2{\two\bunconmod{#1}{#2}{f\!f}}

\def\fa#1{{#1}^{a}}
\def\fb#1{{#1}^{b}}
\def\fc#1{{#1}^{c}}

\def\ob#1{\mathrm{Obj}(#1)}
\def\mor#1{\mathrm{Mor}(#1)}
\def\1mor#1{1\text{-}\mathrm{Mor}(#1)}
\def\2mor#1{2\text{-}\mathrm{Mor}(#1)}

\title{A global perspective to connections on principal 2-bundles}
\author{Konrad Waldorf}
\email{konrad.waldorf@uni-greifswald.de}

\adress{Ernst-Moritz-Arndt-Universität Greifswald\\
Institut für Mathematik und Informatik\\
Walther-Rathenau-Str. 47\\
D-17487 Greifswald}

\msc{}

\denseversion
\usepackage{amstext}
\usepackage{amssymb}
\usepackage{amsmath}

\begin{document}

%\draftstamp{1.0}{21.3.2013}{Konrad}

\maketitle 

\begin{abstract}
For a strict Lie 2-group, we develop a notion  of Lie 2-algebra-valued differential forms on Lie groupoids, furnishing a differential  graded-commutative Lie algebra equipped with an adjoint action of the Lie 2-group and a pullback operation along Morita equivalences between Lie groupoids. Using this notion, we define connections on principal 2-bundles  as  Lie 2-algebra-valued 1-forms on the total space Lie groupoid of the 2-bundle, satisfying a condition in complete analogy to connections on ordinary principal bundles. We carefully treat various notions of curvature, and prove a classification result by the non-abelian differential cohomology of Breen-Messing.   This provides a consistent, global perspective to higher gauge theory.

%\showkeywords
%\showmsc
\end{abstract}

\setcounter{tocdepth}{3}
\mytableofcontents

%\noarxiv
%\nolinks

\setsecnumdepth{2}

\thispagestyle{empty}

\section{Introduction}

We develop a new, global perspective to connections on principal 2-bundles, enjoying various nice analogies to  connections on ordinary principal bundles. The underlying principal 2-bundles have Lie groupoids $\inf P$ as \quot{total spaces}, which are fibred over a smooth manifold $M$ and carry an action  of a strict Lie 2-group $\Gamma$. As a first step we introduce, for $\gamma$ the Lie 2-algebra of $\Gamma$, a theory of $\gamma$-valued differential forms over Lie groupoids, furnishing a differential graded Lie algebra with an adjoint action of  the 2-group $\Gamma$. Using this new language, the definition of a connection on a principal $\Gamma$-2-bundle (\cref{def:connection}) becomes so simple that we can repeat it here in almost one line: a  \emph{connection} on $\inf P$ is a $\gamma$-valued 1-form $\Omega \in \Omega^1(\inf P,\gamma)$ such that the equation
\begin{equation*}
R^{*}\Omega = \mathrm{Ad}_{\pr_\Gamma}^{-1}(\pr_\inf P^{*}\Omega) + \pr_\Gamma^{*}\Theta
\end{equation*}
holds over $\inf P \times \Gamma$, where $\pr_{\inf P}$ and $\pr_{\Gamma}$ are the projections to the two factors, and  $R:\inf P \times \Gamma \to \inf P$ is the action. Further, $\Theta\in\Omega^1(\Gamma,\gamma)$ is a canonical $\gamma$-valued 1-form that we discover on every Lie 2-group $\Gamma$; a higher-categorical analog of the Maurer-Cartan form of an ordinary Lie group.
We remark that above equation is literally the condition that characterizes connections on ordinary principal bundles. Our hope is that due to this analogy with ordinary gauge theory, our new approach  conceptually illuminates higher gauge theory and makes it better accessible.

Before we explain the results of this paper in more detail, let us try to overview some steps in the development of principal 2-bundles. 
The idea of locally trivial fibre 2-bundles in the sense of a total space Lie groupoid has been introduced by
Bartels \cite{bartels} and Baez-Schreiber \cite{baez2}. A characterization of \emph{principal} fibre 2-bundles has then been developed in terms of transition data, or cocycles, with respect to an open cover, in work of Schreiber and Baez  \cite{schreiber4,baez2}. This approach leads directly to Giraud's non-abelian cohomology \cite{giraud} as the corresponding classifying theory.

The first realization of a principal 2-bundle in the sense used here, has been developed by Wockel \cite{wockel1}. Wockel also describes a classification of principal 2-bundles up to Morita equivalence by non-abelian cohomology.  Schommer-Pries  embedded Wockel's principal 2-bundles into a more general framework, and assembled a whole  bicategory of principal 2-bundles \cite{pries2}.

Different approaches   are non-abelian bundle gerbes \cite{aschieri}, which are equivalent to principal 2-bundles as 2-stacks \cite{Nikolaus}, and  $G$-gerbes \cite{Laurent-Gengoux}, whose equivalence to principal 2-bundles was shown in \cite{ginot1}. Here, $\Gamma$  is the automorphism 2-group of a Lie group $G$.

The theory of connections on the before mentioned versions of principal 2-bundles is well-developed in several aspects, in particular by Breen-Messing \cite{breen1}, Baez and  Schreiber \cite{schreiber4,baez2,Schreiber2011}, and also Jurco-Sämann-Wolf \cite{Jurco2015} in terms of transition data,  by Aschieri-Cantini-Jurco in terms of bundle gerbes \cite{aschieri}, and by Laurent-Gengoux-Stiénon-Xu in terms of $G$-gerbes  \cite{Laurent-Gengoux}. However, no treatment of connections on Wockel's principal 2-bundles  is available. The present article closes this gap.

One other aspect that is not much discussed is the question of existence of connections on non-abelian gerbes. Only for abelian 2-groups, existence of connections was proved by Murray \cite{murray}, and \cite{Laurent-Gengoux} contains a discussion for $G$-gerbes. Using our new approach, we makes at least a  small contribution to this question (\cref{th:existence}).
A further motivation for this article is the increasing number of physical applications of connections on non-abelian gerbes, see e.g.  \cite{Gukova,Fiorenza,Parzygnat2015,Jurco}, for which our new perspective might be useful. 

\noindent
Let us now describe in more detail the results of this paper.
\begin{itemize}

\item 
We define a whole  bicategory of principal 2-bundles with connection, whose 1-morphisms are anafunctors (a certain version of a Morita equivalence, a.k.a. Hilsum-Skandalis maps or bibundles) equipped with additional differential form data.

\item
We detect four different classes of connections, ordered by increasing generality: flat, fake-flat, regular, and fully general. Here, \quot{fake-flat} has precisely the meaning of Breen-Messing \cite{breen1}, and \quot{regular} corresponds to the usual connections from the transition data picture. The fully general connections are not so attractive from the transition data point of view, but arise naturally in our global picture. 

\item
We provide equivalences between our four classes of principal 2-bundles with connection and corresponding versions of non-abelian differential cohomology, established by explicit reconstruction and extraction procedures.

\item
We prove an existence result for connections (in the fully general sense) on principal $\Gamma$-bundles, under a mild assumption on the $\Gamma$-action. This assumption is, for instance, satisfied in the case of an abelian 2-group, so that our result is an extension of Murray's existence result \cite{murray}.  

\end{itemize}

The organization of the paper is straightforward. In \cref{sec:pre} we recall some important concepts and results from the theory of Lie groupoids and Lie 2-groups.  In \cref{sec:2bundles} we set up the theory of principal 2-bundles on the basis of \cite{wockel1,Nikolaus}. As a new tool for working with principal 2-bundles we introduce the notion of a \emph{\splitting} for total spaces of principal 2-bundles and anafunctors (\cref{def:Rsection}). This tool will be used frequently in the forthcoming discussion and subsequent papers.
In \cref{sec:diffforms} we introduce our new notion of Lie 2-algebra-valued differential forms on Lie groupoids; this section is disconnected from the theory of principal 2-bundles, and might be useful for other purposes. In \cref{sec:conn} we introduce and study connections on principal 2-bundles. For the convenience of the reader we include as a one-page appendix a formulary for calculations in strict Lie 2-algebras. 

There are two subsequent papers that continue our new approach: in \cite{Waldorf2017} we  construct the parallel transport of a connection in a principal 2-bundle, using horizontal lifts of paths and surfaces to its  total space. In a second paper we will prove that the theory of  connections on principal 2-bundles developed here is equivalent (as 2-stacks) to connections on non-abelian bundle gerbes, and so obtain a fully consistent and complete picture.

\paragraph{Acknowledgements.} This work was supported by the German Research Foundation under project code WA 3300/1-1.

\setsecnumdepth{2}

\section{Preliminaries}

\label{sec:pre}

\subsection{Lie 2-groups and Lie 2-algebras}

\label{sec:lie2}

We fix the following notation  which will be used throughout the present article. 
If $\inf X$ is a Lie groupoid, we denote by $s,t:\mor{\inf X} \to \ob{\inf X}$ the source and target maps, respectively, and by 
\begin{equation*}
\circ : \mor{\inf X} \ttimes s t \mor{\inf X} \to \mor{\inf X}
\end{equation*}
the composition.
We recall that a (strict) Lie 2-group is a Lie groupoid $\Gamma$ with smooth functors
\begin{equation*}
m: \Gamma \times \Gamma \to \Gamma
\quand
i:\Gamma \to \Gamma
\end{equation*}
that  satisfy strictly the axioms of multiplication and inversion of a group, see e.g. \cite{baez7,baez2}.  Lie 2-groups are in one-to-one correspondence with crossed modules (of Lie groups): quadruples $(G,H,t,\alpha)$ consisting of Lie groups $G$ and $H$, a Lie group homomorphism $t:H \to G$, and a smooth action $\alpha:G \times H \to H$ of $G$ on $H$ by Lie group homomorphisms, such that
\begin{equation*}
\alpha(t(h),x)=hxh^{-1}
\quand
t(\alpha(g,h))=gt(h)g^{-1}
\end{equation*}
for all $g\in G$ and $h,x\in H$. The correspondence identifies $\ob{\Gamma}=G$ and $\mor{\Gamma}=H \ltimes G$ (the semi-direct product  with respect to $\alpha$),  $s(h,g)=g$ and $t(h,g)=t(h)g$ (the double usage of $t$ is unavoidable), and $(h_2,g_2)\circ (h_1,g_1)=(h_2h_1,g_1)$. We will instantly switch to the crossed module perspective whenever it seems to be  convenient.

For a crossed module  $(G,H,t,\alpha)$
we  denote by  $\alpha_g \in \mathrm{Aut}(H)$  the action of a fixed $g\in G$ on $H$. For $h \in H$ we consider the map $\tilde \alpha_h: G \to H$ defined by $\tilde \alpha_h(g) := h^{-1}\alpha(g,h)$, which is not a Lie group homomorphism but still satisfies $\tilde \alpha_h(1)=1$.

Passing  to the Lie algebras $\mathfrak{g}$ and $\mathfrak{h}$  of $G$ and $H$, respectively, the differential of $t$ gives a Lie algebra homomorphism $t_{*}:\mathfrak{h}\to \mathfrak{g}$, and the differential of $G \to \mathrm{Aut}(H):g \mapsto \alpha_g$ gives a Lie algebra action $\alpha_{*}:\mathfrak{g} \times \mathfrak{h} \to \mathfrak{h}$ of $\mathfrak{g}$ on $\mathfrak{h}$ by derivations, satisfying
\begin{equation*}
\alpha_{*}(t_{*}(Y_1),Y_2) = [Y_1,Y_2]
\quand
t_{*}\alpha_{*}(X,Y) = [X,t_{*}(Y)]
\end{equation*}
for all $X\in \mathfrak{g}$ and $Y_1,Y_2\in \mathfrak{h}$. 
\begin{comment}
This means that $\alpha_{*}$ is bilinear and satisfies
\begin{eqnarray}
\alpha_{*}([X_1,X_2],Y)&=& \alpha_{*}(X_1,\alpha_{*}(X_2,Y))-\alpha_{*}(X_2,\alpha_{*}(X_1,Y))
\label{cm3}
\\
\alpha_{*}(X,[Y_1,Y_2])&=&[\alpha_{*}(X,Y_1),Y_2]+[Y_1,\alpha_{*}(X,Y_2)]\text{,}
\label{cm4}
\end{eqnarray}
\end{comment}
The collection $\gamma := (\mathfrak{h},\mathfrak{g},t_{*},\alpha_{*})$ forms the \emph{Lie 2-algebra} of the Lie 2-group $\Gamma$. Apart from the usual adjoint actions of $H$ on $\mathfrak{h}$ and $G$ on $\mathfrak{g}$, we have -- for $g\in G$ -- the differential of $\alpha_g$, which is a Lie algebra homomorphism $(\alpha_g)_{*}: \mathfrak{h} \to \mathfrak{h}$,  and -- for $h\in H$ -- the differential of $\tilde\alpha_h$, which is a linear map $(\tilde \alpha_h)_{*}: \mathfrak{g} \to \mathfrak{h}$.

\subsection{Non-abelian differential cohomology}

\label{sec:nonabdiffcoho}

Non-abelian differential cohomology unifies Giraud's non-abelian cohomology \cite{giraud} and differential cohomology, see e.g. \cite{breen1,schreiber2}. Let $\Gamma$ be a Lie 2-group with  crossed module $(G,H,t,\alpha)$.  Let $M$ be a smooth manifold and let $\mathcal{U}=\{U_i\}_{i\in I}$ be a cover of $M$ by open sets. 

\begin{definition}
A differential $\Gamma$-cocycle for $\mathcal{U}$ consists of the following data:
\begin{enumerate}
\item[(a)] 
On every open set $U_i$, a 1-form $A_i \in \Omega^1(U_i,\mathfrak{g})$ and a 2-form $B_i \in \Omega^2(U_i,\mathfrak{h})$.

\item[(b)]
On every two-fold intersection $U_i \cap U_j$, a smooth map $g_{ij}:U_i \cap U_j \to G$ and a 1-form $\varphi_{ij} \in \Omega^1(U_i \cap U_j,\mathfrak{h})$.

\item[(c)]
On every three-fold intersection $U_i \cap U_j \cap U_k$, a smooth map $
a_{ijk}:U_i \cap U_j \cap U_k \to H$. 
\end{enumerate}  
The cocycle conditions are the following:
\begin{enumerate}

\item
Over every two-fold intersection $U_i \cap U_j$,
\begin{align}
A_j + t_{*} (\varphi_{ij}) &= \mathrm{Ad}_{g_{ij}}(A_i) - g_{ij}^{*}\bar\theta 
\label{eq:transconnco}
\\
\label{eq:transcurvco}
B_j + \mathrm{d}\varphi_{ij} + \frac{1}{2}[\varphi_{ij} \wedge \varphi_{ij}]+ \alpha_{*}(A_j \wedge \varphi_{ij})&=(\alpha_{g_{ij}})_{*} (B_i)
\end{align}

\item
Over every three-fold intersection $U_i \cap U_j \cap U_k$,
\begin{align}
\label{eq:transtrans}
g_{ik} &= (t \circ a_{ijk}) \cdot g_{jk}g_{ij}
\\
\label{eq:transconn2co}
\mathrm{Ad}_{a_{ijk}}^{-1}(\varphi_{ik}) +(\tilde \alpha_{a_{ijk}})_{*}(A_k) &=  \varphi_{jk}+(\alpha_{g_{jk}})_{*} (\varphi_{ij}) -a_{ijk}^{*}\theta\text{.}
\end{align}

\item
Over every four-fold intersection $U_i \cap U_j \cap U_k \cap U_l$,
\begin{equation}
\label{31c}
a_{ikl} \alpha(g_{kl},a_{ijk})=a_{ijl}a_{jkl}\text{.}
\end{equation}

\end{enumerate} 
\end{definition}
A differential $\Gamma$-cocycle is called \emph{generalized} if \cref{eq:transcurvco} is not required. Any differential $\Gamma$-cocycle
furnishes two notions of curvature:
\begin{enumerate}[(a)]

\item 
the \quot{3-curvature}
\begin{equation*}
\mathrm{curv}(A_i,B_i) := \mathrm{d}B_i + \alpha_{*}( A_i\wedge B_i)\in\Omega^3
(U_i,\mathfrak{h})\end{equation*}

\item
the \quot{fake-curvature}
\begin{equation*}
\mathrm{fcurv}(A_i,B_i) := \mathrm{d}A_i +\frac{1}{2} [A_i
\wedge A_i] - t_{*} ( B_i)\in\Omega^2(U_i,\mathfrak{g}) 
\end{equation*}

\end{enumerate}
We say that a differential $\Gamma$-cocycle is \emph{fake-flat}, if   $\mathrm{fcurv}(A_i,B_i)=0$, and \emph{flat}, if it is fake-flat and $\mathrm{curv}(A_i,B_i) =0$. 
Two differential $\Gamma$-cocycles $(A,B,g,\varphi,a)$ and $(A',B',g',\varphi',a')$ are called \emph{equivalent}, if the following structure exists: 
\begin{enumerate}
\item[(a)]
On every open set $U_i$, a smooth map $h_{i}:U_{i} \to G$ and a 1-form $\phi_{i}\in\Omega^1(U_i,\mathfrak{h})$.
\item[(b)]
On every two-fold intersection $U_i \cap U_j$, a smooth map $e_{ij}:U_i \cap U_j \to H$.
\end{enumerate}
The following conditions have to be satisfied:
\begin{enumerate}
\item 
Over every open set $U_i$,
\begin{align}
\label{32a}
A'_i+ t_{*}(\phi_i) &= \mathrm{Ad}_{h_i}(A_i)  - h_i^{*}\bar\theta
 \\
 \label{eq:equivtranscurv}
 B_i' + \alpha_{*}(A_i' \wedge \phi_i) + \mathrm{d}\phi_i + \frac{1}{2}[\phi_i \wedge \phi_i] &= (\alpha_{h_i})_{*}(B_i)
\end{align}

\item
Over every two-fold intersection $V_i \cap V_j$,
\begin{align}
\label{32c}
g_{ij}'h_i &= t(e_{ij})h_jg_{ij}
\\
\label{32d}
\mathrm{Ad}_{e_{ij}}^{-1}(\varphi_{ij}'+ (\alpha_{g_{ij}'})_{*}(\phi_i))+ (\tilde\alpha_{e_{ij}})_{*}(A'_{j})   &= \phi_j + (\alpha_{h_j})_{*}(\varphi_{ij}) 
- e_{ij}^{*}\theta
\end{align}

\item
Over every three-fold
 intersection $V_i \cap V_j \cap V_k$,
\begin{equation}
\label{32f}
a_{ijk}'\alpha(g'_{jk},e_{ij})e_{jk} = e_{ik}\alpha(h_k,a_{ijk})\text{.}
\end{equation}
\end{enumerate}
An equivalence is called \emph{generalized} if \cref{eq:equivtranscurv} is not required. 

\begin{remark}
\label{rem:normalized}
A  differential $\Gamma$-cocycle is called \emph{normalized}, if $g_{ii}=1$ and $a_{iij}=a_{ijj}=a_{iji}=1$ for all $i,j\in I$. Conditions $g_{ii}=a_{ijj}=a_{iij}=1$ are easy to achieve by passing to an equivalent cocycle using $e_{ii} := a_{iii}$ and $e_{ij}:=1$ for $i\neq j$. 
\begin{comment}
Indeed, then we have
\begin{equation*}
g_{ii}' \eqcref{32c} t(e_{ii})g_{ii}=t(a_{iii})g_{ii}\eqcref{eq:transtrans}1 \end{equation*}
and
\begin{equation*}
a_{iii}'\eqcref{32f} e_{ii}a_{iii}e_{ii}^{-1}e_{ii}^{-1}=1\text{.}
\end{equation*}
The cocycle condition \cref{31c} for the new cocycle then gives for $l=k=j$
\begin{equation*}
a'_{ikk} =a'_{ikk}a'_{kkk}\alpha(g'_{kk},a'_{ikk})^{-1}=1
\end{equation*}
and for $i=j=k$ 
\begin{equation*}
a_{iil}' =a'_{iil}a'_{iil}\alpha(g'_{il},a'_{iii})^{-1}=a'_{iil}a'_{iil}\text{,}
\end{equation*}
hence $a_{iil}'=1$. 
\end{comment}
For condition $a_{iji}=1$ we choose a total order on the index set $I$ of the open cover, eventually after discarding some open sets\footnote{A smooth manifold is second-countable, hence Lindelöf; therefore, every open cover has a countable subcover. Countable sets have total orders (independently of the Axiom of Choice).}. 
\begin{comment}
A countable set, by definition, has an injective map to $\N$. Total orders can be induced along injective maps. 
\end{comment}
Then, we set
\begin{equation*}
e_{ij} := \begin{cases}
1 & \text{ if } i\leq j \\
a_{jij} & \text{ if }i>j\text{.} \\
\end{cases}
\end{equation*}
Passing to an equivalent cocycle preserves $g_{ii}=a_{iij}=a_{ijj}=1$ and achieves $a_{iji}=1$. 
\begin{comment}
Indeed, we have from \cref{32c,32f}
\begin{align*}
g_{ii}' &= t(e_{ii})g_{ii}=g_{ii}
\\
a_{iij}' &= e_{ij}a_{iij}e_{ij}^{-1}\alpha(g'_{ij},e_{ii})^{-1}= e_{ij}e_{ij}^{-1}\alpha(g'_{ij},e_{ii})^{-1}=1
\\
a_{ijj}' &= e_{ij}a_{ijj}e_{jj}^{-1}\alpha(g'_{jj},e_{ij})^{-1}= e_{ij}\alpha(g'_{jj},e_{ij})^{-1}=1
\end{align*}
We get from \cref{32f} for $i\leq j$:
\begin{equation*}
a_{iji}' = a_{iji}e_{ji}^{-1}\alpha(g'_{ji},e_{ij})^{-1}=a_{iji}e_{ji}^{-1}=a_{iji}a_{iji}^{-1}=1 \end{equation*}
For $i > j$, we have by \cref{32c}:
\begin{equation}
\label{eq:norm:3}
g_{ji}' = t(e_{ji})g_{ji}=g_{ji}
\end{equation}
Then we get
\begin{equation}
\label{eq:norm:4}
a_{iji}\eqcref{31c} \alpha(g_{ji},a_{jij})\eqcref{eq:norm:3}\alpha(g_{ji}',a_{jij})\text{.}
\end{equation}
We conclude
\begin{equation*}
a_{iji}' \eqcref{32f} a_{iji}e_{ji}^{-1}\alpha(g'_{ji},e_{ij})^{-1}=
a_{iji}\alpha(g'_{ji},e_{ij})^{-1} =a_{iji}\alpha(g'_{ji},a_{jij})^{-1} \eqcref{eq:norm:4}1
\end{equation*}
\end{comment}        
Summarizing, every (generalized/fake-flat) differential $\Gamma$-cocycle is equivalent to a normalized (generalized/fake-flat) differential $\Gamma$-cocycle. 
This result slightly improves \cite[Lemma 4.1.4]{schreiber2}.
\end{remark}

We denote by $\hat \h^1(\mathcal{U},\Gamma)$ the set of equivalence classes of differential $\Gamma$-cocycles, and define the differential cohomology with values in $\Gamma$ as the direct limit
\begin{equation*}
\hat \h^1(M,\Gamma) := \lim_{\overrightarrow{\mathcal{U}}} \hat\h^1(\mathcal{U},\Gamma)
\end{equation*}
over refinements of covers. Analogously we proceed with equivalence classes of \emph{fake-flat} differential $\Gamma$-cocycles, leading to a subset
\begin{equation*}
\hat\h^1(M,\Gamma)^{\ff} \subset \hat \h^1(M,\Gamma)\text{.}
\end{equation*} 
Finally we consider \emph{generalized} equivalence classes of \emph{generalized} differential $\Gamma$-cocycles, making up a set $\hat\h^1(M,\Gamma)^{\gen}$, for which the inclusion of differential $\Gamma$-cocycles into generalized ones induces a well-defined map
\begin{equation*}
i: \hat\h^1(M,\Gamma) \to \hat\h^1(M,\Gamma)^{\gen}\text{.}
\end{equation*} 
This map is in general not injective (see \cref{ex:bueins}).

Giraud's non-abelian cohomology $\h^1(M,\Gamma)$ is defined in the same way as non-abelian differential cohomology, just without the differential forms and without all conditions involving differential forms. Discarding the differential forms defines maps from differential non-abelian cohomology to Giraud's non-abelian cohomology. These fit into  a commutative diagram
\begin{equation*}
\alxydim{}{\hat\h^1(M,\Gamma)^{\ff} \ar[dr]_{p^{\ff}} \ar@{^(->}[r] & \hat\h^1(M,\Gamma)\ar[d]^{p} \ar[r]^-{i} & \hat \h^1(M,\Gamma)^{\gen} \ar[dl]^{p^{\gen}} \\ & \h^1(M,\Gamma) &}
\end{equation*}
in which  $p^{gen}$ is surjective (\cref{co:pgen}), and $p^{\ff}$ is in general not surjective (\cref{ex:gdis}). It seems to be  an open question whether $p$ is surjective.

\begin{example}
\label{ex:gdis}
We consider an ordinary Lie group $G$ and $\Gamma=\idmorph G$, i.e. $\Gamma$ has only identity morphisms. This means that $H$ is the trivial group. A (generalized or not) differential $\idmorph G$-cocycle consists only of the 1-forms $A_i\in\Omega^1(U_i,\mathfrak{g})$ and the smooth maps $g_{ij}:U_i \cap U_j \to G$, which satisfy the ordinary cocycle condition $g_{ik} = g_{jk}g_{ij}$ and are gauge transformations: $A_j  = \mathrm{Ad}_{g_{ij}}(A_i) - g_{ij}^{*}\bar\theta$. The 3-curvature is zero, and the fake-curvature is $\mathrm{d}A_i+\frac{1}{2}[A_i \wedge A_i]$. A (generalized or not) equivalence consists of smooth maps $h_i:U_i \to G$ such that $A'_i = \mathrm{Ad}_{h_i}(A_i)  - h_i^{*}\bar\theta$ and $g_{ij}'h_i = h_jg_{ij}$. We conclude that
\begin{equation*}
 \hat \h^1(M,\idmorph G)^{\gen} =\hat\h^1(M,\idmorph G) = \hat\h^1(M,G) \quand
\hat\h^1(M,\idmorph{G})^{\ff}=\hat\h^1(M,G)^{\flat} \text{,}
\end{equation*}
where $\hat\h^1(M,G)$ and $\hat\h^1(M,G)^{\flat}$ are the ordinary differential cohomology groups that classify principal $G$-bundles with (flat) connections up to connection-preserving isomorphisms.
\end{example}

\begin{example}
\label{ex:bueins}
We consider $\Gamma=B\ueins$, i.e. $\Gamma$ has only one object with automorphism group $\ueins$. This means that $H=\ueins$ and $G$ is the trivial group. We identify $\mathfrak{h}=\R$. A differential $B\ueins$-cocycle consists  of the 2-forms $B_i\in \Omega^2(U_i)$, the 1-form $\varphi_{ij} \in \Omega^1(U_i \cap U_j)$ and the smooth map $
a_{ijk}:U_i \cap U_j \cap U_k \to \ueins$, which satisfy
\begin{equation}
\label{eq:bueins:gen}
B_j + \mathrm{d}\varphi_{ij} =B_i
\end{equation} 
and $\varphi_{ik}  =  \varphi_{jk}+\varphi_{ij} -a_{ijk}^{*}\theta$ as well as the cocycle condition $a_{ikl} a_{ijk}=a_{ijl}a_{jkl}$. Generalized means that \cref{eq:bueins:gen} is not present. The fake-curvature is zero, and the 3-curvature is $\mathrm{d}B_i$. An equivalence consists of 1-forms $\phi_{i}\in\Omega^1(U_i,\mathfrak{h})$ and smooth maps $e_{ij}:U_i \cap U_j \to H$ such that 
\begin{equation}
\label{eq:bueins:genequiv}
B_i' + \mathrm{d}\phi_i = B_i
\end{equation} 
and $\varphi_{ij}'+ \phi_i   = \phi_j + \varphi_{ij} 
- e_{ij}^{*}\theta$ as well as $a_{ijk}'e_{ij}e_{jk} = e_{ik}a_{ijk}$.
Generalized means that \cref{eq:bueins:genequiv} is not present. We conclude that
\begin{equation*}
\hat\h^1(M,B\ueins)^{\ff}=\hat\h^1(M,B\ueins) = \hat\h^3(M) 
\end{equation*}
is the ordinary differential cohomology (Deligne cohomology), whereas one can show that
\begin{equation*}
\hat\h^1(M,B\ueins)^{\gen}= \h^3(M,\Z)\text{.}
\end{equation*}
\end{example}

Finally, we prove a lemma about the maps \begin{equation*}
j:\hat\h^1(M,G)\to \hat\h^1(M,\Gamma)
\quand
j^{\flat}:\hat\h^1(M,G)^{\flat} \to \hat\h^1(M,\Gamma)^{\ff} \end{equation*}
induced by the inclusion $\idmorph{G} \subset \Gamma$ and the identifications of \cref{ex:gdis}. We will re-interpret the lemma in global language in \cref{ex:gbunredcon,co:gbunredcon}. 

\begin{lemma}
\label{lem:redGbun}
Let $f:M \to N$ be a smooth map with $\mathrm{rk}(T_xf)\leq 1$ for all $x\in M$. Then, the pullback along $f$ factors through $j$ and $j^{\flat}$; i.e., there exist  maps 
\begin{equation*}
j_G : \hat\h^1(N,\Gamma)^{\gen} \to \hat\h^1(M,G)
\quand
j_G^{\flat} : \hat\h^1(N,\Gamma)^{\ff}  \to \hat\h^1(M,G)
\end{equation*}
such that the following three diagrams commute:
\begin{equation*}
\alxydim{}{ & \hat\h^1(N,\Gamma)^{\gen} \ar[d]^{f^{*}} \ar[dl]_{j_G} \\ \hat\h^1(M,G) \ar[r]_-{i\circ j} & \hat\h^1(M,\Gamma)^{\gen}}
\quad
\alxydim{}{ & \hat\h^1(N,\Gamma) \ar[d]^{f^{*}} \ar[dl]_{j_G \circ i} \\ \hat\h^1(M,G) \ar[r]_-{j} & \hat\h^1(M,\Gamma)}
\quad
\alxydim{}{ & \hat\h^1(N,\Gamma)^{\ff} \ar[d]^{f^{*}} \ar[dl]_{j_G^{\flat}} \\ \hat\h^1(M,G)^{\flat} \ar[r]_-{j^{\flat}} & \hat\h^1(M,\Gamma)^{\ff}}
\end{equation*}
\end{lemma}

\begin{proof}
Consider a  generalized differential $\Gamma$-cocycle $\xi=(A,B,g,\varphi,a)$ with respect to an open cover $\mathcal{V}=\{V_i\}_{i\in J}$. By \cref{rem:normalized} we can assume that it is normalized. Let $X:=f(M)\subset N$.  By \cite[Theorem 2]{sard1} and  \cite[Proposition 1.3]{Church1963} the Lebesgue covering dimension of $X$ is bounded by $2$. Thus, there exists a refinement $\{W_{i}\}_{i\in I}$ of $\mathcal{V}$, such that no point $x\in X$ is contained in more that two open sets.  
\begin{comment}
We let $X$ be equipped with the subspace topology, so that $\{V_{i} \cap X\}_{i\in J}$ is an open cover of $X$. 
Because of the covering dimension, there exists a refinement $W_i \subset V_{r(i)} \cap X$. This, in turn, means that there exist open subsets $\tilde W_i \subset M$ with $W_i = \tilde W_i \cap X$. Now we construct the new cover of $M$ consisting of the open sets $\tilde W_i$ and $V_i \setminus X$ (these are open because $X\subset M$ is closed, as it is a compact subset of a Hausdorff space).  
\end{comment}
The restriction of the given $\Gamma$-cocycle to this refinement is again denoted by $(A,B,g,\varphi,a)$. 
Let $\{\psi_{i}\}_{i\in I}$ be a smooth partition of unity subordinated to the open cover $\{W_{i}\}_{i\in I}$. 
We define
\begin{equation*}
\phi_{i} :=  \sum_{j\in I} \psi_{j} (\alpha_{g_{ij}^{-1}})_{*}(\varphi_{ij})  \in \Omega^1(W_i,\mathfrak{h})\text{.}
\end{equation*}
The cocycle conditions \cref{eq:transtrans,eq:transconn2co} imply
\begin{equation}
\label{eq:thinlemma}
\phi_j- (\alpha_{g_{ij}})_{*}(\phi_i)= -\varphi_{ij} +\sum_{k\in I} \psi_{k} (\alpha_{g_{jk}^{-1}})_{*}\left ((\tilde\alpha_{a_{ijk}})_{*}(A_k)+a_{ijk}^{*}\theta \right )\text{.} 
\end{equation}
\begin{comment}
Indeed
\begin{align*}
\phi_j- (\alpha_{g_{ij}})_{*}(\phi_i) &=  \sum_{k\in I} \psi_{k} ((\alpha_{g_{jk}^{-1}})_{*}(\varphi_{jk})  - (\alpha_{g_{ij}})_{*}(\alpha_{g_{ik}^{-1}})_{*}(\varphi_{ik})  )
\\&\eqcref{eq:transtrans}\sum_{k\in I} \psi_{k} (\alpha_{g_{jk}^{-1}})_{*}\left (\varphi_{jk}  - \mathrm{Ad}^{-1}_{a_{ijk}}(\varphi_{ik})  \right )
\\&\eqcref{eq:transconn2co}\sum_{k\in I} \psi_{k} (\alpha_{g_{jk}^{-1}})_{*}\left (-(\alpha_{g_{jk}})_{*}(\varphi_{ij}) +(\tilde\alpha_{a_{ijk}})_{*}(A_k)+a_{ijk}^{*}\theta \right )
\\&= -\varphi_{ij} +\sum_{k\in I} \psi_{k} (\alpha_{g_{jk}^{-1}})_{*}\left ((\tilde\alpha_{a_{ijk}})_{*}(A_k)+a_{ijk}^{*}\theta \right )
\end{align*}
\end{comment}
Now we change the $\xi$ using equivalence data $\phi_i$, $h_i := 1$ and $e_{ij} := 1$, and obtain a new representative $(A',B',g',\varphi',a')$, with $g_{ij}'=g_{ij}$ due to \cref{32c}  and then
\begin{equation}
\label{eq:thinpullback:5}
\varphi_{ij}'= \phi_j - (\alpha_{g_{ij}'})_{*}(\phi_i)   +\varphi_{ij}=\sum_{k\in I} \psi_{k} (\alpha_{g_{jk}^{-1}})_{*}\left ((\tilde\alpha_{a_{ijk}})_{*}(A_k)+a_{ijk}^{*}\theta \right )
\end{equation}
due to \cref{32d,eq:thinlemma}. Then we pullback under the map $f:M \to N$. By construction, the pullback cover has no non-trivial 3-fold intersections, so that the map $a$ disappears due to its normalization. In particular, we obtain $\varphi_{ij}=0$ due to \cref{eq:thinpullback:5}. Further, the pullback of a 2-form under $f$ vanishes. Thus, we have
\begin{equation*}
f^{*}(A',B',g',\varphi',a')=(\tilde A, 0,\tilde g,0,1)\text{.}
\end{equation*}
This is a differential $\idmorph{G}$-cocycle on $M$, i.e. $j_G(\xi) := (\tilde A,\tilde g)$ is a differential $G$-cocycle, and the diagram on the left  commutes. If $\xi$ is non-generalized, the same construction goes through, noticing that above change of representative is also a non-generalized equivalence. Hence, the diagram in the middle commutes, too. Finally, if $\xi$ is fake-flat, then $(\tilde A,\tilde g)$ is flat; this defines $j_G^{\flat}$ and makes the third diagram commutative.
\end{proof}

\subsection{The bicategory of Lie groupoids}

We recall that there is a bicategory with objects Lie groupoids, 1-morphisms  anafunctors, and 2-morphisms smooth transformations, as explained in detail in \cite[Section 2.3]{Nikolaus} on the basis of \cite{lerman1,metzler}.
This bicategory  is equivalent to the bicategory of differentiable stacks  \cite{pronk}.

We recall some basic notation and terminology. A \emph{right action} of a Lie groupoid $\inf X$ on a smooth manifold $M$
consists of smooth maps $\alpha: M \to \ob{\inf X}$ and  
$\circ : M \ttimes \alpha t \mor{\inf X} \to M$
such that
\begin{equation*}
(x \circ g)\circ h= (x \circ (g \circ h))
\quad\text{, }\quad
x \circ \id_{\alpha(x)}=x
\quad\text{ and }\quad
\alpha(x \circ g) = s(g)
\end{equation*}
for all possible $g,h\in \mor{\inf X}$ and $x\in M$. The map $\alpha$ is  called  \emph{anchor}.  A \emph{left action} of $\inf X$ on $M$ is a right action of the opposite Lie groupoid. A smooth map $f:M \to M'$ between manifolds with $\inf X$-actions is called \emph{$\inf X$-equivariant} if 
\begin{equation*} 
\alpha' \circ f=\alpha
\quad\text{ and }\quad
f(x \circ g) =  f(x) \circ g\text{.}
\end{equation*}
A \emph{principal $\inf X$-bundle over $M$} is a smooth manifold $P$ with a surjective submersion $\pi: P \to M$ and a right $\inf X$-action that respects the projection $\pi$, such that
\begin{equation*}
P  \ttimes \alpha{t} \mor{\inf X} \to P \times_M P : (p,g) \mapsto (p,p \circ g)
\end{equation*}
is a diffeomorphism.
We refer to \cite[Section 2.2]{Nikolaus} for some examples. 

\begin{definition}
Let $\inf X$ and $\inf Y$ be Lie groupoids. 
\begin{enumerate}[(a)]
\item 
An \emph{anafunctor}\ $F: \inf X \to \inf Y$
is a smooth manifold $F$, called \quot{total space}, a left action of $\inf X$ on $F$, and  a right action  of $\inf Y$ on $F$ such that the actions commute and the left anchor $\alpha_l:F \to \ob{\inf X}$ together with the right action of $\inf Y$ is a principal $\inf Y$-bundle over $\ob{\inf X}$. 

\item

A \emph{transformation} between anafunctors $f: F \Rightarrow F'$ is a smooth map $f: F \to F'$ which is $\inf X$-equivariant, $\inf Y$-equivariant, and preserves the anchors. 
\end{enumerate}
\end{definition}

\begin{remark}
\label{re:strucana}
\begin{enumerate}[(a)]

\item 
\label{def:compana}
The composition of anafunctors is defined as follows.  Let $\inf X$, $\inf Y$ and $\mathcal{Z}$ be Lie groupoids, and
$F:\inf X \to \inf Y$
and
$G: \inf Y \to \mathcal{Z}$
be anafunctors.
 The composition
$G \circ F: \inf X \to \mathcal{Z}$
is the anafunctor with total space
\begin{equation*}
 (F \ttimes{\alpha_r}{\alpha_l} G) / \inf Y\text{,}
\end{equation*}
where the quotient identifies $(f,\eta \circ g) \sim (f \circ \eta,g)$ for all $\eta\in \mor{\inf Y}$ with $\alpha_r(f)=t(\eta)$ and $\alpha_l(g)=s(\eta)$.
The anchors are  $(f,g) \mapsto \alpha_l(f)$ and $(f,g) \mapsto \alpha_r(g)$, and the actions are $\gamma \circ(f,g) := (\gamma \circ f,g)$ and $(f,g) \circ \gamma := (f,g \circ \gamma)$.

\item
\label{defweak}
An anafunctor $F:\inf X \to \inf Y$ is called  \emph{weak equivalence}, if there exists an anafunctor $G: \inf Y \to \inf X$ together with transformations $G \circ F \cong \id_{\inf X}$ and $F \circ G \cong \id_\inf Y$. 

\item
\label{th:weakinverse}
It is straightforward to check that an anafunctor $F:\inf X \to \inf Y$ is a weak equivalence if and only if $\alpha_r: F \to \ob{\inf Y}$ is a principal $\mor{\inf X}$-bundle. In this case, one can obtain an inverse anafunctor $F^{-1}:\inf Y \to \inf X$ with the same smooth manifold $F$, but the anchors exchanged and the actions exchanged and inverted.
\end{enumerate}
\end{remark}
 
For  Lie groupoids $\inf X$ and $\inf Y$, anafunctors $F: \inf X \to \inf Y$ and transformations form a groupoid $\af(\inf X,\inf Y)$. The composition defined in \cref{def:compana} extends to a functor
\begin{equation*}
\af(\inf X,\inf Y) \times \af(\inf Y,\inf Z) \to \af(\inf X,\inf Z)\text{.}
\end{equation*}
Equipped with this functor as the composition, Lie groupoids, anafunctors, and transformations form a bicategory \cite{pronk}. Weak equivalences are the 1-isomorphisms in this bicategory.

\begin{remark}
\begin{enumerate}[(a)]

\item 
\label{rem:smoothfunctor:a}
For Lie groupoids $\mathcal{X}$ and $\mathcal{Y}$, we denote by  $\fun^{\infty}(\inf X,\inf Y)$ the category of smooth functors and smooth natural transformations. Anafunctors generalize smooth  functors in terms of a full and faithful functor 
\begin{equation*}
\fun^{\infty}(\inf X,\inf Y) \to \af (\inf X,\inf Y)\text{.}
\end{equation*}
Indeed, if $\phi:\inf X \to \inf Y$ is a smooth functor, we define an anafunctor with total space $F := \ob{\inf X} \ttimes \phi t \mor{\inf Y}$,  anchors  $\alpha_l(x,g) := x$ and $\alpha_r(x,g) := s(g)$, and  actions  $f \circ (x,g) := (t(f),\phi(f) \circ g)$ and $(x,g) \circ f := (x, g \circ f)$. 
\begin{comment}
This is compatible with composition. Indeed, if $\phi':\mathcal{Y} \to \mathcal{Z}$ is a second smooth functor,  $F'$ is the anafunctor associated to $\phi'$, and $\tilde F$ is the anafunctor associated to $\phi'\circ\phi$, then we have a transformation $F'\circ F \Rightarrow \tilde F$ induced by
\begin{equation*}
F \ttimes{\alpha_r}{\alpha_l} F' \to \tilde F:((x,g),(x',g')) \mapsto (x,\phi'(g) \circ g')
\end{equation*}
where $x=\alpha_l(x,g)=\alpha_r(x',g')=s(g')$ and $\phi'(x')=t(g')$ and $\phi(x)=t(g)$ and $x'=s(g)$.
Its inverse is
\begin{equation*}
\tilde F\to F \ttimes{\alpha_r}{\alpha_l} F': (x,g)\mapsto ((x,\id_{\phi(x)}),(\phi(x),g))\text{.}
\end{equation*}
\end{comment}
Likewise, a smooth natural transformation $\eta:\phi\Rightarrow \phi'$ defines a transformation $f_{\eta}:F \Rightarrow F'$ by $f_{\eta}(x,g) := (x,\eta(x) \circ g)$.

\item
\label{th:weak}
A smooth functor $\phi: \inf X \to \inf Y$ induces a weak equivalence under \cref{rem:smoothfunctor:a*} if and only if it is \emph{smoothly essentially surjective}, i.e., the map
\begin{equation*}
s \circ \mathrm{pr}_2: \ob{\inf X} \ttimes {\phi}{t} \mor{\inf Y} \to \ob{\inf Y}
\end{equation*}
is a surjective submersion, and it is \emph{smoothly fully faithful}, i.e., the diagram
\begin{equation*}
\alxydim{@C=1.5cm@R=1.3cm}{\mor{\inf X} \ar[r]^{\phi} \ar[d]_{s \times t} & \mor{\inf Y} \ar[d]^{s \times t} \\ \ob{\inf X} \times \ob{\inf X} \ar[r]_-{\phi \times \phi} & \ob{\inf Y} \times \ob{\inf Y}}
\end{equation*}
is a pullback diagram; see \cite[Lemma 3.34]{lerman1}, \cite[Proposition 60]{metzler}.

\item
\label{rem:smoothfunctor:b}
The following result will be used in \cref{sec:local} and might be useful in other situations. Suppose $F: \mathcal{X} \to \mathcal{Y}$ is an anafunctor, $\phi:\mathcal{X} \to \mathcal{Y}$ is a smooth functor, and $F_{\phi}$ is the anafunctor associated to $\phi$ via \cref{rem:smoothfunctor:a*}. Then, a transformation $f: F_{\phi} \Rightarrow F$ is the same as a smooth map
$\tilde f: \ob{\mathcal{X}} \to F$
satisfying
\begin{equation*}
\alpha_l(\tilde f(x)) = x
\quomma
\alpha_r ( \tilde f(x)) =  \phi(x)
\quand
\alpha\circ \tilde f(x)\circ \beta=\tilde f(t(\alpha))\circ \phi(\alpha)\circ \beta
\end{equation*}
for all $x\in \ob{\mathcal{X}}$, and appropriate $\alpha\in \mor{\mathcal{X}}$ and $\beta\in \mor{\mathcal{Y}}$. The correspondence between $f$ and $\tilde f$ is established by $\tilde f(x) = f(x,\phi(\id_x))$.

\end{enumerate}
\end{remark}

\subsection{Lie 2-group actions and equivariance}

Next we discuss a version of the bicategory of Lie groupoids that is equivariant under the action of a  Lie 2-group $\Gamma$, on the basis of \cite{Nikolaus}. It will be used in order to model the 2-group action on the total space groupoid of a principal 2-bundle. \begin{comment}
The following definition was given in \cite[Section 6.1 \& Appendix A]{Nikolaus}. \end{comment}

\begin{definition}
\label{def:rightaction}
\label{def:actionanafunctor}
\begin{enumerate}[(a)]

\item 
 A \emph{smooth right action} of $\Gamma$ on a Lie groupoid $\inf X$ is a smooth functor
$R: \inf X \times \Gamma \to \inf X$
such that $R(p,1)=p$ and $R(\rho,\id_1)=\rho$ for all $p\in \ob {\inf X}$ and $\rho\in \mor{\inf X}$,  and such that the diagram
\begin{equation*}
\alxydim{@=1.3cm}{\inf X \times \Gamma \times \Gamma \ar[r]^-{\id \times m} \ar[d]_{R \times \id} & \inf X \times \Gamma  \ar[d]^{R} \\ \inf X \times \Gamma \ar[r]_{m} & \inf X}
\end{equation*}
of smooth functors is commutative (strictly, on the nose).

\item
\label{def:actionanafunctor:b}
If $\inf X$ and $\inf Y$ are Lie groupoids with smooth right $\Gamma$-actions, then a $\Gamma$-\emph{equivariant anafunctor}
 $F:\inf X \to \inf Y$ is an anafunctor $F$ together with a  smooth action
$\rho: F \times \mor{\Gamma} \to F$
of the Lie group $\mor{\Gamma}$ on its total space $F$ that preserves the anchors in the sense that the diagrams
\begin{equation*}
\alxydim{@R=1.3cm}{F \times \mor{\Gamma} \ar[d]_{\alpha_l \times t} \ar[r]^-{\rho} & F \ar[d]^{\alpha_l} \\ \ob{\inf X} \times \ob{\Gamma} \ar[r]_-{R} & \ob{\inf X}}
\quand
\alxydim{@R=1.3cm}{F \times \mor{\Gamma} \ar[r]^-{\rho} \ar[d]_{\alpha_r \times s} & F \ar[d]^{\alpha_r} \\ \ob{\inf Y} \times \ob{\Gamma} \ar[r]_-{R} & \ob{\inf Y}}
\end{equation*}
are commutative, and  is compatible with the $\inf X$- and $\inf Y$-actions in the sense that the identity
\begin{equation*}
\rho(\chi \circ f \circ \eta,\gamma_l \circ \gamma \circ \gamma_r) =R(\chi,\gamma_l) \circ  \rho(f,\gamma) \circ R(\eta,\gamma_r)
\end{equation*}
holds for all appropriate  $\chi\in \mor{\inf X}$, $\eta\in \mor{\inf Y}$, $f\in F$, and $\gamma_l,\gamma,\gamma_r\in \mor{\Gamma}$.

\item
If $F_1,F_2:\inf X \to \inf Y$ are $\Gamma$-equivariant anafunctors, a transformation $f: F_1 \Rightarrow F_2$ is called $\Gamma$-\emph{equivariant} if the map $f:F_1 \to F_2$ between total spaces is $\mor{\Gamma}$-equivariant. 
\end{enumerate}
\end{definition}

\begin{remark}
\begin{enumerate}[(a)]
\item 
\label{rem:compequiv}
The structure of \cref{def:rightaction} forms a bicategory. For instance, the composition $G \circ F$ of $\Gamma$-equivariant anafunctors  $F:\inf X \to \inf Y$ and $G:\inf Y \to \mathcal{Z}$ with actions  $\rho:F \times \mor{\Gamma} \to F$ and $\tau: G \times \mor{\Gamma} \to G$, respectively, inherits a $\mor{\Gamma}$-action induced by
\begin{equation}
\label{eq:companafunctorsaction}
(F  \ttimes{\alpha_r}{\alpha_l}  G) \times  \mor{\Gamma} \to (F  \ttimes{\alpha_r}{\alpha_l}  G): ((f,g),\gamma) \mapsto (\rho(f,\gamma),\tau(g,\id_{s(\gamma)}))\text{.}
\end{equation}
\begin{comment}
The identity anafunctor is equipped with the trivial $\mor{\Gamma}$-action on its total space. 
\end{comment}
On the inverse $F^{-1}$ of a weak equivalence $F:\inf X \to \inf Y$, $\mor\Gamma$ acts through inverses (with respect to the groupoid composition of $\Gamma$).

\item
\label{rem:equivsmoothfunctor}
Consider a smooth functor $\phi: \inf X \to \inf Y$ between Lie groupoids with $\Gamma$-actions that is  equivariant in the sense that the diagram
\begin{equation*}
\alxydim{}{\inf X \times \Gamma \ar[r]^-{\phi \times \id} \ar[d]_{R} & \inf Y \times \Gamma \ar[d]^{R} \\ \inf X \ar[r]_{\phi} & \inf Y}
\end{equation*}
is strictly commutative. Then, the associated anafunctor  $F= \ob{\inf X} \ttimes \phi t \mor{\inf Y}$ is $\Gamma$-equivariant under the $\mor{\Gamma}$-action on $F$  given by
\begin{equation*}
\rho: F \times \mor{\Gamma} \to F: ((x,\eta),\gamma) \mapsto  (R(x,t(\gamma)),R(\eta,\gamma))\text{.}
\end{equation*}
\begin{comment}
Indeed, we check well-definedness: we have $\phi(x)=t(\eta)$ and obtain
\begin{equation*}
\phi(R_1(x,t(\gamma)))=R_2(\phi(x),t(\gamma))=R_2(t(\eta),t(\gamma))=t(R_2(\eta,\gamma))\text{.}
\end{equation*}
It is an action:
\begin{eqnarray*}
\rho(\rho((x,\eta),\gamma_1),\gamma_2) &=& \rho((R_1(x,t(\gamma_1)),R_2(\eta,\gamma_1)),\gamma_2)
\\&=& (R_1(R_1(x,t(\gamma_1)),t(\gamma_2)),R_2(R_2(\eta,\gamma_1),\gamma_2))
\\&=& (R_1(x,t(\gamma_1 \cdot \gamma_2)),R_2(\eta,\gamma_1 \cdot \gamma_2))
\\&=& \rho((x,\eta),\gamma_1 \cdot \gamma_2)
\end{eqnarray*}
Now the conditions:
\begin{itemize}

\item 
Anchor preservation:
\begin{equation*}
\alpha_l(\rho((x,\eta),\gamma)) =R_1(x,t(\gamma))= R_1(\alpha_l(x,\eta),t(\gamma))
\end{equation*}
and
\begin{equation*}
\alpha_r(\rho((x,\eta),\gamma)) =s(R_2(\eta,\gamma))=R_2(s(\eta),s(\gamma))= R_2(\alpha_r(x,\eta),s(\gamma))
\end{equation*}

\item
Action compatibility:
\begin{eqnarray*}
\rho(\chi \circ (x,\eta) \circ \eta',\gamma_l \circ \gamma \circ \gamma_r) &=& \rho((t(\chi),\phi(\chi) \circ \eta\circ \eta'),\gamma_l \circ \gamma \circ \gamma_r)
\\&=& (R_1(t(\chi),t(\gamma_l)),R_2(\phi(\chi) \circ \eta\circ \eta',\gamma_l \circ \gamma \circ \gamma_r))
\\&=& (t(R_1(\chi,\gamma_l)),R_2(\phi(\chi) \circ \eta\circ \eta',\gamma_l \circ \gamma \circ \gamma_r))
\\&=& (t(R_1(\chi,\gamma_l) ),\phi(R_1(\chi,\gamma_l) ) \circ R_2(\eta,\gamma)\circ R_2(\eta',\gamma_r)) 
\\&=& R_1(\chi,\gamma_l) \circ  \rho((x,\eta),\gamma) \circ R_2(\eta',\gamma_r)
\end{eqnarray*}

\end{itemize}
\end{comment}
Likewise, if $\beta: \phi_1 \Rightarrow \phi_2$ is a smooth natural transformation between $\Gamma$-equivariant smooth functors, and satisfies the equivariance condition
\begin{equation*}
\beta(R(x,g))=R(\beta(x),\id_{g})
\end{equation*}
for all $x\in \ob{\inf X}$ and $g\in G$, 
then the associated transformation is $\Gamma$-equivariant.
\begin{comment}
Indeed,
\begin{eqnarray*}
\beta(\rho_1((x,\eta),\gamma)) &=&\beta((R_1(x,t(\gamma)),R_2(\eta,\gamma)))
\\&=& (R_1(x,t(\gamma)),\beta(R_1(x,t(\gamma))) \circ R_2(\eta,\gamma))
\\&=& (R_1(x,t(\gamma)),R_2(\beta(x),\id_{t(\gamma)}) \circ R_2(\eta , \gamma))
\\&=& (R_1(x,t(\gamma)),R_2(\beta(x) \circ \eta,\id_{t(\gamma)} \circ \gamma))
\\&=& (R_1(x,t(\gamma)),R_2(\beta(x) \circ \eta,\gamma))
\\&=& \rho_2((x,\beta(x) \circ \eta),\gamma)
\\&=& \rho_2(\beta(x,\eta),\gamma)
\end{eqnarray*}
\end{comment}
Finally, if  $F:\mathcal{X} \to \mathcal{Y}$ is a $\Gamma$-equivariant anafunctor and $\phi:\mathcal{X} \to \mathcal{Y}$ is a $\Gamma$-equivariant smooth functor, then the transformation $f: F_{\phi} \Rightarrow F$ induced from a smooth map $\tilde f: \ob{\mathcal{X}} \to F$ as in \cref{rem:smoothfunctor:b} is $\Gamma$-equivariant if and only if that map satisfies $\tilde f(R(x,g))=\tilde f(x)\cdot \id_g$ for all $x\in \ob{\mathcal{X}}$ and $g\in G$.
\begin{comment}
We check
\begin{align*}
\eta(x,\nu)\cdot\gamma 
&=(\tilde\eta(x)\circ\nu)\cdot\gamma
\\&=(\tilde\eta(x)\circ\nu)\cdot(\id_{t(\gamma)} \circ \gamma)
\\&=(\tilde\eta(x)\cdot \id_{t(\gamma)})\circ R(\nu, \gamma)
\\&=\tilde\eta(R(x,t(\gamma))\circ R(\nu,\gamma)
\\&=\eta(R(x,t(\gamma)),\id_{R(\phi(x),t(\gamma)})\circ R(\nu,\gamma)=\eta(R(x,t(\gamma)),R(\nu,\gamma))
\\&=\eta((x,\nu)\cdot\gamma)
\end{align*}
\end{comment}

\item
In \cite[Section 6.1]{Nikolaus} a more abstract definition of $\Gamma$-equivariant anafunctors was given, and it is shown in \cite[Appendix A]{Nikolaus} it is equivalent to the one given here.
\end{enumerate}
\end{remark}

\setsecnumdepth{2}

\section{Principal 2-bundles}

\label{sec:2bundles}

\subsection{The bicategory of principal 2-bundles}

In this section we review the notion of a principal 2-bundle for a strict Lie 2-group $\Gamma$ on the basis of \cite{wockel1,Nikolaus}.

Let $M$ be a smooth manifold, $\idmorph M$ the Lie groupoid with objects $M$ and only identity morphisms,  and let
 $\inf P$ be a Lie groupoid. We say that a smooth functor $\pi\maps \inf P \to \idmorph M$ is a \emph{surjective submersion functor}, if $\pi: \ob{\inf P} \to M$ is a surjective submersion. 
Let $\pi\maps \inf P \to \idmorph M$ be a surjective submersion functor, and let $\inf Q$ be a Lie groupoid equipped with some smooth functor $\chi:\inf Q \to \idmorph M$. Then, the fibre product $\inf P \times_M \inf Q$ is defined to be the full sub-Lie groupoid of $\inf P \times \inf Q$ over the submanifold $\ob{\inf P} \times_M \ob{\inf Q} \subset \ob{\inf P} \times \ob{\inf Q}$.

\begin{definition}
\label{def:zwoabun}
A \emph{principal $\Gamma$-2-bundle over $M$} is a Lie groupoid $\inf P$, a surjective submersion  functor $\pi: \inf P \to \idmorph M$,  and a smooth right action $R$ of $\Gamma$ on $\inf P$ that preserves  $\pi$,
such that the smooth functor
\begin{equation*}
\tau :=(\mathrm{pr}_1, R) : \inf P \times \Gamma \to \inf P \times_M \inf P \end{equation*}
is a weak equivalence.
\end{definition}

We collect the following four facts about a principal 2-bundle $\inf P$ over $M$.
The first fact comes from the condition that the bundle projection $\pi: \inf P \to \idmorph M$ is a surjective submersion functor:

\begin{lemma}
\label{lem:exsections}
Every point $x\in M$ has an open neighborhood $x\in U \subset M$ supporting a section, i.e. a smooth map $s: U \to \ob{\inf P}$  such that $\pi \circ s=\id_U$. \end{lemma}

 As a functor $\pi:\inf P \to \idmorph{M}$ necessarily sends every morphism $\rho$ of $\inf P$ to an identity morphism, we obtain the second fact:

\begin{lemma}
\label{lem:morphvert}
Every morphism $\rho$ of $\inf P$ is \quot{vertical} in the sense that $\pi(s(\rho))=\pi(t(\rho))$. \end{lemma}

The next two lemmata give a precise formulation of the way in which  the $\Gamma$-action on $\inf P$ is \quot{fibrewise free and transitive}. On the level of morphisms, one uses that $\tau$, as a weak equivalence, is \quot{smoothly fully faithfully} (\cref{th:weak}), which  implies the following:

\begin{lemma}
\label{lem:actmorph}
Suppose $\rho_1,\rho_2 \in \mor{\inf P}$  and $g_1,g_2\in G$ such that
$s(\rho_2)=R(s(\rho_1),g_{1})$ and $t(\rho_2)=R(t(\rho_1),g_2)$. Then, there exists a unique $h\in H$ such that $R(\rho_1,(h,g_1))=\rho_2$ and $t(h)g_1=g_2$.
Moreover, if $\rho_1,\rho_2$ and $g_1,g_2$ depend smoothly on a parameter $x\in X$, where $X$ is a smooth manifold, then $h$ also depends smoothly on this parameter.
\begin{comment}
More explicitly, let $X$ be a smooth manifold equipped with smooth maps $\rho_1,\rho_2: X \to  \mor{\inf P}$ and $g_1,g_2:X \to G$ such that $s(\rho_2(x))=R(s(\rho_1(x)),g_{1}(x))$ and $t(\rho_2(x))=R(t(\rho_1(x)),g_2(x))$. Then there exists a unique smooth map $h:X \to H$ such that $R(\rho_1(x),(h(x),g_1(x)))=\rho_2(x)$ and $g_2(x)=t(h(x))g_1(x)$. \end{comment}
\end{lemma}

\begin{comment}
\begin{proof}
This rephrases the fact that the functor $\tau$ is smoothly fully faithfully: which means that the diagram
\begin{equation*}
\tiny
\alxydim{@C=0cm@R=1.3cm}{(\rho_1,\gamma) &\in &\mor{\inf P} \times \mor{\Gamma} \ar[rr]^{\tau} \ar[d] & \hspace{0.6cm} & \mor{\inf P} \times_M \mor{\inf P} \ar[d] & \ni & (\rho_1,\rho_2) \ar@{|->}[d] \\ ((s(\rho_1),s(\gamma)),(t(\rho_1),t(\gamma))) & \in & \ob{\inf P} \times \ob{\Gamma} \times \ob{\inf P} \times \ob{\Gamma} \ar[rr]_-{\tau \times \tau} && (\ob{\inf P} \times_M \ob{\inf P}) \times (\ob{\inf P} \times_M \ob{\inf P}) & \ni & (s(\rho_1),s(\rho_2),(t(\rho_1),t(\rho_2)))}
\end{equation*}
is a pullback diagram.
If we specify $(\rho_1,\rho_2)$ in the upper right corner, and $(s(\rho_1),g_1,t(\rho_1),g_2)$ in the lower left corner, we have coincidence in the lower right corner, under the assumptions that 
$s(\rho_2)=R(s(\rho_1),g_{1})$ and $t(\rho_2)=R(t(\rho_1),g_2)$. Hence, there exists $\gamma=(h,g_1)$ with $\rho_2=R(\rho_1,(h,g_1))$ and $g_2=t(h)g_1$.
\end{proof}
\end{comment}

On the level of objects the situation is more complicated. We  introduce the following terminology, which will be an important tool throughout this article:

\begin{definition}
\label{def:Rsection}
A  \underline{{\splitting}} of $\inf P$ over a subset $U \subset \ob{\inf P} \times_M \ob{\inf P}$ is a smooth map $\sigma:U \to \mor{\inf P}$, such that there exists a smooth map $g: U \to G$ with
\begin{equation*}
t(\sigma(x,y))=x
\quand
s(\sigma(x,y))=R(y,g(x,y))
\end{equation*} 
for all $(x,y)\in U$.
The map $g$ is called a \emph{transition function} for $\sigma$.
\end{definition}

Thus, a \splitting\ is a combination of morphisms of $\inf P$ and the $G$-action, that allows to pass from one object $x$ of $\inf P$ to another object $y$ in the same fibre.
\begin{comment}
See \cref{fig:Psection}.
\begin{figure}[t] 
\begin{equation*}
\alxydim{@R=1.8cm@C=1cm}{ \mor{\inf P} \ar[d]_{t}  & \mor{\inf P} \times G \ar[l]_-{\pr_1} \ar[r]^-{s \times \id} & \ob{\inf P} \times G \ar[d]^{R} \\ \ob{\inf P} & U  \ar[u]_{\sigma \times g^{-1}} \ar[l]^-{\pr_1} \ar[ul]_{\sigma} \ar[r]_-{\pr_2} & \ob{\inf P} }
\end{equation*}
\caption{A picture of a \splitting\ $\sigma$ with transition function $g$. All subdiagrams are commutative. }
\label{fig:Psection}
\end{figure}
\end{comment}
The following lemma describes the way in which this passage is \quot{fibrewise free and transitive}.

\begin{lemma}
\label{lem:existenceRsections}
Every open and contractible subset $U\subset \ob{\inf P} \times_M \ob{\inf P}$ supports \splitting s. If $\sigma$ and $\sigma'$ are \splitting s over an arbitrary open subset  $U\subset \ob{\inf P} \times_M \ob{\inf P}$ with transition functions $g$ and $g'$, respectively, then there exists a unique smooth map $h:U \to H$ such that $R(\sigma',(h,g'^{-1}g))=\sigma$ and $gt(h)=g'$.\end{lemma}

\begin{comment}
We could also write
\begin{equation*}
\sigma'=R(\sigma,(\alpha(g^{-1}g',h^{-1}),g^{-1}g'))=R(\sigma,(\alpha(t(h),h^{-1}),t(h)))=R(\sigma,(h^{-1},t(h)))
\end{equation*}
\end{comment}

\begin{proof}
We define $P:= \mor{\inf P} \times G$ equipped with the map
\begin{equation*}
\chi: P \to \ob{\inf P} \times_M \ob{\inf P}: (\rho,g) \mapsto  (t(\rho),R(s(\rho),g^{-1}))\text{.} \end{equation*}
We define a $\Gamma$-action on $P$ with anchor $\alpha(\rho,g):= g$ and 
\begin{equation}
\label{eq:lem:existence:1}
(\rho,g) \circ (h,g') := (R(\rho,(\alpha(g^{-1},h),g^{-1}g')),g')\text{.}
\end{equation} 
By \cite[Lemma 7.1.1]{Nikolaus}, $P$ is a principal $\Gamma$-bundle over $\ob{\inf P} \times_M \ob{\inf P}$. It is easy to check that writing a section  $\tilde\sigma: U \to P$ of $P$  as a pair $(\sigma,g)$ with $\sigma:U \to \mor{\inf P}$ and $g:U \to G$ establishes a bijection between sections of $P$ and pairs of \splitting s and transition functions.
\begin{comment}
Indeed, suppose $\tilde\sigma:U \to P$ is a local section. Write $\tilde\sigma=(\sigma,g)$. Then, we have
\begin{equation*}
(x,y)=\chi(\tilde\sigma(x,y))= (t(\sigma(x,y)),R(s(\sigma(x,y)),g(x,y)^{-1}))\text{.}
\end{equation*}
This shows that $\sigma$ is a \splitting\ with transition function $g$.
\end{comment}
Now, the fact that a $\Gamma$-bundle has local sections over contractible open sets implies  the existence of \splitting s.
\begin{comment}
A $\Gamma$-bundle is an $H$-bundle, and $H$-bundles admit sections over contractible open sets: choose a contraction and a connection. 
\end{comment} 
Existence and uniqueness of the map $h$ follow from \cref{lem:actmorph}. 
\end{proof}

\begin{remark}
\label{re:splitting}
\Splitting s split the tangent spaces to $\mor{\inf P}$ into a \quot{horizontal} part (the image of the \splitting) and a \quot{vertical} part (the orbit of the $\Gamma$-action).  Indeed, let $\rho\in \mor{\inf P}$, and let $\sigma:U \to \mor{\inf P}$ be a \splitting\ defined on an open neighborhood $U \subset \ob{\inf P} \times_M \ob{\inf P}$ of $u:=(t(\rho),s(\rho))$, together with a transition function $g:U \to G$. Then, there exists a unique $h\in H$ with $\rho=R(\sigma(u), (h,g(u)))$, and a splitting
\begin{equation}
\label{eq:re:splitting:1}
T_{\rho}\mor{\inf P} = TR_{h,g(u)} T\sigma(T_uU)   \oplus TR_{\sigma(u)}(T_hH \oplus T_{g(u)}G)\text{.}
\end{equation}
Here $R_{\rho}: \mor{\Gamma} \to \mor{\inf P}$ is the action on the fixed $\rho$, and $R_{h,g}: \mor{\inf P} \to \mor{\inf P}$ is the action by fixed  $(h,g)\in \mor{\Gamma}$.
In order to see \cref{eq:re:splitting:1}, let a tangent vector $\tilde v \in T_{\rho}\mor{\inf P}$ be represented by a smooth curve $\tilde\rho$ in $\mor{\inf P}$ with $\tilde\rho(0)=\rho$. Let $\gamma_1,\gamma_2$ be the curves in $\ob{\inf P}$ defined by $\gamma_1 := t \circ \tilde\rho$ and $\gamma_2 :=s \circ  \tilde\rho$, and let $v_1,v_2$ be the corresponding tangent vectors. The pair $(\gamma_1,\gamma_2)$ is a curve in $\ob{\inf P} \times_M \ob{\inf P}$, and we can assume that it is in $U$. We apply  \cref{lem:actmorph}
to the families $\rho_1(t):= \sigma(\gamma_1(t),\gamma_2(t))$ and $\rho_2(t):= \tilde\rho(t)$, as well as $g_1(t):= g(\gamma_1(t),\gamma_2(t))$ and $g_2(t):= 1$.
 Hence, there exists a unique smooth curve $\eta$ in $H$ through $h:=\eta(0)$, with $t\circ \eta\cdot g =1$ and 
\begin{equation}
\label{eq:re:splitting:2}
\tilde\rho(t) = R(\sigma(\gamma_1(t),\gamma_2(t)), (\eta(t),g(\gamma_1(t),\gamma_2(t))))\text{.}
\end{equation}
In particular, $\rho=R(\sigma(u), (h,g(u)))$ at $t=0$.
Taking derivatives in \cref{eq:re:splitting:2} yields
\begin{equation*}
\tilde v =T_{u} R_{h,g(u)}T\sigma(v_1,v_2) + T_{h,g(u)} R_{\sigma(u)}(w,T_{u}g(v_1,v_2)) \text{,}
\end{equation*} 
where $w\in T_hH$ is the tangent vector to the curve $\eta$. This is the claimed splitting.
\end{remark}

Next we discuss the bicategory of principal $\Gamma$-2-bundles, starting with the 1-morphisms.

\begin{definition}
A \emph{1-morphism} between principal $\Gamma$-2-bundles is a  $\Gamma$-equivariant anafunctor 
$F: \inf P_1 \to \inf P_2$
such that $\pi_2 \circ F =\pi_1$. 
\end{definition}

One can show that every 1-morphism between principal $\Gamma$-2-bundles is automatically invertible \cite[Corollary 6.2.4]{Nikolaus}; in particular, every 1-morphism is a weak equivalence. Similar to \cref{def:Rsection}, we say that a \emph{\splitting} of $F$ over a subset $U \subset \ob{\inf P_1} \times_M \ob{\inf P_2}$ is a smooth map $\sigma:U \to F$ such that there exists a smooth map $g: U \to G$ with
\begin{equation*}
\alpha_l(\sigma(x_1,x_2))=x_1
\quand
\alpha_r(\sigma(x_1,x_2))=R(x_2,g(x_1,x_2))
\end{equation*} 
for all $(x_1,x_2)\in U$.
The map $g$ is called a \emph{transition function} for $\sigma$.

\begin{lemma}
\label{lem:Fsections}
Every open and contractible subset $U\subset \ob{\inf P_1} \times_M \ob{\inf P_2}$ supports \splitting s. Moreover, if $\sigma_1$ and $\sigma_2$ are \splitting s over an arbitrary open subset  $U\subset \ob{\inf P_1} \times_M \ob{\inf P_2}$ with transition functions $g_1$ and $g_2$, respectively, then there exists a unique smooth map $h:U \to H$ such that $g_1 = (t \circ h)\cdot g_2$ and $\sigma_2=\rho(\sigma_1,(\alpha(g_1^{-1},h),g_1^{-1}g_{2}))$, where $\rho$ is the $\mor{\Gamma}$-action on $F$. 
\end{lemma}

\begin{proof}
We define $Q := F \times G$ equipped with the map 
\begin{equation*}
\chi: Q \to \ob{\inf P_1} \times_M \ob{\inf P_2}:  (f,g) \mapsto  (\alpha_l(f), R(\alpha_r(f),g^{-1}))\text{.}
\end{equation*}
We define a $\Gamma$-action on $Q$ with anchor $\alpha(f,g):= g$ and
\begin{equation*}
(f,g) \circ (h,g') := (\rho(f,(\alpha(g^{-1},h),g^{-1}g')),g')\text{,}
\end{equation*}
By \cite[Lemma 7.1.4]{Nikolaus} $Q$ is a principal $\Gamma$-bundle over $\ob{\inf P_1} \times_M \ob{\inf P_2}$. Writing a section $\tilde\sigma: U \to Q$  of $Q$ as $\tilde\sigma=(\sigma,g)$ for smooth maps $\sigma:U \to F$ and $g:U \to G$ establishes a bijection between sections of $Q$ and pairs of \splitting s and transition functions. Now the claims follows from the general properties of principal bundles.
\end{proof}

\begin{remark}
\label{re:splittingF}
\Splitting s split the tangent spaces to $F$ into a \quot{horizontal} part (the image of the \splitting) and a \quot{vertical part} (the orbit of the $\mor\Gamma$-action). Indeed, let $f\in F$ and $U \subset \ob{\inf P_1} \times_M \ob{\inf P_2}$ be an open neighborhood of $u := (\alpha_l(f),\alpha_r(f))$, with a \splitting\ $\sigma:U \to F$  with transition function $g$. Then, there exists a unique $h\in H$ such that $f=\rho(\sigma(u),(h,g(u)^{-1}))$ and
\begin{equation}
\label{eq:re:splittingF:1}
T_fF =T\rho_{h,g(u)^{-1}}(T\sigma(T_uU)) \oplus T\rho_{\sigma(u)} (T_{h,g(u)^{-1}}\mor{\Gamma})\text{.} \end{equation}
Here $\rho_{h,g}: F \to F$ is the action on $F$ with fixed $(h,g)\in \mor\Gamma$, and $\rho_{f}: \mor{\Gamma} \to F$ the action on a fixed $f\in F$. 
In order to see \cref{eq:re:splittingF:1}, let a tangent vector $v \in T_{f}F$ be represented by a smooth curve $\phi$ in $F$ with $\phi(0)=f$. Let $\gamma_1,\gamma_2$ be the curves  defined by $\gamma_1 := \alpha_l \circ \phi$ and $\gamma_2 :=  \alpha_r \circ \phi$, and let $v_1,v_2$ be the corresponding tangent vectors. The pair $(\gamma_1,\gamma_2)$ is a curve in $\ob{\inf P_1} \times_M \ob{\inf P_2}$, and we can assume that it is in $U$. Since $\alpha_l(\phi(t))=\alpha_l(\sigma(\gamma_1(t),\gamma_2(t)))$ and $\alpha_l:F \to \ob{\inf P_1}$ is a principal $\inf P_2$-bundle, there exists a unique curve $\tilde\rho$ in $\mor{\inf P_2}$ such that $\phi(t)= \sigma(\gamma_1(t),\gamma_2(t)) \circ \tilde\rho(t)$. 
\begin{comment}
In particular, 
\begin{equation*}
s(\tilde\rho(t))=\alpha_r(\phi(t))=\gamma_2(t)
\quand
t(\tilde\rho(t)) = \alpha_r(\sigma(\gamma_1(t),\gamma_2(t)))=R(\gamma_2(t),g(\gamma_1(t),\gamma_2(t)))\text{,}
\end{equation*}
\end{comment}
Comparing $\tilde\rho(t)$ with $\id_{\gamma_2(t)}$ via \cref{lem:actmorph}, we obtain a unique curve $\eta$ in $H$ such that
\begin{equation*}
\phi(t) =\rho(\sigma(\gamma_1(t),\gamma_2(t)),(\eta(t),g(\gamma_1(t),\gamma_2(t))^{-1})\text{.}
\end{equation*}  
\begin{comment}
Indeed, consider $\rho_2(t) := \tilde\rho(t)$, $\rho_1(t) := \id_{\gamma_2(t)}$, $g_1(t) :=1$ and $g_2(t) :=  g(\gamma_1(t),\gamma_2(t))$.
We have
\begin{equation*}
s(\rho_2(t))=s(\tilde\rho(t))=\gamma_2(t)=R(s(\rho_1(t)),g_1(t))
\end{equation*}
and
\begin{equation*}
t(\rho_2(t))=t(\tilde\rho(t))= R(\gamma_2(t),g(\gamma_1(t),\gamma_2(t))) =R(t(\rho_1(t)),g_2(t))\text{.}
\end{equation*}
By \cref{lem:actmorph} we obtain a curve $\eta'$ in $H$ such that $R(\id_{\gamma_2(t)},(\eta'(t),1)= \tilde\rho(t)$ and $t(\eta'(t))=g(\gamma_1(t),\gamma_2(t))$. Hence 
\begin{align*}
\phi(t) &= \sigma(\gamma_1(t),\gamma_2(t)) \circ R(\id_{\gamma_2(t)},(\eta'(t),1)
\\&= \sigma(\gamma_1(t),\gamma_2(t)) \circ R(\id_{R(\gamma_2(t),g(\gamma_1(t),\gamma_2(t)))},(\alpha(g(\gamma_1(t),\gamma_2(t))^{-1},\eta'(t)),g(\gamma_1(t),\gamma_2(t))^{-1})
\\& = \rho(\sigma(\gamma_1(t),\gamma_2(t)),(\alpha(g(\gamma_1(t),\gamma_2(t))^{-1},\eta'(t)),g(\gamma_1(t),\gamma_2(t))^{-1})
\\& = \rho(\sigma(\gamma_1(t),\gamma_2(t)),(\eta(t),g(\gamma_1(t),\gamma_2(t))^{-1})
\end{align*}
where $\eta(t):= \alpha(g(\gamma_1(t),\gamma_2(t))^{-1},\eta'(t))$, and $\rho: F \times \mor{\Gamma} \to F$  is the action on $F$. 
\end{comment}
Putting $h := \eta(0)$ and letting $w\in T_hH$ be the tangent vector of $\eta$, we obtain by taking derivatives
\begin{equation*}
v = T_{\sigma(u)}\rho_{h,g(u)^{-1}}(T_u\sigma(v_1,v_2)) + T_{h,g(u)^{-1}}\rho_{\sigma(u)} (w,T_ug^{-1}(v_1,v_2))\text{.}
\end{equation*}
This is the claimed splitting.  
\end{remark}

Our final definition is:
 
\begin{definition}
A \emph{2-morphism} between 1-morphisms is a $\Gamma$-equivariant transformation.
\end{definition}

Principal $\Gamma$-2-bundles over $M$ form a bigroupoid that we denote by $\zweibun\Gamma M$. Moreover, the assignment
\begin{equation*}
M \mapsto \zweibun \Gamma M
\end{equation*}
is a stack  over the site of smooth  manifolds \cite[Theorem 6.2.1]{Nikolaus}.

\subsection{Classification by non-abelian cohomology}

\label{sec:2bundleclass}

The following result has been obtain in \cite{wockel1} and  \cite[Theorems 5.3.2 \& 7.1]{Nikolaus}:

\begin{theorem}
\label{thm:classcoho}
Principal $\Gamma$-2-bundles over $M$ are classified by Giraud's  non-abelian cohomology, in terms of a bijection
\begin{equation*}
\hc 0 (\zweibun\Gamma M) \cong  \h^1(M,\Gamma)\text{,}
\end{equation*}
where $\hc 0$ denotes the set of isomorphism classes of a bicategory.
\end{theorem}

In the remainder of this section we give a new  direct proof of \cref{thm:classcoho}, which we will later refine to a situation \quot{with connections}.  Given a $\Gamma$-cocycle  $(g,a)$ with respect to an open cover $\mathcal{U}=\{U_i\}_{i\in I}$, we define a Lie groupoid $\inf P_{(g,a)}$ with 
\begin{equation*}
\ob{\inf P_{(g,a)}} := \coprod_{i\in I} U_i \times G
\quand
\mor{\inf P_{(g,a)}} := \coprod_{i,j\in I} (U_i \cap U_j) \times H \times G\text{.}
\end{equation*}
Target map and source map are defined by 
\begin{equation*}
s(i,j,x,h,g) := (j,x,g)
\quand
t(i,j,x,h,g) := (i,x,g_{ij}(x)^{-1}t(h)g)\text{,}
\end{equation*}
and the composition is
\begin{equation*}
(i,j,x,h_2,g_2) \circ (j,k,x,h_1,g_1) := (i,k,x,a_{ijk}(x)\alpha(g_{jk}(x),h_2)h_1, g_1)\text{.}
\end{equation*}
A smooth functor $\pi: \inf P_{(g,a)} \to \idmorph{M}$ 
is defined by $\pi(i,x,g):= x$. 
A smooth $\Gamma$-action is defined by 
\begin{equation*}
R((i,x,g),g') := (i,x,gg')
\quand
R((i,j,x,h,g),(h',g')):= (i,j,x,h\alpha(g,h'),gg')\text{.}
\end{equation*}
We have three things to check:
\begin{enumerate}[1.)]

\item 
$\inf P_{(g,a)}$ is a principal $\Gamma$-2-bundle. The main part is to show that the functor $\tau$ in \cref{def:zwoabun} is a weak equivalence; this is straightforward and can be done using \cref{th:weak}.

\item
If another $\Gamma$-cocycle $(g',a')$ is equivalent to $(g,a)$ via equivalence data $(h,e)$, then a smooth functor $\phi: \inf P_{(g,a)} \to \inf P_{(g',a')}$ is defined by \begin{equation*}
\phi(i,x,g):= (i,x,h_i(x)g)
\quand
\phi(i,j,x,h,g) := (i,j,x,e_{ij}(x)\alpha(h_j(x),h),h_j(x)g)\text{.}
\end{equation*}
The conditions for a functor are straightforward to check using \cref{32c,32f} for $(h,e)$.
\begin{comment}
For the target map, we have
\begin{multline*}
t(F(i,j,x,h,g))=t(i,j,x,e_{ij}(x)\alpha(h_j,h),h_j(x)g)\\=(i,x,g_{ij}'(x)^{-1}t(e_{ij}(x)\alpha(h_j,h))h_j(x)g)=(i,x,h_i(x)g_{ij}(x)^{-1}t(h)g)\\=F(i,x,g_{ij}(x)^{-1}t(h)g)=F(t(i,j,x,h,g))
\end{multline*}
Here we have used
\begin{equation*}
t(e_{ij})=g_{ij}'h_ig_{ij}^{-1}h_j^{-1}\text{.}
\end{equation*}
For the source map, we have
\begin{multline*}
s(F(i,j,x,h,g))=s(i,j,x,e_{ij}(x)\alpha(h_j,h),h_j(x)g)\\=(j,x,h_j(x)g)=F(j,x,g)=F(s(i,j,x,h,g))
\end{multline*}
For the composition, we have
\begin{align*}
&\mquad F((i,j,x,h_2,g_2) \circ (j,k,x,h_1,g_1)) \\&= F(i,k,x,a_{ijk}(x)\alpha(g_{jk}(x),h_2)h_1, g_1)
\\&= (i,k,x,e_{ik}(x)\alpha(h_k(x),a_{ijk}(x)\alpha(g_{jk}(x),h_2)h_1),h_k(x)g_1)
\\&= (i,k,x,e_{ik}(x)\alpha(h_k(x),a_{ijk}(x))e_{jk}(x)^{-1}\alpha(g'_{jk}(x)h_j,h_2)e_{jk}(x)\alpha(h_k(x),h_1),h_k(x)g_1)
\\&= (i,k,x,a_{ijk}'(x)\alpha(g'_{jk}(x),e_{ij}(x)\alpha(h_j,h_2))e_{jk}(x)\alpha(h_k(x),h_1),h_k(x)g_1)
\\&= (i,j,x,e_{ij}(x)\alpha(h_j,h_2),h_j(x)g_2) \circ (j,k,x,e_{jk}(x)\alpha(h_k(x),h_1),h_k(x)g_1)
\\&= F(i,j,x,h_2,g_2) \circ F(j,k,x,h_1,g_1)
\end{align*}
Here we have used that
\begin{equation*}
a_{ijk}'\alpha(g'_{jk},e_{ij}) = e_{ik}\alpha(h_k,a_{ijk})e_{jk}^{-1}
\end{equation*}
and that
\begin{equation*}
e_{jk}^{-1}\alpha(g'_{jk}h_j,h_2)e_{jk}=\alpha(t(e_{jk}^{-1})g'_{jk}h_j,h_2)=\alpha(h_k g_{jk},h_2)\text{.}
\end{equation*}
\end{comment}
It is also straightforward to check that $\phi$ is $\Gamma$-equivariant.
\begin{comment}
On the level of objects, this is obvious, and on the level of morphisms, this is
\begin{align*}
R(F(i,j,x,h,g),(h',g')) &= R((i,j,x,e_{ij}(x)\alpha(h_j(x),h),h_j(x)g),(h',g'))
\\&= (i,j,x,e_{ij}(x)\alpha(h_j(x),h)\alpha(h_j(x)g,h'),h_j(x)gg')
\\&= (i,j,x,e_{ij}(x)\alpha(h_j(x),h\alpha(g,h')),h_j(x)gg')
\\&= F((i,j,x,h\alpha(g,h'),gg'))
\\&= F(R((i,j,x,h,g),(h',g')))
\end{align*}
\end{comment}
Thus, by \cref{rem:equivsmoothfunctor} $\phi$ induces a 1-isomorphism $F: \inf P_{(g,a)} \to \inf P_{(g',a')}$.

\item
If $\mathcal{U}'$ is a refinement of $\mathcal{U}$, and $(g',a')$ is the refined $\Gamma$-cocycle, then the evident smooth functor $\phi: \inf P_{(g',a')} \to \inf P_{(g,a)}$ is obviously $\Gamma$-equivariant, and hence a 1-isomorphism.

\end{enumerate}
Now we have defined a \quot{reconstruction} map
\begin{equation}
\label{eq:reconmap}
\h^1(M,\Gamma) \to \hc 0 (\zweibun\Gamma M\text{.}
\end{equation}
In order to show that it is surjective, we extract a $\Gamma$-cocycle from a given principal $\Gamma$-2-bundle $\inf P$ over $M$.
We make the following choices:
\begin{enumerate}[1.)]

\item 
A cover $\{U_i\}_{i\in I}$ of $M$ by contractible open sets with contractible double intersections, together with smooth sections $s_i: U_i \to \ob{\inf P}$.

\item
\label{choices:b}
\Splitting s $\sigma_{ij}:U_i \cap U_j \to \mor{\inf P}$ along $(s_i,s_j):U_i \cap U_j \to \ob{\inf P} \times_M \ob{\inf P}$, together with transitions functions $g_{ij}:U_i \cap U_j \to G$.

\end{enumerate}
The sections $s_i$ exist due to \cref{lem:exsections}, and the \splitting s exist due to \cref{lem:existenceRsections}. Over a triple intersection $U_i \cap U_j \cap U_k$ it is straightforward to check that $\sigma_{ijk}:=\sigma_{ij} \circ R(\sigma_{jk},\id_{g_{ij}})$ is a \splitting\ along $(s_i,s_k)$ with associated transition function $ g_{jk}g_{ij}$. Comparing with the \splitting\ $\sigma_{ik}$ using \cref{lem:existenceRsections}, we obtain a unique smooth map
$a_{ijk}:U_i \cap U_j \cap U_k \to H$  satisfying
\begin{equation}
\label{eq:local:sub1}
t(a_{ijk})g_{jk}g_{ij}=g_{ik}
\quand
\sigma_{ijk}=R(\sigma_{ik},(\alpha(g_{ik}^{-1},a_{ijk}),g_{ik}^{-1}g_{jk}g_{ij}))\text{.}
\end{equation}
The first equation is the cocycle condition of \cref{eq:transtrans}.
We compute the expression
\begin{equation*}
\sigma_{ij} \circ R(\sigma_{jk},\id_{g_{ij}})\circ R(\sigma_{kl},\id_{g_{jk}g_{ij}})
\end{equation*}
in two ways:  (1) by substituting using \cref{eq:local:sub1}  $\sigma_{jkl}$ and then $\sigma_{ijl}$, leading to
\begin{equation*}
R(\sigma_{il},(\alpha(g_{il}^{-1},a_{ijl}a_{jkl}),g_{il}^{-1}g_{kl}g_{jk}g_{ij}))\text{,}
\end{equation*} 
\begin{comment}
Indeed,
\begin{eqnarray*}
&&\hspace{-1cm}\tilde \sigma_{ij} \circ R(\tilde\sigma_{jk},\id_{g_{ij}})\circ R(\tilde\sigma_{kl},\id_{g_{jk}g_{ij}}) \\&=& \tilde \sigma_{ij} \circ R(\tilde\sigma_{jk} \circ R(\tilde\sigma_{kl},\id_{g_{jk}}),\id_{g_{ij}})
\\&=&  \tilde \sigma_{ij} \circ R(R(\tilde\sigma_{jl},(\alpha(g_{jl}^{-1},a_{jkl}),g_{jl}^{-1}g_{kl}g_{jk})),\id_{g_{ij}})
\\&=& \tilde \sigma_{ij} \circ R(\tilde\sigma_{jl},(\alpha(g_{jl}^{-1},a_{jkl}),g_{jl}^{-1}g_{kl}g_{jk}g_{ij}))
\\&=& \tilde \sigma_{ij} \circ R(\tilde\sigma_{jl},\id_{g_{ij}}) \circ R(\id,(\alpha(g_{jl}^{-1},a_{jkl}),g_{jl}^{-1}g_{kl}g_{jk}g_{ij}))
\\&=& R(\tilde\sigma_{il},(\alpha(g_{il}^{-1},a_{ijl}),g_{il}^{-1}g_{jl}g_{ij})) \circ R(R(\id,\id_{g_{jl}^{-1}g_{il}}),(1,g_{il}^{-1}g_{jl})(\alpha(g_{jl}^{-1},a_{jkl}),g_{jl}^{-1}g_{kl}g_{jk}g_{ij}))
\\&=& R(\tilde\sigma_{il},(\alpha(g_{il}^{-1},a_{ijl})\alpha(g_{il}^{-1},a_{jkl}),g_{il}^{-1}g_{kl}g_{jk}g_{ij}))
\\&=& R(\tilde\sigma_{il},(\alpha(g_{il}^{-1},a_{ijl}a_{jkl}),g_{il}^{-1}g_{kl}g_{jk}g_{ij}))
\end{eqnarray*}
\end{comment}
and (2) by substituting $\sigma_{ijk}$ and then $\sigma_{ikl}$, leading to \begin{equation*}
R(\sigma_{il},(\alpha(g_{il}^{-1},a_{ikl}\alpha(g_{kl},a_{ijk})),g_{il}^{-1}g_{kl}g_{jk}g_{ij}))\text{.}
\end{equation*}
\begin{comment} 
Indeed,
\begin{eqnarray*}
&&\hspace{-1cm}\tilde \sigma_{ij} \circ R(\tilde\sigma_{jk},\id_{g_{ij}})\circ R(\tilde\sigma_{kl},\id_{g_{jk}g_{ij}}) \\&=& R(\tilde\sigma_{ik},(\alpha(g_{ik}^{-1},a_{ijk}),g_{ik}^{-1}g_{jk}g_{ij}))\circ R(R(\tilde\sigma_{kl},\id_{g_{ik}}),\id_{g_{ik}^{-1}g_{jk}g_{ij}})
\\&=&  R(\tilde\sigma_{ik} \circ R(\tilde\sigma_{kl},\id_{g_{ik}}),(\alpha(g_{ik}^{-1},a_{ijk}),g_{ik}^{-1}g_{jk}g_{ij}))
\\&=&  R(R(\tilde\sigma_{il},(\alpha(g_{il}^{-1},a_{ikl}),g_{il}^{-1}g_{kl}g_{ik})),(\alpha(g_{ik}^{-1},a_{ijk}),g_{ik}^{-1}g_{jk}g_{ij}))
\\&=&  R(\tilde\sigma_{il},(\alpha(g_{il}^{-1},a_{ikl}),g_{il}^{-1}g_{kl}g_{ik})\cdot(\alpha(g_{ik}^{-1},a_{ijk}),g_{ik}^{-1}g_{jk}g_{ij}))
\\&=&  R(\tilde\sigma_{il},(\alpha(g_{il}^{-1},a_{ikl}\alpha(g_{kl},a_{ijk})),g_{il}^{-1}g_{kl}g_{jk}g_{ij}))
\end{eqnarray*}
\end{comment}
Equating the two results, the uniqueness in \cref{lem:actmorph} implies
$a_{ijl}a_{jkl}=a_{ikl}\alpha(g_{kl},a_{ijk})$, 
this is cocycle condition \cref{31c}.
There is a smooth $\Gamma$-equivariant functor  $\phi: \inf P_{(g,a)} \to \inf P$ defined by
\begin{equation*}
\phi(i,x,g) := R(s_i(x),g)
\quand
\phi(i,j,x,h,g) := R(\sigma_{ij}(x),(\alpha(g_{ij}(x)^{-1},h),g_{ij}(x)^{-1}g))\text{.}
\end{equation*}
It induces a 1-isomorphism and hence  shows that the map of \cref{eq:reconmap} is surjective. 

It remains to show that the map  of \cref{eq:reconmap} is injective. We consider two $\Gamma$-cocycles $(g,a)$ and $(g',a')$ with respect to the same open cover, and a 1-morphism $F: \inf P_{(g,a)} \to \inf P_{(g',a')}$. Note that $\inf P_{(g,a)}$ has the sections $s_i(x):=(i,x,1)$, and the \splitting s $\sigma_{ij}(x)=(i,j,x,1,g_{ij}(x))$ with transition functions $g_{ij}$. Similarly  $\inf P_{(g',a')}$ has  sections $s_i'$ and \splitting s $\sigma_{ij}'$. 
 We choose \splitting s $\sigma_i:U_i \to F$ of $F$ along $(s_i,s_i'):U_i \to \ob{\inf P} \times_M \ob{\inf P'}$ with transition functions $h_i: U_i \to G$. Then, we obtain two \splitting s of $F$ along $(s_i,s_j'):U_i \cap U_j \to F$, namely $\sigma_i \circ R(\sigma'_{ij},\id_{h_i})$, with transition function $g_{ij}'h_i$, and $\sigma_{ij}\circ \rho(\sigma_j,\id_{g_{ij}})$, with transition function $h_jg_{ij}$. By \cref{lem:Fsections}, there exists a unique smooth map $e_{ij}:U_i \cap U_j \to H$ such that 
\begin{equation}
\label{eq:local:sub0}
t(e_{ij})h_jg_{ij}=g_{ij}'h_i
\end{equation}
and 
\begin{equation}
\label{eq:local:sub2}
\sigma_{ij}\circ \rho(\sigma_j,\id_{g_{ij}}) = \rho(\sigma_i \circ R(\sigma'_{ij},\id_{h_i}),(\alpha(h_i^{-1}g_{ij}'^{-1},e_{ij}),h_i^{-1}g_{ij}'^{-1}h_jg_{ij}))\text{.}
\end{equation}
Now we compute $\sigma_{ij} \circ R(\sigma_{jk},\id_{g_{ij}}) \circ \rho(\sigma_k,\id_{g_{jk}g_{ij}})$ in two different ways: (1) by substituting $\sigma'_{jk}$, then $\sigma'_{ij}$ via \cref{eq:local:sub2}, and then $\sigma_{ik}'$ via \cref{eq:local:sub1},
and (2) by substituting $\sigma_{ik}$ via \cref{eq:local:sub1} and then $\sigma_{ik}'$ via \cref{eq:local:sub2}.
\begin{comment}
\begin{eqnarray*}
&&\hspace{-2cm}\sigma_{ij} \circ R(\sigma_{jk},\id_{g_{ij}}) \circ R(\sigma_k,\id_{g_{jk}g_{ij}})
\\&=& R(\sigma_{ik},(\alpha(g_{ik}^{-1},a_{ijk}),g_{ik}^{-1}g_{jk}g_{ij})) \circ R(R(\sigma_k,\id_{g_{ik}}),\id_{g_{ik}^{-1}g_{jk}g_{ij}})
\\&=& R(\sigma_{ik} \circ R(\sigma_k,\id_{g_{ik}}),(\alpha(g_{ik}^{-1},a_{ijk}),g_{ik}^{-1}g_{jk}g_{ij})) \\&=& R(\sigma_i \circ R(\sigma'_{ik},\id_{h_i}),(\alpha(h_i^{-1}g_{ik}'^{-1},e_{ik}),h_i^{-1}g_{ik}'^{-1}h_kg_{ik}))\cdot (\alpha(g_{ik}^{-1},a_{ijk}),g_{ik}^{-1}g_{jk}g_{ij})) 
\\&=& R(\sigma_i \circ R(\sigma'_{ik},\id_{h_i}),(\alpha(h_i^{-1}g_{ik}'^{-1},e_{ik})\alpha(h_i^{-1}g_{ik}'^{-1}h_k,a_{ijk}),h_i^{-1}g_{ik}'^{-1}h_kg_{jk}g_{ij})) \\&=& \sigma_i \circ R(R(\sigma'_{ik},\id_{g_{ik}'^{-1}}),(e_{ik}\alpha(h_k,a_{ijk}),h_kg_{jk}g_{ij})) \end{eqnarray*}
\end{comment}
Comparing the two results using that the right $\inf P'$-action  on $F$ is free and the uniqueness of \cref{lem:actmorph} we can then conclude
\begin{equation}
\label{eq:local:sub3}
a_{ijk}'\alpha(g'_{jk},e_{ij})e_{jk} = e_{ik}\alpha(h_k,a_{ijk})\text{.}
\end{equation} 
\Cref{eq:local:sub0,eq:local:sub3} show that the $\Gamma$-cocycles $(g,a)$ and $(g',a')$ are equivalent. This concludes the proof of \cref{thm:classcoho}.

\begin{remark}
If $\Gamma$ is smoothly separable, there is a classifying space for principal $\Gamma$-2-bundles, namely the classifying space of the geometric realization of $\Gamma$, $B|\Gamma|$. Thus, there is a bijection
\begin{equation*}
\hc 0 (\zweibun\Gamma M) \cong [M,B|\Gamma|]
\end{equation*}
with the set of homotopy classes of continuous maps from $M$ to $B|\Gamma|$,
see \cite{baez8} and \cite[Section 4]{Nikolaus}.
\end{remark}

\setsecnumdepth{2}

\section{Lie 2-algebra-valued differential forms on Lie groupoids}

\label{sec:diffforms}

\subsection{The differential graded Lie algebra of forms}

We start with an arbitrary Lie 2-algebra $\gamma = (\mathfrak{h},\mathfrak{g},t_{*},\alpha_{*})$ in the formalism described in \cref{sec:lie2}. We give a short motivation for our \cref{def:forms} below.

We  regard $\gamma$ as a cochain complex $(\gamma_{\bullet},t_{\bullet})$ with $\gamma_{-1} := \mathfrak{h}$, $\gamma_0 := \mathfrak{g}$ and $\gamma_{k}=0$ for all $k\neq 0,-1$. The only non-trivial differential is $t_{-1} := t_{*}$. 
If $\inf P$  is a Lie groupoid, its nerve is a simplicial manifold $\inf P_{\bullet}$ with $\inf P_0 := \ob{\inf P}$, $\inf P_1 := \mor{\inf P}$, and  $\inf P_k \df \mor{\inf P} \ttimes st \mor{\inf P} \ttimes s{}...\times_t \mor{\inf P}$ the manifolds of $k$ composable morphisms.  We  denote the alternating sum over the pullbacks along the face maps by $\Delta_{j}:\Omega^*(\inf P_j) \to \Omega^{*}(\inf P_{j+1})$, satisfying $\Delta_{j+1} \circ \Delta_j=0$.  For instance, $\Delta_0 = t^{*}-s^{*}$, and $\Delta_1=\pr_1^{*}-c^{*}+\pr_2^{*}$ where $c: \inf P_2  \to \inf P_1$ is the composition. 
We form the total complex
\begin{equation*}
T^{p} :=\bigoplus_{p=i+j+k} \Omega^{i}(\inf P_j,\gamma_k)
\end{equation*}
and equip it with the differential
\begin{equation*}
\tau_{ijk} := \mathrm{d}_i + (-1)^{i}\Delta_j + (-1)^{i+j}t_k: T^{p} \to T^{p+1}\text{.} 
\end{equation*}
We want to discard all elements of form degree smaller than the total degree (\quot{virtual forms}), because these components do not naturally appear in the context of connections on principal 2-bundles. Hence, we decompose 
\begin{equation*}
T^{p} = T^{p}_{true} \oplus T^{p}_{virt}
\end{equation*}
with $T^{p}_{true}$  consisting of those summands with form degree $i\geq p$. 
\begin{comment}
\begin{equation*}
T^{p}_{true} :=\bigoplus_{p=i+j+k,i\geq p} \Omega^{i}(\inf P_j,\gamma_k) \quand
T^{p}_{virt} :=\bigoplus_{p=i+j+k,i < p} \Omega^{i}(\inf P_j,\gamma_k)\text{.}
\end{equation*}
\end{comment}
The differential $\tau$ does not restrict to one on $T^{p}_{true}$. Thus, we also decompose
\begin{equation*}
\tau|_{T^{p}_{true}} = \mathrm{D} \oplus \mathrm{D}_{virt}
\end{equation*}
with $\mathrm{D}: T^{p}_{true} \to T^{p+1}_{true}$ and $\mathrm{D}_{virt}:T^{p}_{true} \to T^{p+1}_{virt}$
and define
\begin{equation*}
\Omega^p(\inf P,\gamma) := \mathrm{kern}(\mathrm{D}_{virt}) \subset T^{p}_{true}\text{.}
\end{equation*}
By construction, we obtain a well-defined differential $\mathrm{D}:\Omega^p(\inf P,\gamma) \to \Omega^{p+1}(\inf P,\gamma)$ satisfying $\mathrm{D}^2=0$. 
Spelling everything out, we obtain the following definition:

\begin{definition}
\label{def:forms}
A $p$-form $\Psi\in\Omega^p(\inf P,\gamma)$ consists of three components $\Psi=(\fa\Psi,\fb\Psi,\fc\Psi)$, which are ordinary differential forms
\begin{equation*}
\fa\Psi\in\Omega^{p}(\inf P_{0},\mathfrak{g})
\quomma
\fb\Psi\in\Omega^{p}(\inf P_{1},\mathfrak{h})
\quand
\fc\Psi\in\Omega^{p+1}(\inf P_{0},\mathfrak{h})\text{,}
\end{equation*}
such that $\Delta \fa\Psi = t_{*}(\fb\Psi)$ and $\Delta\fb\Psi = 0$. 
The differential of a $p$-form $\Psi$ is the $(p+1)$-form $\mathrm{D}\Psi$ with components
\begin{eqnarray*}
\fa{(\mathrm{D}\Psi)} &=&  \mathrm{d}\fa\Psi -(-1)^{p}t_{*}(\fc\Psi)
\\
\fb{(\mathrm{D}\Psi)} &=& \mathrm{d}\fb\Psi -(-1)^{p}\Delta(\fc\Psi)
\\
\fc{(\mathrm{D}\Psi)} &=& \mathrm{d}\fc\Psi\text{.}
\end{eqnarray*}
\end{definition}

\begin{comment}
We remark that $\Delta\fb\Psi=0$ implies $\id^{*}\fb\Psi = 0$ and  $i^{*}\fb\Psi = -\Psi$, where $i:\inf P_1 \to \inf P_1$ is the inversion.
For composable $\rho_1,\rho_2\in \mor{\inf P}$ we have $0=\Delta \fb\Psi|_{\rho_1,\rho_2}= \fb\Psi_{\rho_1}-\fb\Psi_{\rho_1\circ \rho_2}+ \fb\Psi_{\rho_2}$. For $\rho_2=\rho_1^{-1}$ we get $\fb\Psi_{\rho_1}=-\fb\Psi_{\rho_2}$. 
\end{comment}

Next we define a graded product. 
Suppose $\Psi\in \Omega^{k}(\inf P,\gamma)$ and $\Phi\in\Omega^l(\inf P,\gamma)$. We define $[\Psi\wedge \Phi]\in\Omega^{k+l}(\inf P,\gamma)$  in the following way:
\begin{eqnarray*}
\fa{[\Psi\wedge \Phi]} &=& [\fa\Psi \wedge \fa\Phi] 
\\
\fb{[\Psi\wedge \Phi]} &=& [\fb\Psi\wedge \fb\Phi]+ \alpha_{*}(s^{*}\fa\Psi \wedge \fb\Phi)- (-1)^{kl}  \alpha_{*}(s^{*}\fa\Phi \wedge \fb\Psi) 
\\
\fc{[\Psi\wedge \Phi]} &=&  \alpha_{*}(\fa\Psi \wedge \fc\Phi)- (-1)^{kl} \alpha_{*}(\fa\Phi \wedge \fc\Psi)\text{.}  
\end{eqnarray*} 
\begin{comment}
We check the conditions \erf{eq:condform}:
\begin{eqnarray*}
\Delta(\fa{[\Psi\wedge \Phi]}) &=& [t^{*}\fa\Psi \wedge t^{*} \fa\Phi]-[s^{*}\fa\Psi \wedge s^{*} \fa\Phi]
\\&=& [\Delta\fa\Psi \wedge t^{*} \fa\Phi]+[s^{*}\fa\Psi \wedge t^{*} \fa\Phi]-[s^{*}\fa\Psi \wedge s^{*} \fa\Phi]
\\&=&  [\Delta\fa\Psi \wedge \Delta \fa\Phi]+[s^{*}\fa\Psi \wedge \Delta \fa\Phi]- (-1)^{kl}  [s^{*}\fa\Phi \wedge \Delta\fa\Psi ]
\\&=&[t_{*}\fb\Psi\wedge t_{*}\fb\Phi]+t_{*} \alpha_{*}(s^{*}\fa\Psi \wedge \fb\Phi)- (-1)^{kl}  t_{*} \alpha_{*}(s^{*}\fa\Phi \wedge \fb\Psi)
\\&=& t_{*}(\fb{[\Psi\wedge \Phi]})
\\
\Delta(\fb{[\Psi\wedge \Phi]} ) &=&  [\Delta\fb\Psi\wedge \Delta\fb\Phi]+ \alpha_{*}(\Delta s^{*}\fa\Psi \wedge\Delta \fb\Phi)+ \alpha_{*}(\Delta s^{*}\fa\Phi \wedge\Delta \fb\Psi)
\\&=&0 
\end{eqnarray*}
\end{comment}
Well-definedness of this definition, bilinearity  and the following lemma are straightforward to check. 

\begin{lemma}
\label{lem:bracket}
The graded product equips $\gamma$-valued differential forms with the structure of a differential graded Lie algebra, i.e. for $\Psi\in\Omega^k(\inf P,\gamma)$, $\Phi\in\Omega^l(\inf P,\gamma)$ and $\Xi\in\Omega^{i}(\inf P,\gamma)$ we have:
\begin{enumerate}[(a)]

\item
\label{eq:comm}
It is graded-anti-commutative, 
\begin{equation*}
[\Phi\wedge \Psi]=-(-1)^{kl}[\Psi \wedge \Phi]\text{.}
\end{equation*}

\item
\label{eq:jacobi}
It satisfies a graded Jacobi identity,
\begin{equation*}
(-1)^{ki}[[\Psi\wedge \Phi]\wedge \Xi] + (-1)^{li}[[\Xi \wedge \Psi] \wedge \Phi] + (-1)^{kl}[[\Phi\wedge \Xi]\wedge \Psi]=0\text{.}
\end{equation*}

\item
\label{eq:leibnitz}
It satisfies a Leibnitz rule with respect to the differential,
\begin{equation*}
\mathrm{D}[\Psi \wedge \Phi] = [\mathrm{D}\Psi \wedge \Phi]+(-1)^{k}[\Psi\wedge \mathrm{D}\Phi]\text{.}
\end{equation*}

\end{enumerate}
\end{lemma}

\begin{example}
We consider $\Gamma=\idmorph{G}$ for a Lie group $G$.
\begin{enumerate}[(a)]

\item
\label{ex:forms:codis}
A $p$-form $\Psi\in\Omega^p(\inf P,\gamma)$ only has the component $\fa\Psi\in\Omega^{p}(\ob{\inf P},\mathfrak{g})$ subject to the condition $\Delta\fa\Psi=0$. Differential and Lie bracket are the ordinary ones.

\item
If $\inf P=X_{dis}$  we obtain $\Omega^p(X_{dis},\gamma)=\Omega^p(X,\mathfrak{g})$ as differential graded Lie algebras.

\item
If $\inf P=\act MK$ is the action groupoid associated to the action of a Lie group $K$ on a smooth manifold $M$, 
we obtain $K$-basic $\mathfrak{g}$-valued forms on $M$,  
\begin{equation*}
\Omega^p(\act MK,\gamma)=\Omega^{p}(M,\mathfrak{g})^{K\text{-}basic}\text{,}
\end{equation*}
as differential graded Lie algebras.
\end{enumerate}
\end{example}

\begin{example}
We consider    $\Gamma=BA,$ for an abelian Lie group $A$ with Lie algebra $\mathfrak{a}$.
\begin{enumerate}[(a)]

\item
\label{ex:forms:codisecht}
A  $p$-form $\Psi\in\Omega^p(\inf P,\gamma)$ has  components $\fb\Psi\in\Omega^{p}(\mor{\inf P},\mathfrak{a})$
and $\fc\Psi\in\Omega^{p+1}(\ob{\inf P},\mathfrak{a})$, subject to the condition $\Delta\fb\Psi=0$. 
The differential  is given by $\fb{(\mathrm{D}\Psi)} = \mathrm{d}\fb\Psi -(-1)^{p}\Delta(\fc\Psi)$ and $\fc{(\mathrm{D}\Psi)} = \mathrm{d}\fc\Psi$.

\item
For the action groupoid $\inf P=\act\ast K$ we obtain \quot{multiplicative} forms on $K$:
\begin{equation*}
\Omega^p(\act\ast K,\gamma) =  \{ \psi\in \Omega^p(K,\mathfrak{a}) \sep  \forall g_1,g_2\in K: \psi_{g_1} + \psi_{g_2}= \psi_{g_1g_2} \}\text{.}
\end{equation*}

\end{enumerate}
\end{example}

\label{sec:wedge}

\begin{example}
\label{ex:mc}
Every Lie 2-group $\Gamma$ supports a canonical, $\gamma$-valued \quot{Maurer-Cartan} $1$-form $\Theta \in \Omega^1(\Gamma,\gamma)$. Its non-trivial components are   $\fa\Theta\in\Omega^1(G,\mathfrak{g})$ and $\fb\Theta\in\Omega^1(H \times G,\mathfrak{h})$, given by 
\begin{equation*}
\fa\Theta := \theta^{G} \in \Omega^1(G,\mathfrak{g})
\quand
\fb\Theta := (\alpha_{\pr_G^{-1}})_{*}(\pr_H^{*}\theta^{H}) \in \Omega^1(H \times G,\mathfrak{h})\text{,}
\end{equation*}
where $\theta^{G}$ and $\theta^{H}$ are the ordinary Maurer-Cartan forms of the Lie groups $G$ and $H$, respectively.
\begin{comment}
Indeed, we have
\begin{multline*}
\Delta \fa\Theta|_{h,g} = \theta_{t(h)g} - \theta_{g}=g^{-1}\theta_{t(h)}g=g^{-1}t_{*}(\theta_h)g\\=t_{*}(\alpha(g^{-1},\theta_h))=t_{*}(\alpha_{g^{-1}})_{*}\theta_h= t_{*}(\fb\Theta)\text{.}
\end{multline*}
Further, we have for elements $(h_2,g_2)$, $(h_1,g_1)$ with $g_2=t(h_1)g_1$
\begin{eqnarray*}
\Delta(\fb\Theta)|_{(h_2,g_2),(h_1,g_1)}  &=& \fb\Theta|_{(h_1,g_1)} + \fb\Theta|_{(h_2,g_2)} - \fb\Theta|_{(h_2h_1,g_1)} \\ &=& (\alpha_{g_1^{-1}})_{*}(\theta_{h_1})+(\alpha_{g_2^{-1}})_{*}(\theta_{h_2}) -(\alpha_{g_1^{-1}})_{*}(\theta_{h_2h_1}) \\&=&(\alpha_{g_1^{-1}})_{*} (\theta_{h_1} + h_1^{-1}\theta_{h_2}h_1-\theta_{h_2h_1})
\\&=& 0\text{.}
\end{eqnarray*}
\end{comment}
$\Theta$ satisfies the Maurer-Cartan equation,
\begin{equation}
\label{eq:mc}
\mathrm{D}\Theta + \frac{1}{2}[\Theta \wedge \Theta]=0\text{.}
\end{equation}
\begin{comment}
Indeed,
\begin{eqnarray*}
\fa{\mathrm{curv}(\Theta)} &=& \mathrm{d}\fa\Theta +\frac{1}{2}[\fa\Theta \wedge \fa\Theta] 
\\&=& \mathrm{d}\theta^{G} +\frac{1}{2}[\theta^{G} \wedge \theta^{G}]
\\&=&0 
\\
\fb{\mathrm{curv}(\Theta)} &=& \mathrm{d}((\alpha_{\pr_2^{-1}})_{*}(\pr_1^{*}\theta))+\frac{1}{2}  [(\alpha_{\pr_2^{-1}})_{*}\pr_1^{*}\theta \wedge (\alpha_{\pr_2^{-1}})_{*}\pr_1^{*}\theta]+ \alpha_{*}(\pr_2^{*}\theta\wedge(\alpha_{\pr_2^{-1}})_{*}\pr_1^{*}\theta )
\\
&=& (\alpha_{\pr_2^{-1}})_{*}(\pr_1^{*}\mathrm{d}\theta+\frac{1}{2}  [\pr_1^{*}\theta \wedge\pr_1^{*}\theta])\\ &=&0
\\
\fc{\mathrm{curv}(\Theta)} &=&  0\text{.}
\end{eqnarray*} 
\end{comment}
\begin{comment}
For $\Gamma=B\ueins$ we have $\mathfrak{h}=\R$ and $\mathfrak{g}=0$. The Maurer-Cartan form $\Theta$ has only one non-trivial component, $\fb\Theta = \theta\in\Omega^1(\ueins)$. For $\Gamma=G_{dis}$ we have $\mathfrak{h}=0$, so that $\Theta$ has only the component $\fa\Theta=\theta\in\Omega^1(G,\mathfrak{g})$. \end{comment}
\end{example}

\begin{remark}
If $\phi: \inf P \to \inf Q$   is a smooth functor between Lie groupoids, then we obtain an obvious pullback map
\begin{equation*}
\phi^{*}:\Omega^k(\inf Q,\gamma) \to \Omega^{k}(\inf P,\gamma)
\end{equation*}
defined component-wise. It is linear and commutes with $\mathrm{D}$ and the wedge product. One can show that if $\eta:\phi_1 \Rightarrow \phi_2$ is a smooth natural transformation and $\Phi\in\Omega^p(\inf Q,\gamma)$ is closed, then there exists $\Phi_{\eta}\in\Omega^{p-1}(\inf P,\gamma)$ such that a \quot{homotopy formula} holds,
\begin{equation*}
\phi_2^{*}\Phi - \phi_1^{*}\Phi = \mathrm{D}\Phi_{\eta}\text{.}
\end{equation*}
We extend the pullback in \cref{sec:pullbackana} to anafunctors. 
\end{remark}

\subsection{Adjoint action}

If $X$ is a smooth manifold equipped with a smooth map $g:X \to G$, then we have an adjoint action $\mathrm{Ad}_g: \Omega^p(X,\mathfrak{g}) \to \Omega^{p}(X,\mathfrak{g})$ on ordinary differential forms defined at each point $x\in X$ by $\mathrm{Ad}_g(\varphi)_x := \mathrm{Ad}_{g(x)}(\varphi_x)$. 
We generalize this to $\gamma$-valued differential forms on a  Lie groupoid $\inf P$, where $\gamma$ is now the Lie 2-algebra of a Lie 2-group $\Gamma$.  

Suppose we have a smooth functor $F: \inf P \to \Gamma$. We  write $F_0: \ob{\inf P} \to G$ for its map of objects, as well as   $F_G:= \pr_G \circ F:  \mor{\inf P} \to G$ and $F_H := \pr_H \circ F: \mor{\inf P} \to H$ on the level of morphisms. 
\begin{comment}
We note that
\begin{eqnarray*}
F_G &=& \pr_G \circ F = s \circ F = F_0 \circ s
\\
F_0 \circ t &=& t\circ F_1 = (t\circ F_H)\cdot F_G
\end{eqnarray*}
\end{comment}
We define
        a linear map\begin{equation*}
\mathrm{Ad}_F: \Omega^k(\inf P,\gamma) \to \Omega^k(\inf P,\gamma)
\end{equation*}
in the following way:
\begin{eqnarray*}
\fa{\mathrm{Ad}_F(\Psi)} &=& \mathrm{Ad}_{F_0}(\fa\Psi) 
\\
\fb{\mathrm{Ad}_F(\Psi)} &=& \mathrm{Ad}_{F_H}((\alpha_{F_G})_{*} (\fb\Psi))+(\tilde\alpha_{F_H^{-1}})_{*}( \mathrm{Ad}_{F_G}(s^{*}\fa\Psi)) 
\\
\fc{\mathrm{Ad}_F(\Psi)} &=&  (\alpha_{F_0})_{*}(\fc\Psi)
\end{eqnarray*}
Well-definedness and the following lemma are straightforward to check.

\begin{lemma}
\label{lem:adjoint}
The adjoint action has the following properties:
\begin{enumerate}[(a)]

\item 
It is an action in the sense that 
\begin{equation*}
\mathrm{Ad}_1=\id
\quand
\mathrm{Ad}_{F \cdot F'} = \mathrm{Ad}_{F} \circ \mathrm{Ad}_{F'}\text{,}
\end{equation*}
where $1$ is the constant functor with values $1\in \ob{\Gamma}$ and $\id_1\in \mor{\Gamma}$, and $F \cdot F'$ is the point-wise product of  two $\Gamma$-valued functors. In particular,  $\mathrm{Ad}^{-1}_F=\mathrm{Ad}_{F^{-1}}$.

\item
It respects the Lie bracket:
\begin{equation*}
\mathrm{Ad}_F([\Psi \wedge \Phi]) = [\mathrm{Ad}_F(\Psi) \wedge \mathrm{Ad}_{F}(\Phi)]\text{.}
\end{equation*}

\item
\label{eq:adjder}
Its derivative is given by
\begin{equation*}
\mathrm{D}(\mathrm{Ad}_F^{-1}(\Psi)) = \mathrm{Ad}_{F}^{-1}(\mathrm{D}\Psi)-[ F^{*}\Theta \wedge \mathrm{Ad}_F^{-1}(\Psi)]\text{,}
\end{equation*}
where  $\Theta \in \Omega^1(\Gamma,\gamma)$ is the Maurer-Cartan form of \cref{ex:mc}.  
\end{enumerate}
\end{lemma}

\begin{example}
\label{ex:mcmult}
Let $m: \Gamma \times \Gamma \to \Gamma$ be the multiplication of $\Gamma$, and let $\Theta \in \Omega^1(\Gamma,\gamma)$ be the Maurer-Cartan form on the Lie 2-group $\Gamma$, see \cref{ex:mc}.  We have
\begin{equation*}
m^{*}\Theta = \mathrm{Ad}^{-1}_{\pr_2}(\pr_1^{*}\Theta) + \pr_2^{*}\Theta\text{.}
\end{equation*}
\end{example}

\subsection{Pullback along anafunctors}

\label{sec:pullbackana}

We want to pull back differential forms along anafunctors. Not to much surprise, this pullback will be relative to additional structure on the anafunctor.

\begin{definition}
\label{def:pullback}
Let $F:\inf X \to \inf Y$ be an anafunctor and 
let $\Phi\in\Omega^p(\inf Y,\gamma)$.
A \emph{$\Phi$-pullback} on $F$ is a pair $\nu=(\nu_0,\nu_1)$ of a $p$-form $\nu_0\in \Omega^{p}(F,\mathfrak{h})$ and  a $(p+1)$-form $\nu_1\in\Omega^{p+1}(F,\mathfrak{h})$ that satisfy over $F \ttimes {\alpha_r}t \mor{\inf Y}$ the conditions
\begin{align}
\label{eq:conana:cr}
\rho_r^{*}\nu_0&= \pr_F^{*}\nu_0+ \pr_{\mor{\inf Y}}^{*}\fb\Phi
\\
\label{eq:conana:cs}
\rho_r^{*}\nu_1 &= \pr_F^{*}\nu_1 + \pr_{\mor{\inf Y}}^{*}\Delta\fc\Phi\text{,}
\end{align}
where $\rho_r$ denotes the right $\inf Y$-action.
\end{definition}

\begin{lemma}
\label{lem:pullexists}
Let $\Phi\in \Omega^{p}(\inf Y,\gamma)$.
Every anafunctor $F:\inf X \to \inf Y$ admits a $\Phi$-pullback, and the set of all  $\Phi$-pullbacks on $F$ is an affine space. \end{lemma}

\begin{proof}
We choose an open cover $\mathcal{V}=\{V_\alpha\}_{\alpha\in A}$ of $\ob{\inf X}$ together with local trivializations $F|_{V_{\alpha}} \cong V_{\!\alpha}\, \ttimes{y_{\alpha}}{t}\mor{\inf Y}$ of $F$, as a principal  $\inf Y$-bundle over $\ob{\inf X}$. We denote by $l_\alpha: F|_{V_\alpha} \to \mor{\inf Y}$ the projection, which is  equivariant in the sense that $l_\alpha(f \circ \eta)=l_\alpha(f)\circ \eta$ for all $f\in F|_{V_\alpha}$ and $\eta\in \mor{\inf Y}$. Let $\{\phi_{\alpha}\}$ be a partition of unity subordinate to $\mathcal{V}$. Define
\begin{equation*}
\nu_0 := \sum_{\alpha\in A} \alpha_l^{*}\phi_{\alpha}\cdot l_\alpha^{*}\fb\Phi
\quand
\nu_1 := \sum_{\alpha\in A} \alpha_l^{*}\phi_{\alpha}\cdot l_\alpha^{*}\Delta\fc\Phi\text{.}
\end{equation*}
Using the equivariance of $l_{\alpha}$, the two conditions of \cref{def:pullback} are easy to check. 
\begin{comment}
Indeed,
\begin{align*}
\nu_0|_{f \circ \eta} &= \sum_{\alpha\in A} \phi_{\alpha}(\alpha_l(f \circ \eta))\cdot \fb\Phi_{l_\alpha(f \circ \eta)}
\\&= \sum_{\alpha\in A} \phi_{\alpha}(\alpha_l(f))\cdot \fb\Phi_{l_\alpha(f) \circ \eta}
\\&= \sum_{\alpha\in A} \phi_{\alpha}(\alpha_l(f))\cdot \fb\Phi_{l_\alpha(f)}+\sum_{\alpha\in A} \phi_{\alpha}(\alpha_l(f))\cdot \fb\Phi_{\eta}
\\&= \nu_0|_{f} + \fb\Phi_{\eta}\text{.}
\end{align*}
and
\begin{align*}
\nu_0|_{f \circ \eta} &= \sum_{\alpha\in A} \phi_{\alpha}(\alpha_l(f \circ \eta))\cdot \Delta\fc\Phi_{l_\alpha(f \circ \eta)}
\\&= \sum_{\alpha\in A} \phi_{\alpha}(\alpha_l(f))\cdot \Delta\fc\Phi_{l_\alpha(f) \circ \eta}
\\&= \sum_{\alpha\in A} \phi_{\alpha}(\alpha_l(f))\cdot (\fc\Phi_{t(l_{\alpha}(f))} -\fc\Phi_{s(l_{\alpha}(f))} +\fc\Phi_{t(\eta)} -\fc\Phi_{s(\eta)})
\\&= \sum_{\alpha\in A} \phi_{\alpha}(\alpha_l(f))\cdot \Delta\fc\Phi_{l_\alpha(f)}+\sum_{\alpha\in A} \phi_{\alpha}(\alpha_l(f))\cdot \Delta\fc\Phi_{\eta}
\\&= \nu_0|_{f} + \Delta\fc\Phi_{\eta}\text{.}
\end{align*}
Here we have used that $s(l_{\alpha}(f)) = \alpha_r(f)=t(\eta)$.
\end{comment}
We have an action of $\Omega^{p}(\ob{\inf X},\mathfrak{h}) \oplus \Omega^{p+1}(\ob{\inf X},\mathfrak{h})$ on the set of $\Phi$-pullbacks, where $(\kappa_0,\kappa_1)$ take a $\Phi$-pullback $\nu$ to  $(\nu_0+ \alpha_l^{*}\kappa_0,\nu_1+\alpha_l^{*}\kappa_1)$. It is straightforward to see that this action is free and  transitive. 
\begin{comment}
Indeed, we have, since the left anchor is preserved by the right action,
\begin{equation*}
\rho_r^{*}(\nu_0+ \alpha_l^{*}\kappa_0) = \pr_F^{*}\nu_0+ \pr_{\mor{\inf Y}}^{*}\fb\Phi + \pr_F^{*}\alpha_l^{*}\kappa_0\text{.}
\end{equation*}
Conversely, suppose $\nu$ and $\nu'$ are $\Phi$-pullback forms on $F$. Then we have for $(f_1,f_2)\in F \times_{\ob{\inf X}}F$ with $f_2=\rho_r(f_1,\eta)$:
\begin{equation*}
(\nu'_0-\nu_0)_{f_2}-(\nu'_0-\nu_0)_{f_1} = \nu'_0|_{\rho_r(f_1,\eta)}-\nu_0|_{\rho_r(f_1,\eta)}-\nu_0'|_{f_1}+\nu_0|_{f_1} =0\text{.}
\end{equation*}
Thus, there exists a unique $\kappa_0\in\Omega^p(\ob{\inf X},\mathfrak{h})$ with $\alpha_l^{*}\kappa_0=\nu'_0-\nu_0$.
The same argument works for $\nu_1$.
\end{comment} 
\end{proof}

\begin{lemma}
\label{lem:pullback}
Suppose $\nu$ is a $\Phi$-pullback on $F$. Then there exists a unique $\Psi\in \Omega^p(\inf X,\gamma)$ such that 
\begin{align}
\label{eq:conana:a}
t_{*}(\nu_0) &= \alpha_l^{*}\fa\Psi-\alpha_r^{*}\fa\Phi
\\
\label{eq:conana:7}
\rho_l^{*}\nu_0&= \pr_{\mor{\inf X}}^{*}\fb\Psi+\pr_F^{*}\nu_0
\\
\label{eq:conana:6}
\nu_1 &= \alpha_l^{*}\fc\Psi-\alpha_r^{*}\fc\Phi
\text{.}
\end{align}
\end{lemma}

\begin{proof}
Uniqueness is clear because the pullbacks of $\fa\Psi$, $\fb\Psi$ and $\fc\Psi$ along surjective submersions ($\alpha_l$ and $\pr_{\mor{\inf X}}$) are fixed. For existence, consider $\fa\psi := t_{*}(\nu_0) + \alpha_r^{*}\fa\Phi \in \Omega^p(F,\mathfrak{g})$. Over $F \times_{\ob{\inf X}} F$ we have an identity $\pr_1^{*}\fa\psi = \pr_2^{*}\fa\psi$.
\begin{comment}
We have the diffeomorphism $F \lli{\alpha_r}\times_{t} \mor{\inf Y} \to  F \times_{\ob{\inf X}} F$ and can write $(f_1,f_2)=(f_1,\rho_r(f_1,\eta))$, where $\alpha_r(f_1)=t(\eta)$ and $\alpha_r(f_2)=\alpha_r(\rho_r(f_1,\eta))=s(\eta)$.
\begin{eqnarray*}
(t_{*}\nu_0+ \alpha_r^{*}\fa\Phi)_{f_2}-(t_{*}\nu_0+ \alpha_r^{*}\fa\Phi)_{f_1} 
&=& t_{*}\nu_0|_{f_2} + \fa\Phi_{\alpha_r(f_2)}-t_{*}\nu_0|_{f_1} - \fa\Phi_{\alpha_r^{*}(f_1)}
\\&=& t_{*}\nu_0|_{\rho_r(f_1,\eta)} + \fa\Phi_{s(\eta)}-t_{*}\nu_0|_{f_1} - \fa\Phi_{t(\eta)}
\\&=& t_{*}(\fb\Phi_{\eta}) + t_{*}\nu_0|_{f}  + \fa\Phi_{s(\eta)}-t_{*}\nu_0|_{f_1} - \fa\Phi_{t(\eta)}
 \\&=&  0\text{.}
\end{eqnarray*}
\end{comment}
Since  differential forms form a sheaf, there exists  $\fa\Psi \in \Omega^p(\ob{\inf X},\mathfrak{g})$ such that $\alpha_l^{*}\fa\Psi = \fa\psi$. This shows \cref{eq:conana:a}. 
Next we consider an open subset $U \subset \ob{\inf X}$ with a section $\sigma:U \to F$ against $\alpha_l$. On $V:= s^{-1}(U)\subset \mor{\inf X}$ we have a smooth map
\begin{equation*}
\tilde\sigma: V \to \mor{\inf X} \ttimes{s}{\alpha_l} F: v \mapsto (v,\sigma(s(v)))\text{.}
\end{equation*}
We define $\fb\psi := \tilde\sigma^{*}(\rho_l^{*}\nu-\pr_{F}^{*}\nu_0)\in\Omega^{p}(V,\mathfrak{h})$. \Cref{eq:conana:cr}  implies that $\fb\psi$ is independent of the choice of $\sigma$.
\begin{comment}
If $\tau: U \to F$ is another section, we obtain a smooth map $\gamma:U \to \mor{\inf Y}$ such that $t(\gamma(u))=\alpha_r(\sigma(u))$ and $\rho_r(\sigma(u), \gamma(u))=\tau(u)$. Then
\begin{eqnarray*}
\fb\psi_{\tau}|_{v} 
&=& \nu_0|_{\rho_l(v,\tau(s(v)))}-\nu_0|_{\tau(s(v))}
\\ &=& \nu_0|_{\rho_l(v,\rho_r(\sigma(s(v)), \gamma(s(v))))}-\nu_0|_{\rho_r(\sigma(s(v)), \\ &=& \nu_0|_{\rho_r(\rho_l(v,\sigma(s(v))),\gamma(s(v)))}-\nu_0|_{\rho_r(\sigma(s(v)), \gamma(s(v)))}
\\ &=& \nu_0|_{\rho_l(v,\sigma(s(v)))}+\fb\Phi_{\gamma(s(v))}-\nu_0|_{\sigma(s(v))}-\fb\Phi_{\gamma(s(v))}
\\&=& \nu_0|_{\rho_l(v,\sigma(s(v)))}-\nu_0|_{\sigma(s(v))}
\\&=& \fb\psi_{\sigma}|_{v}
\end{eqnarray*}
\end{comment} 
Hence we obtain $\fb\Psi\in\Omega^p(\mor{\inf X},\mathfrak{h})$ such that $\fb\Psi|_V=\fb\psi$. It is straightforward to check the conditions $\Delta\fa\Psi = t_{*}(\fb\Psi)$ and $\Delta\fb\Psi = 0$.
 The definition of $\fb\psi$ implies \cref{eq:conana:7}.
 \begin{comment}
For $\chi\in \mor{\inf X}$ and $f_s\in F$ such that $\alpha_l(f_s)=s(\chi)$ we set $f_t := \rho_l(\chi,f_s)$ and obtain $\alpha_l(f_t)=t(\chi)$ and $\alpha_r(f_t)=\alpha_r(\rho_l(\chi,f_s))=\alpha_r(f_s)$ and hence
\begin{eqnarray*}
(\Delta\fa\Psi)_{\chi}&=&\fa\Psi_{t(\chi)} -
\fa\Psi_{s(\chi)}
\\&=&\fa\psi_{f_t} - \fa\psi_{f_s}
\\&=& t_{*}(\nu_0|_{f_t}) + \fa\Phi_{\alpha_r(f_t)} - t_{*}(\nu_0|_{f_s}) - \fa\Phi_{\alpha_r(f_s)} \\&=& t_{*}(\nu_0|_{\rho_l(\chi,f_s)}) - t_{*}(\nu_0|_{f_s})
\\&=& t_{*}(\fb\Psi_{\chi})\text{.}
\end{eqnarray*}
For composable $\chi_1,\chi_2\in \mor{\inf X}$ choose $f_2\in F$ such that $\alpha_l(f_2)=s(\chi_2)$. Set $f_1 := \rho_l(\chi_2,f_2)$. Then we have $\alpha_l(f_1)=t(\chi_2)=s(\chi_1)$, and thus
\begin{eqnarray*}
(\Delta\fb\Psi)_{\chi_1,\chi_2} &=& \fb\Psi_{\chi_1}-\fb\Psi_{\chi_1\circ \chi_2}+\fb\Psi_{\chi_2}
\\&=&  \nu_0|_{\rho_l(\chi_1,f_1)}-\nu_0|_{f_1}-\nu_0|_{\rho_l(\chi_1\circ \chi_2,f_2)}+\nu_0|_{f_2}+\nu_0|_{\rho_l(\chi_2,f_2)}-\nu_0|_{f_2}
\\&=& 0\text{.}
\end{eqnarray*}
\end{comment}
Finally, we consider $\fc\psi := \alpha_r^{*}\fc\Phi + \nu_1 \in \Omega^{p+1}(F,\mathfrak{h})$. \Cref{eq:conana:cs} implies over $F \times_{\ob{\inf X}} F$  an identity $\pr_1^{*}\fc\psi = \pr_2^{*}\fc\psi$.
\begin{comment}
We have the diffeomorphism $F \ttimes {\alpha_r}{t} \mor{\inf Y} \to  F \times_{\ob{\inf X}} F$ and can write $(f_1,f_2)=(f_1,\rho_r(f_1,\eta))$, where $\alpha_r(f_1)=t(\eta)$ and $\alpha_r(f_2)=\alpha_r(\rho_r(f_1,\eta))=s(\eta)$.
\begin{eqnarray*}
\fc\psi_{f_2}-\fc\psi_{f_1} &=& \fc\Phi_{\alpha_r(f_2)} + \nu_1|_{f_2}-\fc\Phi_{\alpha_r(f_1)} - \nu_1|_{f_1}
\\&=& \fc\Phi_{s(\eta)} + \nu_1|_{\rho_r(f_1,\eta)}-\fc\Phi_{t(\eta)} - \nu_1|_{f_1}
\\&=& 0
\end{eqnarray*}
\end{comment}
Thus, there exists  $\fc\Psi\in\Omega^{p+1}(\ob{\inf X},\mathfrak{h})$ such that $\alpha_l^{*}\fc\Psi=\fc\psi$. This shows \cref{eq:conana:6}.
\end{proof}

We write $F_{\nu}^{*}\Phi := \Psi$ for the unique $p$-form of \cref{lem:pullback}. 
\begin{comment}
If $\nu$ is changed to $(\nu_0+\alpha_l^{*}\kappa_0,\nu_1+\alpha_l^{*}\kappa_1)$, then $\Psi$ changes to a form with components $\fa\Psi + t_{*}(\kappa_0)$, $\fb\Psi + \Delta\kappa_0$ and $\fc\Psi + \kappa_1$.
\end{comment}
The following two lemmata are straightforward to deduce from the definitions. The first describes the compatibility of the pullback with the Lie algebra structure on differential forms.

\begin{lemma}
\label{lem:pullback:rules}
Let $F:\inf X \to \inf Y$ be an anafunctor. We consider differential forms $\Phi\in \Omega^k(\inf Y,\gamma)$ and $\Phi'\in\Omega^l(\inf Y,\gamma)$, a $\Phi$-pullback  $\nu$ on $F$, a $\Phi'$-pullback $\nu'$ on $F$, and $s\in \R$. 
\begin{enumerate}[(a)]

\item
$\nu+\nu' := (\nu_0+\nu_0',\nu_1+\nu_1')$ is a $(\Phi + \Phi')$-pullback on $F$, and
$F^{*}_{\nu+\nu'}(\Phi + \Phi') = F^{*}_{\nu}\Phi + F^{*}_{\nu'}\Phi'$.

\item
$s\nu := (s\nu_0,s\nu_1)$ is a $(s\Phi)$-pullback on $F$, and $F^{*}_{s\nu}(s\Phi)=sF^{*}_{\nu}\Phi$.

\item
$\mathrm{D}\nu := (\mathrm{d}\nu_0- (-1)^{k}\nu_1,\mathrm{d}\nu_1)$ is a $\mathrm{D}\Phi$-pullback on $F$, and
$\mathrm{D}(F_{\nu}^{*}\Phi) = F_{\mathrm{D}\nu}^{*}(\mathrm{D}\Phi)$.

\item
$[\nu\wedge \nu']$ defined by 
\begin{align*}
[\nu\wedge \nu']_0 &:= [\nu_0\wedge \nu'_0]-(-1)^{kl}\alpha_{*}(\alpha_r^{*}\fa{\Phi'}\wedge \nu_0)+\alpha_{*}(\alpha_r^{*}\fa\Phi \wedge\nu'_0)
\\
[\nu\wedge \nu']_1 &:= [\nu_0 \wedge \nu'_1 ]+ (-1)^{l} [\nu_1  \wedge \nu'_0]+[\nu_0 \wedge \alpha_r^{*} \fc{\Phi'}] + (-1)^{l} [\alpha_r^{*}\fc\Phi \wedge \nu'_0]\\&\hspace{5cm}+\alpha_{*}(\alpha_r^{*} \fa\Phi\wedge \nu'_1 )- (-1)^{kl} \alpha_{*}(\alpha_r^{*} \fa{\Phi'}\wedge \nu_1 )
\end{align*}
is a $[\Phi \wedge \Phi']$-pullback on $F$, and $F^{*}_{[\nu\wedge \nu']}([\Phi \wedge \Phi']) = [F^{*}_{\nu}\Phi \wedge F_{\nu'}^{*}\Phi']$. 
\end{enumerate}
\end{lemma}

The second lemma describes the compatibility of the pullback with the structure of anafunctors: composition, inversion, and transformations (see \cref{re:strucana}).   

\begin{lemma}
\begin{enumerate}[(a)]
\item 
\label{lem:pullback:comp}
Suppose $F:\inf X \to \inf Y$ and $G:\inf Y \to \mathcal{Z}$ are anafunctors. If $\Phi\in\Omega^k(\mathcal{Z},\gamma)$, $\nu$ is a $\Phi$-pullback on $G$, and $\nu'$ is a $(G_{\nu}^{*}\Phi)$-pullback on $F$, then the forms  $\tilde\nu_0:= \pr_F^{*}\nu_0'+ \pr_G^{*}\nu_0$ and $\tilde\nu_1:= \pr_F^{*}\nu_1'+ \pr_G^{*}\nu_1$ on $F \ttimes {\alpha_r}{\alpha_l} G$ descend to the total space $(F \ttimes {\alpha_r}{\alpha_l} G)/\inf Y$ of $G \circ F$, and $\nu'\circ \nu :=(\tilde\nu_0,\tilde\nu_1)$ is a $\Phi$-pullback on $G \circ F$ such that
\begin{equation*}
(G \circ F)_{\nu'\circ \nu}^{*}\Phi=F^{*}_{\nu'}(G^{*}_{\nu}\Phi)\text{.}
\end{equation*} 

\item
\label{lem:pullback:inv}
Suppose an anafunctor $F: \inf X \to \inf Y$ is a weak equivalence. If $\Phi\in \Omega^k(\inf Y,\gamma)$, $\nu$ is a $\Phi$-pullback on $F$, and $\Psi := F_{\nu}^{*}\Phi$, then $-\nu := (-\nu_0,-\nu_1)$ is a $\Psi$-pullback on the inverse anafunctor $F^{-1}$, and $(F^{-1}_{-\nu})^{*}\Psi = \Phi$. 
\begin{comment}
The following relations hold:
\begin{align}
\label{eq:pullinv:a}
\nu_0|_{f \circ \eta}&= \nu_0|_f+ \fb\Phi_{\eta}
\\
\label{eq:pullinv:b}
\nu_1|_{f \circ \eta} &= \nu_1|_f + \Delta\fc\Phi_{\eta}\text{.}
\\
\label{eq:pullinv:c}
t_{*}(\nu_0|_f) &= \fa\Psi_{\alpha_l(f)}-\fa\Phi_{\alpha_r(f)}
\\
\label{eq:pullinv:d}
\nu_0|_{\chi\circ f}&= \fb\Psi_{\chi}+\nu_0|_f
\\
\label{eq:pullinv:e}
\nu_1|_f &= \fc\Psi_{\alpha_l(f)}-\fc\Phi_{\alpha_r(f)}
\text{.}
\end{align}
In order to show that $-\nu$ is a $\Psi$-pullback on $F^{-1}$, we show:
\begin{align*}
-\nu_0|_{\chi^{-1} \circ f}&\eqcref{eq:pullinv:d} -\fb\Psi_{\chi^{-1}}-\nu_0|_f= -\nu_0|_f+ \fb\Psi_{\chi}
\\
-\nu_1|_{\chi^{-1}\circ f} &\eqcref{eq:pullinv:e}-\fc\Psi_{s(\chi)}-\fc\Phi_{\alpha_r(f)}=-\fc\Psi_{\alpha_l(f)}+\fc\Phi_{\alpha_r(f)}+\fc\Psi_{t(\chi)}-\fc\Psi_{s(\chi)}\eqcref{eq:pullinv:e}- \nu_1|_f + \Delta\fc\Psi_{\chi}\text{.}
\end{align*}
In order to show that $\Phi=(F^{-1}_{-\nu})^{*}\Psi$ we show:
\begin{align*}
t_{*}(-\nu_0|_f) &\eqcref{eq:pullinv:c} \fa\Phi_{\alpha_r(f)}-\fa\Psi_{\alpha_l(f)}
\\
-\nu_0|_{f\circ \eta^{-1}}&\eqcref{eq:pullinv:a} \fb\Phi_{\eta}-\nu_0|_f
\\
-\nu_1|_f &\eqcref{eq:pullinv:e} \fc\Phi_{\alpha_r(f)}-\fc\Psi_{\alpha_l(f)}
\end{align*}
\end{comment}

\item
Suppose $f: F \Rightarrow G$ is a transformation between anafunctors $F,G: \inf X \to \inf Y$. If  $\Phi\in\Omega^k(\inf Y,\gamma)$ and $\nu$ is a $\Phi$-pullback on $G$, then $f^{*}\nu :=(f^{*}\nu_0,f^{*}\nu_1)$ is a $\Phi$-pullback on $F$ such that $F^{*}_{f^{*}\nu}\Phi=G^{*}_{\nu}\Phi$.

\end{enumerate}
\end{lemma}

\begin{remark}
\label{re:canpullback}
Let $\mathcal{X},\mathcal{Y}$ be Lie groupoids and $\Phi\in \Omega^k(\inf Y,\gamma)$.
\begin{enumerate}[(a)]

\item 
\label{re:canpullback:a}
If $\phi:\inf X \to \inf Y$ is a smooth functor, then the  anafunctor $F$ associated to $\phi$ has a canonical $\Phi$-pullback $\nu$ such that $F^{*}_{\nu}\Phi=\phi^{*}\Phi$. More concretely, recall that the total space of $F$ is $F=\ob{\inf X} \ttimes \phi t \mor{\inf Y}$. Then, the canonical $\Phi$-pullback $\nu$ is given by  $\nu_0 := \pr_{2}^{*}\fb\Phi$ and $\nu_1 :=- \pr_{2}^{*}s^{*}\fc\Phi + \pr_{1}^{*}\phi^{*}\fc\Phi$. \begin{comment}
Indeed, 
\begin{eqnarray*}
\nu_{\rho_r((x,\eta),\eta')} &=& \nu_{(x,\eta\circ \eta')}=\fb\Phi_{\eta\circ\eta'} = \fb\Phi_{\eta} + \fb\Phi_{\eta'} = \nu_{(x,\eta)} + \fb\Phi_{\eta'}
\\ \mu_{\rho_r((x,\eta),\eta')}&=&\mu_{(x,\eta\circ\eta')}=-\fc\Phi_{s(\eta')}+\fc\Phi_{\phi(x)}=-\fc\Phi_{s(\eta)}+\fc\Phi_{t(\eta')}-\fc\Phi_{s(\eta')}+\fc\Phi_{\phi(x)}=\mu_{(x,\eta)}+\Delta\fc\Phi_{\eta'} \text{.}
\end{eqnarray*}
Indeed,
\begin{equation*}
\fa\psi_{(x,\eta)} = t_{*}(\nu_{(x,\eta)}) + \fa\Phi_{\alpha_r(x,\eta)} = t_{*}(\fb\Phi_{\eta})+\fa\Phi_{s(\eta)}=\fa\Phi_{t(\eta)}=\fa\Phi_{\phi(x)}=\fa\Phi_{\phi(\alpha_l(x,\eta))}
\end{equation*}
and thus $\fa\Psi=\phi^{*}\fa\Phi$. Further, we have a canonical section $\sigma:\inf X_0 \to F:x \mapsto (x,\id_{\phi(x)})$, and hence
\begin{equation*}
\fb\psi_{\chi} = \nu_{\rho_l(\chi,(s(\chi),\id_{s(\chi)}))}-\nu_{(s(\chi),\id_{s(\chi)})} =\nu_{(t(\chi),\phi(\chi))}-\nu_{(s(\chi),\id_{s(\chi)})}=\fb\Phi_{\phi(\chi)}\text{.}
\end{equation*}
This means $\fb\Psi = \phi^{*}\fb\Phi$. Finally, we have
\begin{equation*}
\fc\psi_{x,\eta} = \fc\Phi_{s(\eta)} - \fc\Phi_{s(\eta)} +\fc\Phi_{\phi(x)}=\fc\Phi_{\phi(x)}
\end{equation*}
Then we have $\fc\Psi = \phi^{*}\fc\Phi$.
\end{comment}

\item
\label{re:canpullback:b}
Using the affine space structure of \cref{lem:pullexists}, the canonical $\Phi$-pullback $\nu$ on $F$ can be shifted by $\kappa=(\kappa_0,\kappa_1)$, where $\kappa_0\in \Omega^k(\inf X_0,\mathfrak{h})$ and $\kappa_1\in\Omega^{k+1}(\inf X_0,\mathfrak{h})$, namely to  $\nu^\kappa :=(\nu_0 + \pr_1^{*}\kappa_0,\nu_1+\pr_1^{*}\kappa_1)$.
\begin{comment}
This is a $\Phi$-pullback:
\begin{align*}
\nu_0^{\kappa}|_{(x,\eta)\circ \eta'}&= \nu_0^{\kappa}|_{(x,\eta\circ\eta')}= \nu_0|_{(x,\eta\circ\eta')}+\kappa_0|_x=\nu_0|_{(x,\eta)}+\kappa_0|_{x}+ \fb\Phi_{\eta'}= \nu_0^{\kappa}|_{(x,\eta)}+ \fb\Phi_{\eta'}
\\
\nu_1^{\kappa}|_{(x,\eta)\circ\eta'} &= \nu_1^{\kappa}|_{(x,\eta\circ\eta')}=\nu_1|_{(x,\eta\circ\eta')}+\kappa_1|_{x} = \nu_1|_{(x,\eta)}+\kappa_1|_x+\Delta\fc\Phi|_{\eta'} =\nu_1^{\kappa}|_{(x,\eta)}+\Delta\fc\Phi_{\eta'}\text{.}
\end{align*}
\end{comment}
For the shifted pullback $\nu^{\kappa}$ we find 
\begin{equation}
\label{eq:shiftedpullback}
\fa {(F_{\nu^{\kappa}}^{*}\Phi)} = \phi^{*}\fa\Phi + t_{*}(\kappa_0)
\quomma
\fb {(F_{\nu^{\kappa}}^{*}\Phi)} = \phi^{*}\fb\Phi + \Delta\kappa_0
\quomma
\fc {(F_{\nu^{\kappa}}^{*}\Phi)} = \phi^{*}\fc\Phi + \kappa_1\text{.}
\end{equation}
\begin{comment}
Indeed, 
\begin{align*}
t_{*}(\nu_0^{\kappa}|_{(x,\eta)})=t_{*}(\fb\Phi_{\eta}+\kappa_0|_x) &= \fa\Phi_{\phi(x)} + t_{*}(\kappa_0|_{x})-\fa\Phi_{s(\eta)}
\\
\nu_0^{\kappa}|_{\chi \circ (x,\eta)}=\nu_0|_{t(\chi),\phi(\chi) \circ \eta} +\kappa_0|_{t(\chi)}&= \fb\Phi_{\phi(\chi)\circ\eta} +\kappa_0|_{t(\chi)}\\&=\fb\Phi_{\phi(\chi)}+\fb\Phi_{\eta}+\kappa_0|_{t(\chi)}-\kappa_0|_{s(\chi)}+\kappa_0|_{s(\chi)}\\&=\fb{\Phi}|_{\phi(\chi)}+\Delta\kappa_0|_{\chi}+\nu_0|_{(x,\eta)}+\kappa_0|_x=\fb{\Phi}|_{\phi(\chi)}+\Delta\kappa_0|_{\chi}+\nu_0^{\kappa}|_{(x,\eta)}
\\
\nu_1^{\kappa}|_{(x,\eta)}=\nu_1|_{(x,\eta)}+\kappa_1|_x &= \fc\Phi_{\phi(x)} + \kappa_1|_x-\fc\Phi_{s(\eta)}
\text{.}
\end{align*}
\end{comment}

\item
\label{re:canpullback:c}
If $\eta:\phi\Rightarrow \phi'$ is a smooth natural transformation, then the induced transformation $f:F \Rightarrow F'$ between the corresponding anafunctors
in general does not preserve the canonical $\Phi$-pullbacks $\nu$ and $\nu'$, i.e., $f^{*}\nu' \neq \nu$. 
\begin{comment}
Indeed, 
\begin{align*}
f^{*}\nu_0'|_{(x,\beta)} -\nu_0|_{(x,\beta)}&=\nu_0'|_{f(x,\beta)}  -\nu_0|_{(x,\beta)}\\
&=\nu_0'|_{(x,\eta(x) \circ \beta)}-\nu_0|_{(x,\beta)}\\
&= \fb\Phi_{\eta(x)\circ\beta}-\fb\Phi_{\beta}\\
&=\fb\Phi_{\eta(x)} \end{align*}
and
\begin{align*}
f^{*}\nu_1'|_{(x,\beta)}-\nu_1|_{(x,\beta)}&=\nu_1'|_{f(x,\beta)}-\nu_1|_{(x,\beta)}\\
&=\nu_1'|_{x,\eta(x) \circ  \beta}-\nu_1|_{(x,\beta)}\\
&=-\fc\Phi|_{s(\beta)} + \fc\Phi|_{\phi'(x)}+\fc\Phi|_{s(\beta)}-\fc\Phi|_{\phi(x)}
\\
&= \fc\Phi|_{\phi'(x)}- \fc\Phi|_{\phi(x)}
\\&= \Delta\fc\Phi|_{\eta(x)}
\end{align*}
\end{comment}
However, if $\kappa$ and $\kappa'$ are shifts for the canonical $\Phi$-pullbacks, then $f$ satisfies $f^{*}(\nu'^{\kappa'})=\nu^{\kappa}$ if and only if $\eta^{*}\fb\Phi=\kappa_0 - \kappa_0'$ and $\eta^{*}\Delta\fc\Phi=\kappa_1-\kappa_1'$.
\begin{comment}
Indeed,
\begin{align*}
f^{*}\nu_0'|_{(x,\beta)} &=\nu_0'|_{f(x,\beta)} \\& =\nu_0'|_{(x,\eta(x) \circ \beta)}\\&= \fb\Phi_{\eta(x)\circ\beta}+ \kappa_0'|_x\\&=\fb\Phi_{\eta(x)}+\fb\Phi_{\beta} + \kappa_0'|_x\\&= \fb\Phi_{\beta}+\kappa_0|_x\\&= \nu_0|_{(x,\beta)}
\\
f^{*}\nu_1'|_{(x,\beta)} &=\nu_1'|_{f(x,\beta)}  \\&=\nu_1'|_{(x,\eta(x) \circ \beta)}\\&= -\fc\Phi_{s(\beta)} + \fc\Phi|_{\phi'(x)}+ \kappa_1'|_x
\\&= - \fc\Phi_{s(\beta)}+\fc\Phi_{\phi(x)}+ \Delta\fc\Phi|_{\eta(x)}+\kappa_1'|_x
\\&=- \fc\Phi_{s(\beta)}+\fc\Phi_{\phi(x)}+\kappa_1|_x\\&= \nu_1|_{(x,\beta)}
\end{align*}
\end{comment}

\item
\label{re:canpullback:d}
Suppose $F':\mathcal{X} \to \mathcal{Y}$ is an anafunctor
equipped with a $\Phi$-pullback $\nu'=(\nu_0',\nu_1')$ and $f: F \Rightarrow F'$ is the transformation induced by a smooth map $\tilde f: \ob{\mathcal{X}} \to F'$ via \cref{rem:smoothfunctor:b}. Then, we have $f^{*}\nu' = \nu^{\kappa}$, where the shift $\kappa=(\kappa_0,\kappa_1)$ is defined by $\kappa_0 := \tilde f^{*}\nu_0'$ and $\kappa_1 := \tilde f^{*}\nu_1'$. 
\begin{comment}
Indeed,
\begin{align*}
\eta^{*}\nu_0'|_{x,\beta} &= \nu_0'|_{\eta(x,\id_{\phi(x)} \circ \beta)}
\\&= \nu_0'|_{\eta(x,\id_{\phi(x)})\circ \beta}
\\&= \nu_0'|_{\tilde\eta(x)} + \fb\Phi_{\beta}
\\&=\fb\Phi_{\beta} + \kappa_0|_x
\\
\eta^{*}\nu_1'|_{x,\beta} &= \nu_1'|_{\eta(x,\id_{\phi(x)} \circ \beta)}
\\&= \nu_1'|_{\eta(x,\id_{\phi(x)})\circ \beta}
\\&= \nu_1'|_{\tilde\eta(x)} + \Delta\fc\Phi_{\beta}
\\&=-\fc\Phi_{s(\beta)}+\fc\Phi_{t(\beta)} + \kappa_1|_x
\\&=-\fc\Phi_{s(\beta)}+\fc\Phi_{\phi(x)} + \kappa_1|_x
\end{align*}
\end{comment}
\end{enumerate}
\end{remark}

\section{Connections on principal 2-bundles}

\label{sec:conn}

\setsecnumdepth{2}

In this section, $M$ is a smooth manifold and $\Gamma$ is a Lie 2-group, with associated crossed module $(G,H,t,\alpha)$ and Lie 2-algebra $\gamma$.

\subsection{Connections and curvature}

Let $\inf P$ be a principal $\Gamma$-2-bundle over $M$. 

\begin{definition}
\label{def:connection}
A \emph{connection} on $\inf P$ is a $\gamma$-valued 1-form $\Omega \in \Omega^1(\inf P,\gamma)$ such that
\begin{equation*}
R^{*}\Omega = \mathrm{Ad}_{\pr_\Gamma}^{-1}(\pr_\inf P^{*}\Omega) + \pr_\Gamma^{*}\Theta
\end{equation*}
over $\inf P \times \Gamma$, where $\pr_{\inf P}$ and $\pr_{\Gamma}$ are the projections to the two factors, and $\Theta$ is the Maurer-Cartan form on $\Gamma$.
\end{definition}

Let us spell out explicitly all structure and conditions that are packed into \cref{def:connection}. A connection $\Omega$ consists of the following components:
\begin{equation*}
\fa\Omega\in\Omega^1(\ob{\inf P},\mathfrak{g})
\quomma 
\fb\Omega\in\Omega^1(\mor{\inf P},\mathfrak{h})
\quand
\fc\Omega\in\Omega^2(\ob{\inf P},\mathfrak{h})\text{.}
\end{equation*}
These satisfy the following conditions:
\begin{align}
R^{*}\fa\Omega &= \mathrm{Ad}_{g_0}^{-1}(p_0^{*}\fa\Omega) + g_0^{*}\theta &&\text{over }\ob{\inf P}\times \ob{\Gamma} 
\label{eq:conform:a}
\\
R^{*}\fb\Omega 
&=(\alpha_{g_1^{-1}})_{*}\left (\mathrm{Ad}_{h}^{-1}(p_1^{*}\fb\Omega)+(\tilde \alpha_{h})_{*}(p_1^{*}s^{*}\fa\Omega) +h^{*}\theta \right)\hspace{-1em} &&\text{over }\mor{\inf P}\times \mor{\Gamma}
\label{eq:conform:b}
\\
\label{eq:conform:c}
R^{*}\fc\Omega &= (\alpha_{g_0^{-1}})_{*}(p_0^{*}\fc\Omega)
&&\text{over }\ob{\inf P}\times \ob{\Gamma}
\text{.}\end{align}
Here we have used the maps  $p_0,g_0$ defined on $\ob{\inf P} \times \ob{\Gamma}$ by $p_0(p,g):=p$ and $g_0(p,g):=g$, as well as the maps $h,p_1,g_1$ defined on $\mor{\inf P} \times \mor{\Gamma}$ by $h(\rho,(h,g)):= h$, $p_1(\rho,(h,g)):=\rho$, and $g_1(\rho,(h,g)):= g$.

\begin{remark}
For $g\in G$ we have a functor $R_g:\inf P \to \inf P: p \mapsto R(p,g)$ and \cref{def:connection} implies  
$R_g^{*}\Omega = \mathrm{Ad}_g^{-1}(\Omega)$. \begin{comment}
We spell this out:
\begin{align*}
R_g^{*}\fa\Omega &= \mathrm{Ad}_{g}^{-1}(\fa\Omega) 
\\
R_{\id_g}^{*}\fb\Omega 
&=(\alpha_{g^{-1}})_{*} (\fb\Omega )
\\
R_g^{*}\fc\Omega &= (\alpha_{g^{-1}})_{*}(\fc\Omega)
\text{.}
\end{align*}
\end{comment}
Similarly, for $p\in \ob{\inf P}$ we have a functor $R_p:\Gamma \to \inf P: g \mapsto R(p,g)$, and we obtain $R_{p}^{*}\Omega=\Theta$. \begin{comment}
We spell this out:
\begin{align*}
R_p^{*}\fa\Omega &=  \theta
\\
R_{\id_p}^{*}\fb\Omega 
&=(\alpha_{g_1^{-1}})_{*}\left (h^{*}\theta \right)
\\
R_p^{*}\fc\Omega &= 0
\text{.}
\end{align*}
\end{comment}
The two equations
\begin{equation}
\label{eq:connection:sep}
R_g^{*}\Omega = \mathrm{Ad}_g^{-1}(\Omega)
\quand
R_{p}^{*}\Omega=\Theta
\end{equation}
are well-known for connections in ordinary principal bundles, where they are in fact equivalent to the combined equation analogous to the one in \cref{def:connection}. For connections on principal 2-bundles, however, \cref{eq:connection:sep} is not equivalent to \cref{def:connection} (consider, e.g., $\Gamma=B\ueins$). \begin{comment}
For $\Gamma=B\ueins$ \cref{def:connection} says $\fb\Omega_{R(\rho,(h,1))} 
&=&\fb\Omega_{\rho} +h^{*}\theta$; see \cref{ex:bueinsconn}. We have $\fa\Omega=0$ and $\fc\Omega$ is arbitrary. On the other hand, the condition $R_g^{*}\Omega = \mathrm{Ad}_g^{-1}(\Omega)$ is empty, while $R_{p}^{*}\Omega=\Theta$ gives $\fb\Omega_{R(\id_p,(h,1))} 
&=h^{*}\theta$. These are certainly not equivalent. \end{comment} 
\end{remark}

\begin{definition}
The 2-form
\begin{equation*}
\mathrm{curv}(\Omega) := \mathrm{D}\Omega + \frac{1}{2}[\Omega\wedge \Omega] \in \Omega^2(\inf P,\gamma)
\end{equation*}
is called the curvature of $\Omega$.
\end{definition}

Thus, the curvature has the following components:
\begin{align}
\label{eq:curv:a}
\fa{\mathrm{curv}(\Omega)} &= \mathrm{d}\fa\Omega +\frac{1}{2}[\fa\Omega \wedge \fa\Omega] +t_{*}(\fc\Omega)
\\
\label{eq:curv:b}
\fb{\mathrm{curv}(\Omega)} &= \Delta\fc\Omega +\mathrm{d}\fb\Omega+\frac{1}{2}[\fb\Omega\wedge \fb\Omega]+ \alpha_{*}(s^{*}\fa\Omega \wedge \fb\Omega)
\\
\label{eq:curv:c}
\fc{\mathrm{curv}(\Omega)} &=  \mathrm{d}\fc\Omega+\alpha_{*}(\fa\Omega \wedge \fc\Omega)\text{.}
\end{align}
The following statement can be deduced abstractly from the definitions,  \cref{lem:bracket,lem:adjoint,eq:mc}. 

\begin{theorem}
\label{th:curv}
\begin{enumerate}[(a)]
\item 
The curvature 
satisfies a Bianchi identity:
$
\mathrm{D}\mathrm{curv}(\Omega) = [\mathrm{curv}(\Omega) \wedge \Omega]$.

\item
\label{th:curv:b}
The curvature  is $R$-invariant:
$R^{*}\mathrm{curv}(\Omega) = \mathrm{Ad}_{\pr_{\Gamma}}^{-1}(\pr_{\inf P}^{*}\mathrm{curv}(\Omega))$. 

\end{enumerate}
\end{theorem}

\begin{comment}
\begin{proof}
The first is proved as follows:
\begin{equation*}
\mathrm{D}\mathrm{curv}(\Psi) = \mathrm{D}\mathrm{D}\Psi + \frac{1}{2}\mathrm{D}[\Psi \wedge \Psi]
\eqcref{eq:comm,eq:leibnitz} [\mathrm{D}\Psi \wedge \Psi]
= [\mathrm{curv}(\Psi) \wedge \Psi]\text{,}
\end{equation*}
since $[[\Psi \wedge \Psi] \wedge \Psi] = 0$ by \cref{eq:jacobi}.  The second follows from \cref{def:connection,eq:adjder,eq:mc}.
\end{proof}
\end{comment}

\begin{definition}
A connection $\Omega$ on a principal $\Gamma$-2-bundle $\inf P$ is called:
\begin{enumerate}[(a)]

\item
\emph{flat} if $\mathrm{curv}(\Omega)=0$. 

\item
\emph{fake-flat} if  $\fa{\mathrm{curv}(\Omega)}=0$ and $\fb{\mathrm{curv}(\Omega)}=0$.

\item 
\emph{regular}, if every point $(p,q)\in \ob{\inf P} \times_M \ob{\inf P}$ has an open neighborhood $U$ supporting a \splitting\ $\sigma$ with $\sigma^{*}\fb{\mathrm{curv}(\Omega)}=0$. 
\end{enumerate}
\end{definition}

\begin{remark}
It follows directly from the definition of a connection that connections form a convex subset of $\Omega^1(\inf P,\gamma)$.
\begin{comment}
Indeed, let $\Omega_1$ and $\Omega_2$ be connections on $\inf P$, and $t\in [0,1]$. Consider $\tilde\Omega = (1-t)\Omega_1 + t\Omega_2$. That $\tilde\Omega$ is a connection follows directly via \cref{def:connection} from the linearity of pullback and adjoint action.
Indeed,
\begin{align*}
R^{*}\tilde\Omega &=(1-t)R^{*}\Omega_1 + tR^{*}\Omega_2 \\&=(1-t)(\mathrm{Ad}_{\pr_\Gamma}^{-1}(\pr_\inf P^{*} \Omega_1) + \pr_\Gamma^{*}\Theta) + t(\mathrm{Ad}_{\pr_\Gamma}^{-1}(\pr_\inf P^{*}\Omega_2) + \pr_\Gamma^{*}\Theta)\\&= \mathrm{Ad}_{\pr_\Gamma}^{-1}(\pr_\inf P^{*}\tilde \Omega) + \pr_\Gamma^{*}\Theta
\end{align*}
\end{comment}
However, regular, fake-flat, or flat connections do in general not form a convex subset.
\begin{comment}
The curvature is non-linear: we find
\begin{equation*}
\mathrm{curv}(\tilde\Omega)= (1-t)\mathrm{curv}(\Omega_1)+ t \mathrm{curv}(\Omega_2)+ \frac{1}{2}(t^2-t) \left ( [\Omega_1+\Omega_2 \wedge \Omega_1+\Omega_2 ]\text{.}
\end{equation*}
Indeed, 
\begin{align*}
\mathrm{curv}(\tilde\Omega) &= \mathrm{D}\tilde\Omega + \frac{1}{2}[\tilde\Omega\wedge \tilde\Omega]
\\&= \mathrm{D}((1-t)\Omega_1 + t\Omega_2) + \frac{1}{2}[((1-t)\Omega_1 + t\Omega_2)\wedge ((1-t)\Omega_1 + t\Omega_2)]
\\&= (1-t)\mathrm{D}\Omega_1 + t\mathrm{D}\Omega_2 + \frac{1}{2}(1-t)^2[\Omega_1 \wedge \Omega_1 ]+ \frac{1}{2}t^2[\Omega_2\wedge  \Omega_2]+ (1-t)t[\Omega_1 \wedge  \Omega_2]
\\&= (1-t)(\mathrm{D}\Omega_1 +\frac{1}{2}[\Omega_1\wedge \Omega_1])+ t (\mathrm{D}\Omega_2 +\frac{1}{2}[\Omega_2\wedge \Omega_2])
\\&\qquad+ \frac{1}{2}(-t+t^2)[\Omega_1 \wedge \Omega_1 ]+ \frac{1}{2}(t^2-t)[\Omega_2\wedge  \Omega_2]+ (1-t)t[\Omega_1 \wedge  \Omega_2]
\\&= (1-t)\mathrm{curv}(\Omega_1)+ t \mathrm{curv}(\Omega_2)+ \frac{1}{2}(t^2-t) \left ( [\Omega_1 \wedge \Omega_1 ]+[\Omega_2\wedge  \Omega_2]-2[\Omega_1 \wedge  \Omega_2]\right )
\\&= (1-t)\mathrm{curv}(\Omega_1)+ t \mathrm{curv}(\Omega_2)+ \frac{1}{2}(t^2-t) \left ( [\Omega_1+\Omega_2 \wedge \Omega_1+\Omega_2 ]
\end{align*}
\end{comment}
\end{remark}

\begin{remark}
\label{re:regfake}
Obviously, flat implies fake-flat and fake-flat implies regular. Conversely, if $\Omega$ is regular and $\fa{\mathrm{curv}(\Omega)}=0$, then $\Omega$ is fake-flat. Indeed, the vanishing of $\fa{\mathrm{curv}(\Omega)}$ implies via \cref{th:curv:b}
\begin{equation}
\label{eq:re:regfake:1}
R^{*}\fb{\mathrm{curv}(\Omega)} = (\alpha_{g_1^{-1}})_{*} (\mathrm{Ad}_{h}^{-1}(\fb{\pr_{p_1}^{*}\mathrm{curv}(\Omega)}) )\text{.} 
\end{equation}
\begin{comment}
Indeed,
\begin{eqnarray*}
R^{*}\fb{\mathrm{curv}(\Omega)} &=& \mathrm{Ad}_{\pr_{\Gamma}}^{-1}(\pr_{\inf P}^{*}\mathrm{curv}(\Omega))
\\&=& (\alpha_{g_1^{-1}})_{*}\left (\mathrm{Ad}_{h}^{-1}(\fb{\pr_{\inf P}^{*}\mathrm{curv}(\Omega)})+(\tilde\alpha_{h})_{*}(s^{*}\fa{\pr_{\inf P}^{*}\mathrm{curv}(\Omega)})\right ) 
\\&=& (\alpha_{g_1^{-1}})_{*}\left (\mathrm{Ad}_{h}^{-1}(\fb{\pr_{\inf P}^{*}\mathrm{curv}(\Omega)})\right )
\end{eqnarray*}
\end{comment}
\begin{comment}
This means, explicitly, for $\rho\in\mor{\inf P}$, $(h,g)\in \mor{\Gamma}$ and $\tilde v_1,\tilde v_2\in T_{\rho}\mor{\inf P}$, $w_1,w_2\in T_hH$, $u_1,u_2\in T_gG$:
\begin{multline*}
\fb{\mathrm{curv}(\Omega)}|_{R(\rho,(h,g))}(T_{\rho,(h,g)}R(\tilde v_1,(w_1,u_1)),T_{\rho,(h,g)}R(\tilde v_2,(w_2,u_2))) \\= (\alpha_{g^{-1}})_{*} (\mathrm{Ad}_{h}^{-1}(\fb{\mathrm{curv}(\Omega)}|_{\rho}(\tilde v_1,\tilde v_2)) )\text{.}
\end{multline*}
\end{comment}
For $\rho\in \mor{\inf P}$ consider the splitting of $T_{\rho}\mor{\inf P}$ of \cref{re:splitting} with respect to a \splitting\ $\sigma$ of $\inf P$ with $\sigma^{*}\fb{\mathrm{curv}(\Omega)}=0$. \Cref{eq:re:regfake:1} implies that $\fb{\mathrm{curv}(\Omega)}$ vanishes when applied to at least one \quot{vertical} tangent vector, i.e. one in the image of $T\mor{\Gamma}$.
\begin{comment}
Let $w\in T_hH$ and $u\in T_gG$. Then, 
\begin{equation*}
T_{(h,g)}R_{\rho}(w,u) =T_{\rho,(h,g)}R(0,(w,u))\text{.}
\end{equation*}
Indeed, if $w$ and $u$ are represented by smooth curves $\omega$ and $\gamma$, respectively, then
\begin{equation*}
T_{(h,g)}R_{\rho}(\ddt 1t0 (\omega,\gamma))=\ddt 1t0 R_{\rho}(\omega,\gamma)=\ddt 1t0 R(\rho,(\omega,\gamma))=T_{\rho,(h,g)}R(\ddt 1t0 (\rho,(\omega,\gamma)))\text{.}
\end{equation*}
We conclude that $\fb{\mathrm{curv}(\Omega)}|_{R(\rho,(h,g))}(\tilde v_1,\tilde v_2)$ is zero whenever  $\tilde v_1$ or $\tilde v_2$ is of the form $T_{(h,g)}R_{\rho}(w,u)$. \end{comment}
\Cref{eq:re:regfake:1} implies further that
\begin{equation}
\sigma^{*}R_{h,g}^{*}\fb{\mathrm{curv}(\Omega)} = (\alpha_{g^{-1}})_{*} (\mathrm{Ad}_{h}^{-1}(\sigma^{*}\fb{\mathrm{curv}(\Omega)}) )=0\text{,} 
\end{equation}
where $R_{h,g}: \mor{\inf P} \to \mor{\inf P}$ is the action with fixed $(h,g)\in \mor{\Gamma}$. 
Hence, regularity of $\Omega$ implies that $\fb{\mathrm{curv}(\Omega)}$ vanishes when applied to two \quot{horizontal} tangent vectors. In summary, the splitting into horizontal and vertical parts gives $\fb{\mathrm{curv}(\Omega)}=0$.
\begin{comment}
Decomposing another tangent vector $\tilde v'\in T_{\rho}\mor{\inf P}$ in the same way, and inserting into $\fb{\mathrm{curv}(\Omega)}$, we have
\begin{align*}
\fb{\mathrm{curv}(\Omega)}|_{\rho}(\tilde v_1,\tilde v_2) 
\\&= \fb{\mathrm{curv}(\Omega)}|_{R(\sigma(u),(h,g(u)))}(T_{u} R_{h,g(u)}(v_1,v_2),T_{u} R_{h,g(u)}(v_1',v_2'))
\\&= (R_{h,g(u)}^{*}\fb{\mathrm{curv}(\Omega)})|_{u}((v_1,v_2),(v_1',v_2'))
\\&=0
\end{align*}
\end{comment}
\end{remark}

\begin{example}
\label{ex:trivialconn}
The \emph{trivial principal $\Gamma$-2-bundle} over $M$ is defined by  $\inf I := \idmorph{M} \times \Gamma$, with the action by multiplication on the second factor. Using the formula of \cref{ex:mcmult} for the Maurer-Cartan form $\Theta$ one can check that $\Omega := \pr_\Gamma^{*}\Theta$ is a  connection on $\inf I$; it is flat due to \cref{eq:mc}. More generally, for  $\Psi\in\Omega^1(\idmorph M,\gamma)$ 
\begin{equation*}
\Omega_{\Psi} := \mathrm{Ad}^{-1}_{\pr_\Gamma}(\pr_M^{*}\Psi) + \pr_\Gamma^{*}\Theta\text{.} 
\end{equation*}
is a connection on $\inf I$. 
\begin{comment}
Explicitly, we have
\begin{align*}
\fa{\Omega}_{\Psi} &= \mathrm{Ad}_{g_0}^{-1}(\fa\Psi) + g_0^{*}\theta 
\\
\fb{\Omega}_{\Psi} &= (\alpha_{g_1^{-1}})_{*} ( (l_{h}^{-1} \circ \alpha_{h})_{*}(s^{*}\fa\Psi)+h^{*}\theta) 
\\
\fc{\Omega}_{\Psi} &= (\alpha_{g_0^{-1}})_{*}(\fc\Psi) 
\end{align*}
\end{comment}
With \cref{lem:bracket,lem:adjoint,eq:mc} one can check that
\begin{equation}
\label{eq:ex:trivbun:1}
\mathrm{curv}(\Omega_{\Psi}) = \mathrm{Ad}_{\pr_\Gamma}^{-1}(\pr_M^{*}(\mathrm{D}\Psi+\frac{1}{2}[\Psi \wedge \Psi]))\text{.} 
\end{equation}
\begin{comment}
Indeed
\begin{align*}
\mathrm{curv}(\Omega_{\Psi}) 
&= \mathrm{D}\Omega_{\Psi} + \frac{1}{2}[\Omega_{\Psi}\wedge \Omega_{\Psi}] \\&= \mathrm{D}(\mathrm{Ad}^{-1}_{\pr_\Gamma}(\pr_M^{*}\Psi) + \pr_\Gamma^{*}\Theta) + \frac{1}{2}[(\mathrm{Ad}^{-1}_{\pr_\Gamma}(\pr_M^{*}\Psi) + \pr_\Gamma^{*}\Theta)\wedge (\mathrm{Ad}^{-1}_{\pr_\Gamma}(\pr_M^{*}\Psi) + \pr_\Gamma^{*}\Theta)]  
\\&= \mathrm{D}\mathrm{Ad}^{-1}_{\pr_\Gamma}(\pr_M^{*}\Psi) + \pr_\Gamma^{*}\mathrm{D}\Theta + \frac{1}{2}[\mathrm{Ad}^{-1}_{\pr_\Gamma}(\pr_M^{*}\Psi) \wedge \mathrm{Ad}^{-1}_{\pr_\Gamma}(\pr_M^{*}\Psi) ] \\&\qquad+ [\mathrm{Ad}^{-1}_{\pr_\Gamma}(\pr_M^{*}\Psi) \wedge  \pr_\Gamma^{*}\Theta]  + \frac{1}{2}[ \pr_\Gamma^{*}\Theta\wedge  \pr_\Gamma^{*}\Theta]   
\\&= \mathrm{Ad}_{\pr_\Gamma}^{-1}(\mathrm{D}\pr_M^{*}\Psi)-[ \pr_\Gamma^{*}\Theta \wedge \mathrm{Ad}_{\pr_\Gamma}^{-1}(\pr_M^{*}\Psi)] + \frac{1}{2}[\mathrm{Ad}^{-1}_{\pr_\Gamma}(\pr_M^{*}\Psi) \wedge \mathrm{Ad}^{-1}_{\pr_\Gamma}(\pr_M^{*}\Psi)] \\&\qquad+ [\mathrm{Ad}^{-1}_{\pr_\Gamma}(\pr_M^{*}\Psi) \wedge  \pr_\Gamma^{*}\Theta] 
\\&= \mathrm{Ad}_{\pr_\Gamma}^{-1}(\pr_M^{*}(\mathrm{D}\Psi+\frac{1}{2}[\Psi \wedge \Psi]))
\end{align*}
\end{comment}
This has the following implications:
\begin{itemize}

\item 
$\Omega_{\Psi}$ is  regular. To see this,  we get from \cref{eq:ex:trivbun:1} the formula
\begin{equation*}
\fb{\mathrm{curv}(\Omega_{\Psi})} = (\alpha_{g_1^{-1}})_{*}\left ((\tilde\alpha_{h})_{*}(\pr_M^{*}(\mathrm{D}\Psi+\frac{1}{2}[\Psi \wedge \Psi])\fa{})\right)\text{.}
\end{equation*}
Consider the global \splitting\ $\sigma((x,g_1),(x,g_2)) := (\id_x,g_1)$. Since $h \circ \sigma=1$ we have $(\tilde\alpha_{h \circ \sigma})_{*}=0$ and thus $\sigma^{*}\fb{\mathrm{curv}(\Omega_{\Psi})}=0$. 

\item
$\Omega_{\Psi}$ is fake-flat if and only if $(\mathrm{D}\Psi+\frac{1}{2}[\Psi \wedge \Psi])\fa{}=0$. 
\begin{comment}
To see this, we get the formula
\begin{equation*}
\fa{\mathrm{curv}(\Omega_{\Psi})} =\mathrm{Ad}_{g_0}^{-1}(\pr_M^{*}(\mathrm{D}\Psi+\frac{1}{2}[\Psi \wedge \Psi])\fa{})\text{.}
\end{equation*}
Now, the \quot{if}-part is clear. 

\end{comment}

\item
$\Omega_{\Psi}$ is flat if and only if $\mathrm{D}\Psi+\frac{1}{2}[\Psi \wedge \Psi]=0$.

\end{itemize} 
Finally, it is easy to check that
the assignment
$\Psi \mapsto \Omega_{\Psi}$
establishes a bijection between $\Omega^1(\idmorph{M},\gamma)$ and the set of connections on $\inf I$.
\begin{comment}
Indeed, this is a connection: on $\idmorph M \times \Gamma \times \Gamma$ with $R=\id \times m$ we have
\begin{multline*}
R^{*}\Omega_{\Psi} = \mathrm{Ad}^{-1}_{\pr_2\cdot \pr_3}(\pr_1^{*}\Psi)+ (\pr_2\cdot \pr_3)^{*}\Theta=  \mathrm{Ad}_{\pr_3}^{-1}(\mathrm{Ad}^{-1}_{\pr_2}(\pr_1^{*}\Psi)) +\mathrm{Ad}_{\pr_3}^{-1}( \pr_2^{*}\Theta) + \pr_3^{*}\Theta \\= \mathrm{Ad}_{\pr_3}^{-1}(\pr_{12}^{*}\Omega_{\Psi}) + \pr_3^{*}\Theta\text{.}
\end{multline*}
Converse, if $\Omega$ is a connection on $\inf I$, consider $\Psi := s^{*}\Omega$, where $s := (\id,1):\idmorph M \to \idmorph M \times \Gamma$ is the canonical section. Then, $\Omega = \Omega_{\Psi}$. 
Indeed,
\begin{equation*}
\Omega_{\Psi} = \mathrm{Ad}^{-1}_{\pr_\Gamma}(\pr_M^{*}s^{*}\Omega) + \pr_\Gamma^{*}\Theta=\mathrm{Ad}^{-1}_{\pr_\Gamma}(p_M^{*}\Omega) + \pr_\Gamma^{*}\Theta=\Omega\text{.}
\end{equation*}
Here, $p_M: \idmorph{M} \times \Gamma \to \idmorph{M} \times \Gamma$ is $p_M(x,g):=(x,1)$. Note that $\id_{\idmorph{M} \times \Gamma} = R(p_M,\pr_{\Gamma})$.
We have, since $\fb\Psi=0$,
\begin{align*}
\fa{\Omega}_{\Psi} &= \mathrm{Ad}_{g_0}^{-1}(\fa\Psi) + g_0^{*}\theta 
\\
\fb{\Omega}_{\Psi} &= (\alpha_{g_1^{-1}})_{*} ( (l_{h}^{-1} \circ \alpha_{h})_{*}(s^{*}\fa\Psi)+h^{*}\theta) 
\\
\fc{\Omega}_{\Psi} &= (\alpha_{g_0^{-1}})_{*}(\fc\Psi) 
\end{align*}
\end{comment}
In particular, every connection on $\inf I$ is regular.
\end{example}

\begin{example}
Let $\Gamma = \idmorph{G}$, and $\Omega$ be a connection on a principal $\Gamma$-2-bundle. By \cref{ex:forms:codis} we have $\fb\Omega=\fc\Omega=0$, and the only condition for $\Omega$ is $R^{*}\fa\Omega = \mathrm{Ad}_{g_0}^{-1}(p_0^{*}\fa\Omega) + g_0^{*}\theta$. The curvature components are $\fa{\mathrm{curv}(\Omega)} = \mathrm{d}\fa\Omega +\frac{1}{2}[\fa\Omega \wedge \fa\Omega]$, while $\fb{\mathrm{curv}(\Omega)}=0$ and $\fc{\mathrm{curv}(\Omega)}=0$. Thus, every connection is automatically regular, and flat is equivalent to fake-flat. 
\end{example}

\begin{example}
\label{ex:bueinsconn}
Let $\Gamma = B\ueins$, and $\Omega$ be a connection on a principal $\Gamma$-2-bundle $\mathcal{P}$. By \cref{ex:forms:codisecht} we have $\fa\Omega=0$,  $\fb\Omega\in\Omega^{p}(\inf P_{1})$ subject to the condition $\Delta\fb\Omega=0$,
and $\fc\Omega\in\Omega^{p+1}(\inf P_{0})$. \Cref{def:connection} implies only that $R^{*}\fb\Omega
=p_1^{*}\fb\Omega +h^{*}\theta$. 
The non-trivial curvature-components are
\begin{equation*}
\fb{\mathrm{curv}(\Omega)} = \Delta\fc\Omega +\mathrm{d}\fb\Omega
\quand
\fc{\mathrm{curv}(\Omega)} =  \mathrm{d}\fc\Omega\text{.}
\end{equation*}
Thus, with \cref{re:regfake}, regular is equivalent to fake-flat. 
\end{example}

\begin{example}
\label{ex:gbunred}
Consider an ordinary principal $G$-bundle $P$ over $M$. Let $\act PH$ be the action groupoid for the right $H$-action on $P$ induced along $t:H \to G$. That is,  $\ob{\act PH}=P$ and $\mor{\act PH}=P \times H$,  source and target maps are $s(p,h)= p$ and $t(p,h)= pt(h)$, and composition is $(p_2,h_2) \circ (p_1,h_1) = (p_1,h_1h_2)$.  We define a functor
\begin{equation*}
R: \act PH \times \Gamma \to \act PH
\end{equation*}
by $R(p,g):=pg$ and $R((p,h),(h',g)) := (pg,\alpha(g^{-1},hh'))$. It is straightforward to check that this is a $\Gamma$-action.
\begin{comment}
For functoriality, we check
\begin{align*}
R((p_2,h_2),(h_2',g_2))\circ R((p_1,h_1),(h_1',g_1)) &=(p_2g_2,\alpha(g_2^{-1},h_2h_2'))\circ (p_1g_1,\alpha(g_1^{-1},h_1h'_1))
\\&= (p_1g_1,\alpha(g_1^{-1},h_1h'_1)\alpha(g_2^{-1},h_2h_2'))
\\&= (p_1g_1,\alpha(g_1^{-1},h_1h_2h'_2h_1'))
\\&= R((p_1,h_1h_2),(h_2'h_1',g_1))
\\&= R((p_2,h_2)\circ (p_1,h_1),(h_2',g_2)\circ (h_1',g_1))
\end{align*}
In the middle we have used that $g_2=t(h_1')g_1$.
We check
\begin{align*}
((p,h)\cdot(h_1,g_1))\cdot(h_2,g_2) &= (pg_1,\alpha(g_1^{-1},hh_1))\cdot (h_2,g_2)
\\&= (pg_1g_2,\alpha(g_2^{-1},\alpha(g_1^{-1},hh_1)h_2)
\\&= (pg_1g_2,\alpha(g_2^{-1}g_1^{-1},hh_1\alpha(g_1,h_2))
\\&= (p,h) \cdot (h_1\alpha(g_1,h_2),g_1g_2)\text{.}
\end{align*}
Let $U\subset M$ and supporting a section $s: U \to P$. 
\end{comment}
This construction is compatible with pullbacks, and is functorial in the sense that bundle morphisms $P \to P'$ induce smooth $\Gamma$-equivariant functors $\act PH \to \act {P'}H$. 
That $\act PH$ is a principal $\Gamma$-2-bundle can now be seen from the criteria of \cite[Prop. 6.2.8]{Nikolaus}: if $U\subset M$ is open and supports a section into $P$, we obtain the required 1-morphism 
\begin{equation*}
\inf I \cong \act{(U\times G)}H \cong \act{P|_U}H\cong (\act PH)|_U\text{.}
\end{equation*}
Next we assume that $\omega\in\Omega^1(P,\mathfrak{g})$ is a connection on $P$. 
We put 
\begin{equation*}
\fa\Omega := \omega
\quomma
\fb\Omega := (\tilde\alpha_{\pr_H})_{*}(\pr_P^{*}\omega)+\pr_H^{*}\theta \quand
\fc\Omega := 0\text{.}
\end{equation*}
It is straightforward to check that this is a connection on $\act PH$, 
\begin{comment}
Indeed, \cref{eq:conform:a} is satisfied because $\omega$ is a connection on $P$, and \cref{eq:conform:c} is clear. For \cref{eq:conform:b,} we have to show
\begin{align*}
(R^{*}\fb\Omega)_{(p,h)\cdot (h',g)} &= \fb\Omega_{(pg,\alpha(g^{-1},hh'))} \\&= (\tilde\alpha_{\alpha(g^{-1},hh')})_{*}(\omega_{pg})+\theta_{\alpha(g^{-1},hh')} 
\\&=  (\alpha_{g^{-1}})_{*} \left ( (\tilde\alpha_{hh'})_{*}( \mathrm{Ad}_g(\omega_{pg})+\theta_{g^{-1}})+\theta_{hh'}\right ) 
\\&=  (\alpha_{g^{-1}})_{*} \left ( \mathrm{   Ad}_{h'}^{-1}  (\tilde\alpha_{h} )_{*}(\omega_p  )+(\tilde\alpha_{h'})_{*}(\omega_p  
)+\mathrm{Ad}_{h'}^{-1}(\theta_{h})+\theta_{h'}\right ) 
\\&=(\alpha_{g^{-1}})_{*}\left (\mathrm{Ad}_{h'}^{-1}((\tilde\alpha_{h})_{*}(\omega_p)+\theta_h)+(\tilde \alpha_{h'})_{*}(\omega_p) +h'^{*}\theta \right)
\\&=(\alpha_{g^{-1}})_{*}\left (\mathrm{Ad}_{h'}^{-1}(\fb\Omega_{(p,h)})+(\tilde \alpha_{h'})_{*}(\fa\Omega_{p}) +h'^{*}\theta \right)
\end{align*}
Here we have used
\begin{align*}
(\tilde\alpha_{h_1h_2})_{*} &=  \mathrm{   Ad}_{h_2}^{-1} \circ (\tilde\alpha_{h_1} )_{*}+(\tilde\alpha_{h_2})_{*}
\\
(\tilde\alpha_{\alpha(g,h)})_{*} &=  (\alpha_g)_{*} \circ (\tilde\alpha_{h})_{*}\circ \mathrm{Ad}_g^{-1}
\\
\theta_{\alpha(g,h)} &= (\alpha_g)_{*}((l_h^{-1} \circ \alpha_h)_{*}(\theta_g) + \theta_h)
\end{align*}
\end{comment}
and to compute its curvature via \cref{eq:curv:a,eq:curv:b,eq:curv:c}:
\begin{align*}
\fa{\mathrm{curv}(\Omega)} = \psi
\quomma
\fb{\mathrm{curv}(\Omega)} = (\tilde\alpha_{\pr_H})_{*}(\pr_P^{*}\psi)
\quand
\fc{\mathrm{curv}(\Omega)} = 0\text{,}
\end{align*}
where $\psi := \mathrm{d}\omega+\frac{1}{2}[\omega\wedge\omega]\in\Omega^2(P,\mathfrak{g})$ is the ordinary curvature of the connection $\omega$.
\begin{comment}
Everything is obvious except for:
\begin{align*}
\fb{\mathrm{curv}(\Omega)}_{p,h} &=\mathrm{d}\fb\Omega_{p,h}+\frac{1}{2}[\fb\Omega_{p,h}\wedge \fb\Omega_{p,h}]+ \alpha_{*}(\omega_p \wedge \fb\Omega_{p,h}) 
\\&=  \mathrm{d}((\tilde\alpha_h)_{*}(\omega_p)+\theta_h)+\frac{1}{2}[((\tilde\alpha_h)_{*}(\omega_p)+\theta_h)\wedge ((\tilde\alpha_h)_{*}(\omega_p)+\theta_h)]
\\&\qquad+ \alpha_{*}(\omega_p \wedge ((\tilde\alpha_h)_{*}(\omega_p)+\theta_h))
\\&=  \mathrm{d}(\tilde\alpha_h)_{*}(\omega_p)+\mathrm{d}\theta_h+\frac{1}{2}[(\tilde\alpha_h)_{*}(\omega_p)\wedge (\tilde\alpha_h)_{*}(\omega_p)]+\frac{1}{2}[\theta_h\wedge \theta_h]+[(\tilde\alpha_h)_{*}(\omega_p)\wedge \theta_h]
\\&\qquad+ \alpha_{*}(\omega_p \wedge (\tilde\alpha_h)_{*}(\omega_p))+ \alpha_{*}(\omega_p \wedge \theta_h)
\\&=  (\tilde\alpha_h)_{*}(\mathrm{d}\omega_p)  +\frac{1}{2}[(\tilde\alpha_h)_{*}(\omega_p)\wedge (\tilde\alpha_h)_{*}(\omega_p)]+ \alpha_{*}(\omega_p \wedge (\tilde\alpha_h)_{*}(\omega_p))
\\&=  (\tilde\alpha_h)_{*}(\mathrm{d}\omega_p)  +\frac{1}{2}(\tilde\alpha_h)_{*}([\omega_p \wedge\omega_p])
\end{align*}
\end{comment} 
The connection $\Omega$ is always regular: consider over $P\times_MP$ the (global) transition span with $\sigma(p,p'):=(p,1)$ and $g:P \times_M P \to G$ defined such that $p'g(p,p')=p$. 
\begin{comment}
This is a transition span:
\begin{equation*}
t(\sigma(p,p'))=p
\quand
s(\sigma(p,p'))=p=R(p',g(p,p'))\text{.}
\end{equation*}
\end{comment}
Obviously, we have $\sigma^{*}\fb{\mathrm{curv}(\Omega)}=0$. Further, and even more obviously, we have
\begin{equation*}
\Omega\text{ is flat} 
\quad\gdw\quad
\Omega\text{ is fake-flat}
\quad\gdw\quad
\omega\text{ is flat.}
\end{equation*}
\end{example}

\subsection{Connection-preserving morphisms}

We recall from \cref{sec:pullbackana} that pulling back a $\gamma$-valued differential form along an anafunctor requires the choice of a pullback (see \cref{def:pullback}). 
\begin{definition}
If $\inf P_1$ and $\inf P_2$ are principal $\Gamma$-bundles over $M$ equipped with connections $\Omega_1$ and $\Omega_2$, respectively, then an $\Omega_2$-pullback $\nu$ on a 1-morphism $F: \inf P_1 \to \inf P_2$ is  called \emph{connection-preserving}, if $\Omega_1 = F_{\nu}^{*}\Omega_2$. If $F': \inf P_1 \to \inf P_2$ is a second 1-morphism equipped with an $\Omega_2$-pullback $\nu'$, then a 2-morphism $f:F \Rightarrow F'$ is called \emph{connection-preserving}, if $f^{*}\nu'=\nu$.
\end{definition}

Next we introduce two further conditions for an $\Omega_2$-pullback on $F$. Firstly, we require compatibility with the $\Gamma$-equivariance of $F$.  

\begin{definition}
Let $F: \inf P_1 \to \inf P_2$ be a $\Gamma$-equivariant anafunctor and let $\Omega_2\in \Omega^1(\inf P_2,\gamma)$ be a 1-form. An $\Omega_2$-pullback $\nu=(\nu_0,\nu_1)$ on $F$  is called \emph{connective},  if over $F \times \mor{\Gamma}$ we have
\begin{align}
\label{eq:conana:b}
\rho^{*}\nu_0 &= (\alpha_{g^{-1}})_{*}\left (\mathrm{Ad}_{h}^{-1}(\pr_F^{*}\nu_0)+(\tilde\alpha_{h})_{*}(\pr_F^{*}\alpha_r^{*}\fa{\Omega}_2) +h^{*}\theta \right) 
\\
\label{eq:conana:mu}
\rho^{*}\nu_1 &=  (\alpha_{g^{-1}})_{*} \big ( \mathrm{Ad}_h^{-1} (\pr_F^{*}\nu_1)+(\tilde \alpha_h)_{*}(t_{*} (\pr_F^{*}\alpha_r^{*}\fc\Omega_2)) \big )
 \text{,}
\end{align}
where $\rho: F \times \mor{\Gamma} \to F$ is the $\mor{\Gamma}$-action,
and the maps $g,h,\pr_F$ are defined by $g(f,(h',g')):= g'$, $h(f,(h',g')):= h'$ and $\pr_F(f,(h',g')):= f$. 
\end{definition}

\begin{proposition}
\label{prop:exconnpull}
Connective $\Omega_2$-pullbacks form a convex subset in the affine space of all $\Omega_2$-pullbacks. 
If the $G$-action on $\ob{\inf P_1}$ is free and proper, then connective $\Omega_2$-pullbacks exist.  
\end{proposition}

\begin{proof}
Convexity follows immediately.
\begin{comment}
Indeed,
\begin{align*}
\rho^{*}((1-t)\nu_0+t\nu_0')&=(1-t)(\alpha_{g^{-1}})_{*}\left (\mathrm{Ad}_{h}^{-1}(\pr_F^{*}\nu_0)+(\tilde\alpha_{h})_{*}(\pr_F^{*}\alpha_r^{*}\fa{\Omega}_2) +h^{*}\theta \right) \\ &\qquad+t(\alpha_{g^{-1}})_{*}\left (\mathrm{Ad}_{h}^{-1}(\pr_F^{*}\nu_0')+(\tilde\alpha_{h})_{*}(\pr_F^{*}\alpha_r^{*}\fa{\Omega}_2) +h^{*}\theta \right) 
\\&=(\alpha_{g^{-1}})_{*}\left (\mathrm{Ad}_{h}^{-1}(\pr_F^{*}((1-t)\nu_0+t\nu_0'))+(\tilde\alpha_{h})_{*}(\pr_F^{*}\alpha_r^{*}\fa{\Omega}_2) +h^{*}\theta \right) 
\end{align*}
\end{comment}
By \cref{lem:pullexists} there exists an $\Omega_2$-pullback $\nu$, but in general it will not be connective. We can write any element $(h,g)\in \mor{\Gamma}$ as $(h,g)=(h,t(h)^{-1})\cdot (1,t(h)g)$. The action of elements of the form of the first factor can be rewritten as
\begin{equation*}
\rho(f,(h,t(h)^{-1})) =  f \circ R(\id_{\alpha_r(f)},(h,t(h)^{-1}))\text{,}
\end{equation*}
using the compatibility between the $\mor{\Gamma}$-action and the right $\inf P_2$-action. From the conditions on $\Omega_2$-pullbacks and the fact that $\Omega_2$ is a connection, one can deduce that $\nu$ satisfies \cref{eq:conana:b,eq:conana:mu} for group elements of the form $(h,t(h)^{-1})$. \begin{comment}
Indeed,
\begin{align*}
\nu_0|_{\rho(f,(h,t(h)^{-1}))}&=\nu_0|_{f \circ R(\id_{\alpha_r(f)},(h,t(h)^{-1}))}
\\&=\nu_0|_f + \fb\Omega_2|_{ R(\id_{\alpha_r(f)},(h,t(h)^{-1}))}
\\&=\nu_0|_f + (\alpha_{t(h)})_{*}\left ((\tilde \alpha_{h})_{*}(\fa\Omega_2|_{\alpha_r(f)}) +h^{*}\theta \right)
\\&=(\alpha_{t(h)})_{*}\left (\mathrm{Ad}_h^{-1}(\nu_0|_f) + (\tilde \alpha_{h})_{*}(\fa\Omega_2|_{\alpha_r(f)}) +h^{*}\theta \right)
\end{align*}
and
\begin{align*}
\nu_1|_{\rho(f,(h,t(h)^{-1}))}&=\nu_1|_{f \circ R(\id_{\alpha_r(f)},(h,t(h)^{-1}))}
\\&=\nu_1|_f + (\Delta\fc\Omega_2)|_{ R(\id_{\alpha_r(f)},(h,t(h)^{-1}))}
\\&=\nu_1|_f + \fc\Omega_2|_{ \alpha_r(f)}-\fc\Omega_2|_{ R(\alpha_r(f),t(h)^{-1})}
\\&=\nu_1|_f + \fc\Omega_2|_{ \alpha_r(f)}-(\alpha_{t(h)})_{*}(\fc\Omega_2|_{ \alpha_r(f)})
\\&=\nu_1|_f +(\alpha_{t(h)})_{*} (\mathrm{Ad}_h^{-1}(\fc\Omega_2|_{ \alpha_r(f)})-\fc\Omega_2|_{ \alpha_r(f)})
\\&=(\alpha_{t(h)})_{*}\left (\mathrm{Ad}_h^{-1}(\nu_0|_f) + (\tilde \alpha_{h})_{*}(t_{*} (\fc\Omega_2|_{\alpha_r(f)})) \right)
\end{align*}
\end{comment}
Thus, we can reduce the problem of finding connective $\Omega_2$-pullbacks to the problem of finding an $\Omega_2$-pullback that satisfies \cref{eq:conana:b,eq:conana:mu} with respect to the $G$-action induced along the identity map $\id_{\Gamma}:G \to \mor{\Gamma}$ of $\Gamma$. We first treat $\nu_0$.
Consider the 1-form
\begin{equation*}
\varphi := (\id_{F} \times \id_{\Gamma})\rho^{*}\nu_0 - (\alpha_{\pr_G^{-1}})_{*}(\pr_F^{*}\nu_0) \in \Omega^1(F \times G,\mathfrak{h})
\end{equation*}
which measures the failure in \cref{eq:conana:b}.
\begin{comment}
In point-wise notation, this is
\begin{align*}
\varphi_{f,g} := \nu_{\rho(f,(1,g))} - (\alpha_{g^{-1}})_{*}(\nu_f) \in \Omega^1(F \times G,\mathfrak{h})
\end{align*}
\end{comment}
The 1-form $\varphi$  satisfies a cocycle condition over $F \times G \times G$, namely
\begin{equation}
\label{eq:cocyclephi}
(\id_F \times m)^{*}\varphi = (\rho \times \id_G)^{*}\varphi + \pr_{12}^{*}\varphi\text{.}
\end{equation}
\begin{comment}
In the point-wise notation, this is 
\begin{equation*}
\varphi_{f,gg'}=\varphi_{\rho(f,g),g'}+\varphi_{f,g}\text{.}
\end{equation*}
\end{comment}
That $\Omega_2$ is a connection implies that $\varphi$ descends along  $\alpha_l \times \id: F \times G \to \ob{\inf P_1} \times G$. 
\begin{comment}
We check
\begin{align*}
\varphi_{f \circ \rho,g}&=\nu_{\rho(f\circ \rho,(1,g))} - (\alpha_{g^{-1}})_{*}(\nu_{f\circ \rho}) 
\\&=\nu_{\rho(f,(1,g)) \circ R(\rho,(1,g))} - (\alpha_{g^{-1}})_{*}(\nu_{f}+\fb\Omega_2|_{\rho})
\\&=\nu_{\rho(f,(1,g))}+\fb\Omega_2|_{R(\rho,(1,g))} - (\alpha_{g^{-1}})_{*}(\nu_{f})- (\alpha_{g^{-1}})_{*}(\fb\Omega_2|_{\rho})
\\&=  \nu_{\rho(f,(1,g))} - (\alpha_{g^{-1}})_{*}(\nu_f)
\\&= \varphi_{f,g}
\end{align*}
\end{comment}
Thus, we have a unique $\psi\in\Omega^1(\ob{\inf P_1}\times G,\mathfrak{h})$ such that $(\alpha_l \times \id)^{*}\psi=\varphi$. The cocycle condition \cref{eq:cocyclephi} implies a analogous condition for $\psi$ over $\ob{\inf P_1}\times G \times G$. Our assumptions for the $G$-action guarantee that $\ob{\inf P_1}$ is a principal $G$-bundle over the quotient $X:= \ob{\inf P_1}/G$. We obtain diffeomorphisms $\ob{\inf P_1} \times G^{k} \cong \ob{\inf P_1}^{[k]}_X$, under which $\psi$ becomes a cocycle in the complex
\begin{equation*}
\Omega^1(X) \to \Omega^1(\ob{\inf P_1}) \to \Omega^1(\ob{\inf P_1}_X^{[2]})\to \Omega^1(\ob{\inf P_1}_X^{[3]})\to...
\end{equation*}
This complex is exact \cite{murray}, so that there exists $\kappa_0\in\Omega^1(\ob{\inf P_1})$ such that $\psi = R^{*}\kappa_0-\pr_{\ob{\inf P_1}}^{*}\kappa_0$. Treating $\nu_1$ the very same way, we obtain $\kappa_1\in\Omega^2(\ob{\inf P_1})$.
It is then straightforward to check that the shifted pullback $\nu' := \nu + \alpha_l^{*}(\kappa_0,\kappa_1)$ is an $\Omega_2$-pullback and satisfies  \cref{eq:conana:b,eq:conana:mu} with respect to the induced $G$-action. 
\begin{comment}
Indeed,
\begin{align*}
\rho^{*}(\nu_i+\alpha_l^{*}\kappa_i)-\pr_F^{*}(\nu_i+\alpha_l^{*}\kappa_i)&=\rho^{*}\nu_i -\pr_F^{*}\nu_i+ (\alpha_l \times t)^{*}(R^{*}\kappa_i-\pr_{\ob{\inf P_1}}^{*}\kappa_i)
\\&= \varphi + (\alpha_l \times t)^{*}\psi
\\&=0\text{.}
\end{align*}
\end{comment}
\end{proof}

Secondly, we introduce a notion of fake-curvature for pullbacks of 1-forms, and then impose flatness conditions.

\begin{definition}
\label{def:curvpullback}
Let $F: \inf P_1 \to \inf P_2$ be an anafunctor, let $\Omega_2\in \Omega^1(\inf P_2,\gamma)$ be a 1-form, and let $\nu=(\nu_0,\nu_1)$ be an $\Omega_2$-pullback on $F$. 
\begin{enumerate}[(a)]

\item 
The 2-form
\begin{equation*}
\mathrm{fcurv}(\nu):=\mathrm{d}\nu_0+ \frac{1}{2}[\nu_0\wedge \nu_0]+ \alpha_{*}(\alpha_r^{*}\fa\Omega_2 \wedge \nu_0)+\nu_1 \in \Omega^2(F,\mathfrak{h})
\end{equation*}
is called the \emph{fake-curvature} of $\nu$, and $\nu$ is called \emph{fake-flat}, if $\mathrm{fcurv}(\nu)=0$. 
\item
$\nu$ is called \emph{regular}, if every point $(p_1,p_2)\in \ob{\inf P_1} \times_M \ob{\inf P_2}$ has an open neighborhood $U$ supporting a \splitting\ $\sigma$ of $F$ with $\sigma^{*}\mathrm{fcurv}(\nu)=0$.

\item
The 3-form
\begin{equation*}
\mathrm{curv}(\nu) := \mathrm{d}\nu_1+ [\nu_0 \wedge \nu_1]+\alpha_{*}(\alpha_r^{*} \fa\Omega_2\wedge \nu_1 ) \in \Omega^3(F,\mathfrak{h})
\end{equation*}
is called the \emph{curvature} of $\nu$, and $\nu$ is called \emph{flat}, if $\mathrm{fcurv}(\nu)=0$ and $\mathrm{curv}(\nu)=0$.

\end{enumerate}
\end{definition}

\begin{remark}
\label{re:curvpull}
Fake-curvature and curvature in \cref{def:curvpullback} satisfy
\begin{equation*}
\mathrm{D}\nu+\frac{1}{2}[\nu\wedge\nu] = (\mathrm{fcurv}(\nu),\mathrm{curv}(\nu))
\end{equation*}
with respect to the operations on pullbacks introduced in \cref{lem:pullback:rules}. In particular, $\nu' := (\mathrm{fcurv}(\nu),\mathrm{curv}(\nu))$ is a $\mathrm{curv}(\Omega_2)$-pullback on $F$, and we have $\mathrm{curv}(\Omega_1) = F^{*}_{\nu'}(\mathrm{curv}(\Omega_2))$.
\begin{comment}
Indeed,
\begin{equation*}
\mathrm{curv}(\Omega_1)=\mathrm{D}F_{\nu}^{*}\Omega_2 + \frac{1}{2}[F_{\nu}^{*}\Omega_2\wedge F_{\nu}^{*}\Omega_2] = F_{\mathrm{D}\nu}^{*}(\mathrm{D}\Omega_2) + \frac{1}{2}F_{[\nu \wedge \nu]}^{*}([\Omega_2 \wedge \Omega_2]) = F^{*}_{\mathrm{D}\nu + \frac{1}{2}[\nu \wedge \nu]}(\mathrm{curv}(\Omega_2))
\end{equation*}
where $\mathrm{D}\nu = (\mathrm{d}\nu_0+\nu_1,\mathrm{d}\nu_1)$ and
\begin{align*}
[\nu\wedge \nu]_0 &= [\nu_0\wedge \nu_0]+2\alpha_{*}(\alpha_r^{*}\fa{\Omega}_2\wedge \nu_0)
\\
[\nu\wedge \nu]_1 &=2 [\nu_0 \wedge \nu_1]+2\alpha_{*}(\alpha_r^{*} \fa\Omega_2\wedge \nu_1 )\text{,}
\end{align*}
so that
\begin{equation*}
\nu' :=\mathrm{D}\nu + \frac{1}{2}[\nu \wedge \nu] =\left(\mathrm{fcurv}(\nu), \mathrm{d}\nu_1+ [\nu_0 \wedge \nu_1]+\alpha_{*}(\alpha_r^{*} \fa\Omega_2\wedge \nu_1 )\right )\text{.}
\end{equation*}
\end{comment}
\end{remark}

\begin{remark}
Fake-flat obviously implies regular. Conversely,  if $\Omega_2$ is a fake-flat connection on $\inf P_2$, and $\nu$ is regular and connective, then $\nu$ is fake-flat. In order to see this, we first derive using connectivity of $\nu$ the following transformation rule of the fake-curvature under the $\mor\Gamma$-action on $F$:
\begin{equation}
\label{eq:transcurv}
\rho^{*}\mathrm{fcurv}(\nu)= (\alpha_{g^{-1}})_{*}\big(\mathrm{Ad}_h^{-1}(\pr_F^{*}\mathrm{fcurv}(\nu))
+(\tilde\alpha_h)_{*}(\pr_F^{*}\alpha_r^{*}\fa{\mathrm{curv}(\Omega_2)}) \big )\text{.}
\end{equation}
Fake-flatness of $\Omega_2$ and \cref{eq:transcurv} imply
\begin{equation}
\label{eq:re:fakeflatmorph:1}
\rho^{*}\mathrm{fcurv}(\nu)= (\alpha_{g^{-1}})_{*}\big(\mathrm{Ad}_h^{-1}(\pr_F^{*}\mathrm{fcurv}(\nu))\text{.}
\end{equation}
For $f\in F$ consider the splitting of $T_{f}F$ of \cref{re:splittingF} with respect to a \splitting $\sigma$ with $\sigma^{*}\mathrm{fcurv}(\nu)=0$. \Cref{eq:re:fakeflatmorph:1} implies that $\mathrm{fcurv}(\nu)$ vanishes when applied to at least one \quot{vertical} tangent vector, i.e. one in the image of $T\mor{\Gamma}$.
\begin{comment}
Let $w\in T_hH$ and $u\in T_gG$. Then, 
\begin{equation*}
T_{(h,g)}R_{\rho}(w,u) =T_{\rho,(h,g)}R(0,(w,u))\text{.}
\end{equation*}
Indeed, if $w$ and $u$ are represented by smooth curves $\omega$ and $\gamma$, respectively, then
\begin{equation*}
T_{(h,g)}R_{\rho}(\ddt 1t0 (\omega,\gamma))=\ddt 1t0 R_{\rho}(\omega,\gamma)=\ddt 1t0 R(\rho,(\omega,\gamma))=T_{\rho,(h,g)}R(\ddt 1t0 (\rho,(\omega,\gamma)))\text{.}
\end{equation*}
We conclude that $\fb{\mathrm{curv}(\Omega)}|_{R(\rho,(h,g))}(\tilde v_1,\tilde v_2)$ is zero whenever  $\tilde v_1$ or $\tilde v_2$ is of the form $T_{(h,g)}R_{\rho}(w,u)$. \end{comment}
\Cref{eq:re:fakeflatmorph:1} implies further that
\begin{equation*}
\sigma^{*}\rho_{h,g}^{*}\mathrm{fcurv}(\nu) = (\alpha_{g^{-1}})_{*} (\mathrm{Ad}_{h}^{-1}(\sigma^{*}\mathrm{fcurv}(\nu)) )\text{.} 
\end{equation*}
Hence, regularity of $\nu$ implies that $\mathrm{fcurv}(\nu)$ vanishes when applied to two \quot{horizontal} tangent vectors. In summary, the splitting into horizontal and vertical parts gives $\mathrm{fcurv}(\nu)=0$.
\end{remark}

\begin{remark}
If $\Omega_2$ is fake-flat and $F: \inf P_1 \to \inf P_2$ admits a connective $\Omega_2$-pullback, then $F$ also admits a connective and fake-flat $\Omega_2$-pullback. In order to see this, we derive the following transformation law under the right $\inf P_2$-action $\rho_r$, for the fake-curvature of  a  $\Omega_2$-pullback $\nu=(\nu_0,\nu_1)$:
\begin{equation*}
\rho_r^{*}\mathrm{fcurv}(\nu)= \pr_F^{*} \mathrm{fcurv}(\nu) + \pr_{\mor{\inf P_2}}^{*} \fb{\mathrm{curv}}(\Omega_2)\text{.}
\end{equation*}
\begin{comment}
Indeed,
\begin{align*}
\mathrm{fcurv}(\nu)|_{f \circ \eta} &= \mathrm{d}\nu_0|_{f \circ \eta}+ \frac{1}{2}[\nu_0|_{f \circ \eta}\wedge \nu_0|_{f \circ \eta}]+ \alpha_{*}(\fa\Omega_2 |_{s(\eta)}\wedge \nu_0|_{f \circ \eta})+\nu_1|_{f \circ \eta}
\\&= \mathrm{d}\nu_0|_{f} +\mathrm{d} \fb\Omega_2|_{\eta}+ \frac{1}{2}[\nu_0|_f \wedge \nu_0|_f ]+ \frac{1}{2}[ \fb\Omega_2|_{\eta}\wedge  \fb\Omega_2|_{\eta}]+ [\nu_0|_f \wedge  \fb\Omega_2|_{\eta}]
\\&\qquad + \alpha_{*}(\fa\Omega_2 |_{s(\eta)}\wedge  \nu_0|_f )+ \alpha_{*}(\fa\Omega_2 |_{s(\eta)}\wedge  \fb\Omega_2|_{\eta})+\nu_1|_f + (\Delta\fc\Omega_2)_{\eta}
\\&= \mathrm{d}\nu_0|_{f} + \frac{1}{2}[\nu_0|_f \wedge \nu_0|_f ]+ \alpha_{*}(\fa\Omega_2 |_{t(\eta)}\wedge  \nu_0|_f )+\nu_1|_f
\\&\qquad+\mathrm{d} \fb\Omega_2|_{\eta}+ \frac{1}{2}[ \fb\Omega_2|_{\eta}\wedge  \fb\Omega_2|_{\eta}]+ \alpha_{*}(\fa\Omega_2 |_{s(\eta)}\wedge  \fb\Omega_2|_{\eta}) + (\Delta\fc\Omega_2)_{\eta}
\\&\qquad + \alpha_{*}(\fa\Omega_2 |_{s(\eta)}\wedge  \nu_0|_f )+ [\nu_0|_f \wedge  \fb\Omega_2|_{\eta}]- \alpha_{*}(\fa\Omega_2 |_{t(\eta)}\wedge  \nu_0|_f )
\\&= \mathrm{fcurv}(\nu)|_{f} + \fb{\mathrm{curv}}(\Omega_2)|_\eta
\end{align*}
\end{comment}
Since $\Omega_2$ is fake-flat, $\mathrm{fcurv}(\nu)$ descends to a 2-form $\kappa\in \Omega^2(\ob{\inf P_1},\mathfrak{h})$ with $\alpha_{l}^{*}\kappa =  \mathrm{fcurv}(\nu)$. The shifted pullback $\nu^{\kappa}=(\nu,\nu_1-\alpha_l^{*}\kappa_1)=(\nu_0,\nu_1-\mathrm{fcurv}(\nu))$ is then obviously  fake-flat. For connectivity of the new pullback $\nu^{\kappa}$, we only have to check \cref{eq:conana:mu}, i.e.
\begin{equation*}
\rho^{*}\nu_1^{\kappa} =  (\alpha_{g^{-1}})_{*} \big ( \mathrm{Ad}_h^{-1} (\pr_F^{*}\nu_1^{\kappa})+(\tilde \alpha_h)_{*}(t_{*} (\pr_F^{*}\alpha_r^{*}\fc\Omega_2)) \big )\text{.}
\end{equation*} 
This follows directly from the connectivity of $\nu_1$ and 
\cref{eq:re:fakeflatmorph:1}.

\end{remark}

We recall that the composition of anafunctors comes along with a composition  of pullbacks on anafunctors, see \cref{lem:pullback:comp}. 

\begin{proposition}
The composition of pullbacks on 1-morphisms between principal $\Gamma$-2-bundles preserves the following conditions for pullbacks: \quot{connection-preserving}, \quot{connective}, \quot{regular}, and \quot{fake-flat}.
\end{proposition}

\begin{proof}
That the composite pullback is connection-preserving is the content  of \cref{lem:pullback:comp}. Connectivity is straightforward to check using the formulas for the composite pullback (\cref{lem:pullback:comp}) and the definition of the  $\mor{\Gamma}$-action on the composite anafunctor (\cref{lem:pullback:rules}).
Concerning fake-flatness, one can deduce the formula
\begin{equation}
\label{eq:prop:compo:1}
\mathrm{fcurv}(\nu' \circ \nu) = \pr_{F_2}^{*} \mathrm{fcurv}(\nu') + \pr_{F_1}^{*} \mathrm{fcurv}(\nu)
\end{equation}
for the fake-curvature of the composite pullback $\nu' \circ \nu$  on a composite  $F' \circ F$ of anafunctors $F:\inf P_1 \to \inf P_2$ and $F': \inf P_2 \to \inf P_3$. 
\begin{comment}
Indeed
\begin{eqnarray*}
 \mathrm{curv}_{\tilde\nu,\tilde\mu}(G\circ F)_{(f,g)}&=&\mathrm{d}\tilde\nu_{(f,g)}+ \frac{1}{2}[\tilde\nu _{(f,g)}\wedge \tilde \nu_{(f,g)}]+ \alpha_{*}(\fa\Omega_3|_{\alpha_r(f,g)} \wedge \tilde \nu_{(f,g)})+\tilde \mu_{(f,g)}
\\&=& \mathrm{d}\nu'_{f}+\mathrm{d}\nu_g+ \frac{1}{2}[\nu' _{f}+\nu_g\wedge \nu' _{f}+\nu_g]+ \alpha_{*}(\fa\Omega_3|_{\alpha_r(f,g)} \wedge \nu' _{f}+\nu_g)+\mu' _{f}+\mu_g
\\&=& \mathrm{d}\nu'_{f}+ \frac{1}{2}[\nu' _{f}\wedge \nu' _{f}]+\mu' _{f}+ \alpha_{*}(\fa\Omega_2|_{\alpha_l(g)} - t_{*}(\nu_g) \wedge \nu' _{f})\\&&\qquad+\mathrm{d}\nu_g+ \frac{1}{2}[\nu_g\wedge \nu_g]+ \alpha_{*}(\fa\Omega_3|_{\alpha_r(g)} \wedge \nu_g)+\mu_g + [\nu' _{f}\wedge \nu_g]
\\&=&  \mathrm{curv}_{\nu',\mu'}(F)_{f}+ \mathrm{curv}_{\nu,\mu}(G)_{g}\text{.} \end{eqnarray*}
\end{comment}
This shows that fake-flat pullbacks compose to fake-flat ones. For regularity, consider $p_i\in \ob{\inf P_i}$ for $i=1,2,3$ all projecting to the same point $x\in M$. Choose an open neighborhood $U_{12}$ of $(p_1,p_2)$ with a \splitting\ $\sigma_{12}$ of $F$ such that $\sigma_{12}^{*}\mathrm{fcurv}(\nu)=0$, 
and an open neighborhood $U_{23}$ of $(p_2,p_3)$ with a \splitting\ $\sigma_{23}$ of $F$ such that $\sigma_{23}^{*}\mathrm{fcurv}(\nu')=0$. Choose further an open neighborhood $V$ of $x$ with a section $\tau: V \to \ob{\inf P_2}$ such that $\tau(x)=p_2$. For $x_i\in \ob{\inf P_i}$ such that $\pi_i(x_i)\in V$ we write $x_i^{\tau} := \tau(\pi_i(x_i))$.
We define the open neighborhood
\begin{equation*}
U := \{ (x_1,x_3) \in  \ob{\inf P_1} \times_M \ob{\inf P_3} \sep \pi_1(x_1),\pi_3(x_3)\in V \text{ , } (x_1,x_1^{\tau})\in U_{12}\text{ , }(x_3^{\tau},x_3)\in U_{23} \}
\end{equation*}
of $(p_1,p_3)$. 
\begin{comment}
This is open: we first define $\tilde U := \pi_{13}^{-1}(V)$, where $\pi_{13}: \ob{\inf P_1} \times_M \ob{\inf P_3} \to M$ is the projection. We have a map $\tilde\sigma: \tilde U \to (\ob{\inf P_1} \times_M \ob{\inf P_2}) \times_{\ob{\inf P_2}} (\ob{\inf P_2} \times_M \ob{\inf P_3})$ defined by $\tilde\sigma(x_1,x_3) := ((x_1,\sigma(\pi_1(x_1))),(\sigma(\pi_2(x_2)),x_3))$. Let $\tilde\sigma_1$ and $\tilde\sigma_2$  be the projections to $\ob{\inf P_1} \times_M \ob{\inf P_2}$ and $\ob{\inf P_2} \times_M \ob{\inf P_3}$, respectively. Then, $U := \tilde\sigma_1^{-1}(U_{12}) \cap \tilde\sigma_2^{-1}(U_{23})$. Obviously, $\pi_1(p_1),\pi_3(p_3)\in V$, and $p_1^{\tau}=p_3^{\tau}=\tau(x)=p_2$, so that $(p_1,p_1^{\tau})\in U_{12}$ and $(p_3^{\tau},p_3) \in U_{23}$. 
\end{comment}
We find a smooth map $\sigma:U \to F \ttimes {\alpha_r}{\alpha_l} G$ defined by 
\begin{equation*}
\sigma(x_1,x_3)\mapsto  (\sigma_{12}(x_1,x_1^{\tau}) , \sigma_{23}(x_3^{\tau},x_3) \cdot \id_{g_{12}(x_1,x_1^{\tau})})\text{,}
\end{equation*}
where $g_{12}:U_{12} \to G$ is a transition function for $\sigma_{12}$.
\begin{comment}
This  is well-defined:
\begin{multline*}
\alpha_{r}(\sigma_{12}(x_1,x_1^{\tau}))=R(x_1^{\tau},g_{12}(x_1,x_1^{\tau}))=R(x_3^{\tau},g_{12}(x_1,x_1^{\tau}))\\=R(\alpha_l(\sigma_{23}(x_3^{\tau},x_3)),g_{12}(x_1,x_1^{\tau}))=\alpha_l(\sigma_{23}(x_3^{\tau},x_3) \cdot \id_{g_{12}(x_1,x_1^{\tau})})\text{.}
\end{multline*}
\end{comment}
This is a \splitting\ of $F'\circ F$,
\begin{comment}
Indeed, 
\begin{equation*}
\alpha_l(\sigma(x_1,x_3))=\alpha_l(\sigma_{12}(x_1,x_1^{\tau}))=x_1
\end{equation*}
and 
\begin{multline*}
\alpha_r(\sigma(x_1,x_3))=\alpha_r(\sigma_{23}(x_3^{\tau},x_3) \cdot \id_{g_{12}(x_1,x_1^{\tau})}) \\
= R(\alpha_r(\sigma_{23}(x_3^{\tau},x_3)),g_{12}(x_1,x_1^{\tau}))=R(x_3,g_{23}(x_3^{\tau},x_3)g_{12}(x_1,x_1^{\tau}))\text{,}
\end{multline*}
so that $(x_1,x_2)\mapsto g_{23}(x_3^{\tau},x_3)g_{12}(x_1,x_1^{\tau})$ is a transition function for $\sigma$. 
\end{comment}
and \cref{eq:prop:compo:1} implies that $\sigma^{*}\mathrm{curv}(\nu'\circ \nu)=0$. This shows that $\nu'\circ \nu$ is regular.
\begin{comment}
Indeed, we have
\begin{eqnarray*}
\mathrm{curv}_{\tilde\nu,\tilde\mu}(G\circ F)_{\sigma(x_1,x_3)} &=&\mathrm{curv}_{\nu',\mu'}(F)_{\sigma_{12}(x_1,p_2)}+ \mathrm{curv}_{\nu,\mu}(G)_{\sigma_{23}(p_2,x_3) \cdot \id_{g_{12}(x_1,p_2)}}
\\&=& (\alpha_{g_{12}(x_1,p_2)^{-1}})_{*}\big(\mathrm{curv}_{\nu,\mu}(G)_{\sigma_{23}(p_2,x_3)}
 \big )
 \\&=& 0\text{.}
\end{eqnarray*}
\end{comment}
\end{proof}

\noindent
Now we are in position to set up three bicategories of principal $\Gamma$-2-bundles with connection:
\begin{enumerate}[(a)]

\item 
A bicategory $\zweibuncon\Gamma M$ consisting of principal $\Gamma$-2-bundles with connections,  1-morphisms  with connective, connection-preserving pullbacks, and connection-preserving 2-morphisms. 

\item
A bicategory $\zweibunconreg\Gamma M$ consisting of principal $\Gamma$-2-bundles with regular connections,  1-morphisms  with regular, connective, connection-preserving pullbacks, and connection-preserving 2-morphisms. 
\item
A bicategory $\zweibunconff\Gamma M$ consisting of principal $\Gamma$-2-bundles with fake-flat connections,  1-morphisms with fake-flat, connective, connection-preserving pullbacks, and connection-preserving 2-morphisms. 

\end{enumerate}
We show in the next section that these three bicategories are classified by the three versions of non-abelian differential cohomology we have described in \cref{sec:nonabdiffcoho}. Moreover, it is straightforward to see that they form presheaves of bicategories,
\begin{equation*}
\zweibuncon\Gamma-
\quomma
\zweibunconreg\Gamma-
\quand
\zweibunconff\Gamma-
\end{equation*}
over the category of smooth manifolds, i.e., there are consistent pullback 2-functors along smooth maps. We will show in a separate paper that these presheaves are in fact 2-stacks.
 
\begin{remark}
\label{re:inv1morph}
In each of the three bicategories of principal $\Gamma$-2-bundles, every 1-morphism in invertible. In order to see this, we recall from \cite[Corollary 6.2.4]{Nikolaus} that every 1-morphism $F: \inf P_1 \to \inf P_2$  between the underlying principal $\Gamma$-2-bundles is invertible; in particular, a weak equivalence. By \cref{lem:pullback:inv}, every connection-preserving pullback $\nu$ on  $F$ induces a connection-preserving pullback $-\nu$ on $F^{-1}$. It is straightforward to check that $-\nu$ inherits each of the properties \quot{connective}, \quot{regular} and \quot{fake-flat} from $\nu$.
\end{remark}

\begin{remark}
\label{rem:smoothfunctorconnection}
Suppose $\phi: \inf P_1 \to \inf P_2$ is a smooth functor between principal $\Gamma$-2-bundles that is strictly equivariant with respect to the $\Gamma$-actions. Suppose $\Omega_1$ and $\Omega_2$ are connections on $\inf P_1$ and $\inf P_2$, respectively. The canonical $\Omega_2$-pullback $\nu$ 
on the associated anafunctor $F:= \ob{\inf P_1}  \ttimes \phi t \mor{\inf P_2}$ (see \cref{re:canpullback:a}) has the following properties:
\begin{enumerate}[(a)]

\item 
\label{rem:smoothfunctorconnection:a}
It is connection-preserving if and only if $\phi^{*}\Omega_2=\Omega_1$ (because $\phi^{*}\Omega_2=F_{\nu}^{*}\Omega_2$).

\item
\label{rem:smoothfunctorconnection:b}
It is always connective (a straightforward calculation).
\begin{comment}
Indeed, we have
\begin{eqnarray*}
\nu_0|_{\rho((x,\eta),(h,g))} &=&  \nu_{R_1(x,t((h,g))),R_2(\eta,(h,g))}
\\&=& \fb\Phi_{R_2(\eta,(h,g))}
\\&=& (\alpha_{g^{-1}})_{*}\left (\mathrm{Ad}_{h}^{-1}(\fb\Phi_{\eta})+(l_{h}^{-1} \circ \alpha_{h})_{*}(\fa\Phi_{s(\eta)}) +\theta_h \right) 
\\&=& (\alpha_{g^{-1}})_{*}\left (\mathrm{Ad}_{h}^{-1}(\nu_{x,\eta})+(l_{h}^{-1} \circ \alpha_{h})_{*}(\fa\Phi_{s(\eta)}) +\theta_h \right) 
\end{eqnarray*}
and
\begin{eqnarray*}
\nu_1|_{\rho((x,\eta),(h,g))} &=& \mu_{R_1(x,t((h,g))),R_2(\eta,(h,g))}
\\&=& - \fc\Phi_{s(R_2(\eta,(h,g))} + \fc\Phi_{\phi(R_1(x,t((h,g))))}
\\&=& - \fc\Phi_{R_2(s(\eta),g)} + \fc\Phi_{R_2(\phi(x),t(h)g)}
\\&=&- (\alpha_{g^{-1}})_{*}( \fc\Phi_{s(\eta)}) + (\alpha_{g^{-1}t(h)^{-1}})_{*} (\fc\Phi_{\phi(x)})
\\&=& (\alpha_{g^{-1}})_{*} \big ( \mathrm{Ad}_h^{-1} (\mu_{x,\eta})+\mathrm{Ad}_h^{-1}(\fc\Phi_{s(\eta)})- \fc\Phi_{s(\eta)} \big )
\\&=& (\alpha_{g^{-1}})_{*} \big ( \mathrm{Ad}_h^{-1} (\mu_{x,\eta})+(l_h^{-1}\circ \alpha_h)_{*}(t_{*} (\fc\Phi_{s(\eta)})) \big )
\end{eqnarray*}
\end{comment}

\item
\label{re:canpullback:fake}
Its fake-curvature is $\mathrm{fcurv}(\nu) = \pr_{\mor{\inf P_2}}^{*} \fb{\mathrm{curv}(\Omega_2)}$; \begin{comment}
Indeed,
\begin{eqnarray*}
&&\hspace{-2cm} \mathrm{d}\nu_{x,\eta}+ \frac{1}{2}[\nu_{x,\eta} \wedge \nu_{x,\eta}]+ \alpha_{*}(\fa\Phi_2|_{s(\eta)} \wedge \nu_{x,\eta})+\mu_{x,\eta} 
\\&=& \mathrm{d}\fb\Phi_{\eta}+ \frac{1}{2}[\fb\Phi_{\eta} \wedge \fb\Phi_{\eta}]+ \alpha_{*}(\fa\Phi_2|_{s(\eta)} \wedge\fb\Phi_{\eta})-\fc\Phi_{s(\eta)}+\fc\Phi_{\phi(x)}
\end{eqnarray*}
\end{comment}
hence it is fake-flat if $\Omega_2$ is fake-flat.

\item
\label{re:canpullback:reg}
It is regular if $\Omega_2$ is regular. To see this, suppose  $\sigma: U \to \inf P_2$ is a \splitting. Define $U' := (\phi \times \id)^{-1}(U)$ and $\tau(x_1,x_2) := (x_1,\sigma(\phi(x_1),x_2))$. This is a \splitting\ of $F$,
\begin{comment}
We have 
\begin{equation*}
\alpha_l(\tau(x_1,x_2)) = x_1
\quand
\alpha_r(\tau(x_1,x_2)) 
\end{equation*}
\end{comment}
and from \cref{re:canpullback:fake*} we get $\tau^{*}\mathrm{fcurv}(\nu) = (\phi\times \id)^{*}\sigma^{*}\fb{\mathrm{curv}(\Omega_2)}$. 
\end{enumerate}
If the canonical $\Omega_2$-pullback $\nu$ is shifted  by $\kappa=(\kappa_0,\kappa_1)$ in the sense of \cref{re:canpullback:b},  then the shifted pullback $\nu^{\kappa}$ has the following properties:
\begin{enumerate}[(a)]

\setcounter{enumi}{4}

\item 
It is connection-preserving if and only if 
\begin{equation*}
\fa\Omega_1 = \phi^{*}\fa\Omega_2 + t_{*}(\kappa_0)
\quomma
\fb\Omega_1 = \phi^{*}\fb\Omega_2 + \Delta\kappa_0
\quand
\fc\Omega_1=\phi^{*}\fc\Omega_2 + \kappa_1\text{;}
\end{equation*}
this is a reformulation of \cref{eq:shiftedpullback}.

\item
\label{re:canpullback:shifted:connective}
It is connective if  $\kappa_0$ and $\kappa_1$ are $R$-invariant in the sense that    $R^{*}\kappa_i = (\alpha_{\pr_2^{-1}})_{*}(\pr_{1}^{*}\kappa_i)$ over $\ob{\inf P_1} \times G$, for $i=0,1$; this is a straightforward calculation.
\begin{comment}
Indeed, then we have
\begin{equation*}
\kappa_i|_{\pr_1(\rho((x,\eta),(h,g)))} &=\kappa_i|_{R_1(x,t(h)g)}=(\alpha_{g^{-1}})_{*}(\mathrm{Ad}_h^{-1}(\kappa_i|_x))
\end{equation*}
and the $\kappa_i$ drop out of the conditions for connectivity.
\end{comment}

\item
\label{re:canpullback:shifted:fc}
Its fake-curvature is 
$\mathrm{fcurv}(\nu^{\kappa})=\mathrm{fcurv}(\nu) + \pr_1^{*}\big ( \mathrm{d}\kappa_0+ \frac{1}{2}[\kappa_0\wedge \kappa_0]+ \alpha_{*}(\phi^{*}\fa\Omega_2 \wedge \kappa_0)+\kappa_1   \big )$.
\begin{comment}
Indeed,
\begin{eqnarray*}
&&\hspace{-2cm} \mathrm{d}\nu_{x,\eta} + \mathrm{d}\kappa_0|_x+ \frac{1}{2}[\nu_{x,\eta} + \kappa_0|_x \wedge \nu_{x,\eta}+ \kappa_0|_x]+ \alpha_{*}(\fa\Phi_2|_{s(\eta)} \wedge \nu_{x,\eta}+ \kappa_0|_x)+\nu_1|_{x,\eta}+ \kappa_1|_x
\\&=& \mathrm{d}\nu_{x,\eta} + \frac{1}{2}[\nu_{x,\eta}  \wedge \nu_{x,\eta}]+ \alpha_{*}(\fa\Phi_2|_{s(\eta)} \wedge \nu_{x,\eta})+\nu_1|_{x,\eta}
 \\&&\qquad + \mathrm{d}\kappa_0|_x+ \frac{1}{2}[\kappa_0|_x \wedge \kappa_0|_x]+ [\fb\Phi_2|_{\eta} \wedge  \kappa_0|_x]+ \alpha_{*}(\fa\Phi_2|_{s(\eta)} \wedge  \kappa_0|_x)+ \kappa_1|_x 
\\&=& \mathrm{d}\nu_{x,\eta} + \frac{1}{2}[\nu_{x,\eta}  \wedge \nu_{x,\eta}]+ \alpha_{*}(\fa\Phi_2|_{s(\eta)} \wedge \nu_{x,\eta})+\nu_1|_{x,\eta}
 \\&&\qquad + \mathrm{d}\kappa_0|_x+ \frac{1}{2}[\kappa_0|_x \wedge \kappa_0|_x]+\alpha_{*} (t_{*}(\fb\Phi_2|_{\eta}) \wedge  \kappa_0|_x)+ \alpha_{*}(\fa\Phi_2|_{s(\eta)} \wedge  \kappa_0|_x)+ \kappa_1|_x 
\end{eqnarray*}
\end{comment}
Hence, using \cref{re:canpullback:fake*,re:canpullback:reg*}, we see that $\nu^{\kappa}$ is regular/fake-flat if 
\begin{equation}
\label{eq:canpullback:shifted}
\mathrm{d}\kappa_0+ \frac{1}{2}[\kappa_0\wedge \kappa_0]+ \alpha_{*}(\phi^{*}\fa\Omega_2 \wedge \kappa_0)+\kappa_1=0
\end{equation}
and $\Omega_2$ is regular/fake-flat. 
\begin{comment}
Indeed, we have $(\pr_1 \circ \tau)(x_1,x_2) = x_1$ and hence
\begin{equation*}
\tau^{*}\mathrm{fcurv}(\nu^{\kappa}) = \tau^{*}\mathrm{fcurv}(\nu)= \pr_1^{*}(\mathrm{d}\kappa_0+ \frac{1}{2}[\kappa_0\wedge \kappa_0]+ \alpha_{*}(\phi^{*}\fa\Omega_2 \wedge \kappa_0)+\kappa_1)\text{.}
\end{equation*}
\end{comment}

\end{enumerate}
\end{remark}

\begin{example}
\label{ex:gbunredcon}
We take up \cref{ex:gbunred}, where we  associated a principal $\Gamma$-2-bundle $\act PH$ over $M$ with regular connection $\Omega$ to each ordinary principal $G$-bundle $P$ over $M$ with connection $\omega$. As remarked there, every bundle morphism $\varphi:P\to P'$ induces a $\Gamma$-equivariant smooth functor $\tilde \varphi: \act PH \to \act{P'}H$, acting as $\varphi$ on the level of objects, and as $(p,h)\mapsto (\varphi(p),h)$ on the level of morphisms. If $\varphi$ is connection-preserving, then it is obvious from the definition of the induced connections $\Omega$ and $\Omega'$ that $\tilde\varphi^{*}\Omega'=\Omega$. Hence, by \cref{rem:smoothfunctorconnection} the canonical pullback on the induced anafunctor is connective and connection-preserving, it is regular because $\Omega'$ is regular, and it is flat if $\omega$ is flat. 
Summarizing, we have constructed   functors
\begin{equation*}
\buncon GM \to \zweibunconreg\Gamma M
\quand
\buncon GM^{\flat} \to \zweibunconff\Gamma M
\text{.}
\end{equation*}
These functors consistently realize ordinary gauge theory as a special case of higher gauge theory. Using \cref{lem:redGbun} we will provide in \cref{co:gbunredcon} a sufficient condition for being in the image of these functors.
\end{example}

Next we discuss how connections can be \quot{induced} along 1-morphisms between principal $\Gamma$-2-bundles.

\begin{proposition}
\label{re:inducedconn}
Let $F:\inf P_1\to\inf P_2$ be a 1-morphism between principal $\Gamma$-bundles,  $\Omega_2$  be a connection on $\inf P_2$,  $\nu$ be a connective  $\Omega_2$-pullback on $F$, and $\Omega_1 := F_{\nu}^{*}\Omega_2 \in \Omega^1(\inf P_1,\gamma)$. 
\begin{enumerate}[(a)]

\item 
\label{re:pullbackcond:a}
$\Omega_1$ is a connection (and $\nu$ is connection-preserving).

\item
\label{re:pullbackcond:c}
If  $\nu$ and $\Omega_2$ are fake-flat, then $\Omega_1$ is fake-flat.

\item
\label{re:pullbackcond:d}
If  $\nu$ and $\Omega_2$ are flat, then $\Omega_1$ is flat.

\item
\label{re:pullbackcond:b}
If $\nu$  and $\Omega_2$ are regular, then $\Omega_1$ is regular.

\end{enumerate}
\end{proposition}

\begin{proof}
For \cref{re:pullbackcond:a*} we have to verify \cref{eq:conform:a,eq:conform:b,eq:conform:c}.
It is straightforward to check using \cref{eq:conana:mu}  over $F \times G$ the relation 
\begin{equation*}
(\alpha_l \times \id)^{*}R^{*}\fc \Omega_1=(\alpha_{\pr_G^{-1}})_{*}(\pr_F^{*}\alpha_l^{*}\fc{\Omega}_1)\text{.}
\end{equation*}
\begin{comment}
We check
\begin{eqnarray*}
\nu_1|_{\rho(f,(h,g))} &=& \fc\Omega_1|_{\alpha_l(\rho(f,(h,g)))} - \fc\Omega_2|_{\alpha_r(\rho(f,(h,g)))}
\\&=& \fc\Omega_1|_{R_1(\alpha_l(f),t(h)g)} - \fc\Omega_2|_{R_2(\alpha_r(f),g)}
\\&=&  (\alpha_{g^{-1}t(h)^{-1}})_{*}(\fc\Omega_1|_{\alpha_l(f)}) - (\alpha_{g^{-1}})_{*} (\fc\Omega_2|_{\alpha_r(f)})
\\&=&  (\alpha_{g^{-1}})_{*}(\mathrm{Ad}_h^{-1}(\fc\Omega_1|_{\alpha_l(f)}) -\fc\Omega_2|_{\alpha_r(f)})
\\&=&  (\alpha_{g^{-1}})_{*}(\mathrm{Ad}_h^{-1}(\mu_f + \fc\Omega_2|_{\alpha_r(f)}) -\fc\Omega_2|_{\alpha_r(f)})
\\ &=& (\alpha_{g^{-1}})_{*} \big ( \mathrm{Ad}_h^{-1} (\nu_1|_{f})+(l_h^{-1}\circ \alpha_h)_{*}(t_{*} (\fc\Omega_2|_{\alpha_r(f)})) \big )\text{.}
\end{eqnarray*}
This calculation  can be interpreted via $(h=1)$ as $\fc\Omega_1|_{R_1(\alpha_l(f),g)}= (\alpha_{g^{-1}})_{*}(\fc\Omega_1|_{\alpha_l(f)})$.
\end{comment}
Since $\alpha_l$ is a surjective submersion this implies $R^{*}\fc\Omega_1=(\alpha_{g^{-1}})_{*}(\fc\Omega_1)$, which is \cref{eq:conform:c}. Similarly one shows \cref{eq:conform:a} using \cref{eq:conana:b}.  
\begin{comment}
Indeed
\begin{eqnarray*}
\fa\Omega_1|_{R_1(\alpha_l(f),g)} &=& \fa\Omega_1|_{\alpha_l(\rho(f,(1,g)))}
\\&=& t_{*}(\nu_0|_{\rho(f,(1,g))}) + \fa\Omega_2|_{\alpha_r(\rho(f,(1,g)))}
\\&=& t_{*}((\alpha_{g^{-1}})_{*}(\nu_0|_{f}) )+ \fa\Omega_2|_{R_2(\alpha_r(f),g)}
\\&=& \mathrm{Ad}_g^{-1}(t_{*}(\nu_0|_{f})) + \mathrm{Ad}_{g}^{-1}(\fa\Omega_2|_{\alpha_r(f)}) + \theta_g 
\\&=& \mathrm{Ad}_g^{-1}(\fa\Omega_1|_{\alpha_l(f)}) + \theta_g 
\end{eqnarray*}
\end{comment}
In order to prove \cref{eq:conform:b} we work over an open subset $V \subset \mor{\inf P_1} \times \mor{\Gamma}$ that admits a smooth map $f:V \to F$ such that $\alpha_l(f(\chi,(h,g)))=R(s(\chi),g)$. Denoting by $\rho_l$ the left action of $\inf P_1$ on $F$, we have from \cref{eq:conana:7} \begin{equation*}
R^{*}\fb\Omega_1|_V =(R \times f)^{*} \rho_l^{*}\nu_0 -f^{*}\nu_0\text{.}  
\end{equation*}
Rewriting $\rho_l$ in terms of the $\mor{\Gamma}$-action $\rho$ on $F$ using the conditions of \cref{def:actionanafunctor}, and then using \cref{eq:conana:b} yields \cref{eq:conform:b}. 
This shows \cref{re:pullbackcond:a*}. 
For \cref{re:pullbackcond:c*} we have from \cref{re:curvpull} $\mathrm{curv}(\Omega_1)=F^{*}_{\nu'}\mathrm{curv}(\Omega_2)$, where $\nu' =\left(\mathrm{fcurv}(\nu), \mathrm{curv}(\nu)\right ).$
By \cref{lem:pullback} the component $\fb{\mathrm{curv}(\Omega_1)}$ is uniquely determined by
\begin{equation}
\label{eq:re:inducedconn:1}
\pr_{\mor{\inf P_1}}^{*}(\fb{\mathrm{curv}(\Omega_1)})= \rho_l^{*}\mathrm{fcurv}(\nu)-\pr_F^{*}\mathrm{fcurv}(\nu)
\end{equation}
over $\mor{\inf P_1} \ttimes{s}{\alpha_l} F$, where $\rho_l$ is the left action and $\pr_F$ is the projection. The component $\fa{\mathrm{curv}(\Omega_1)}$ is determined by $\alpha_l^{*}\fa{\mathrm{curv}(\Omega_1)} = t_{*}(\mathrm{fcurv}(\nu)) +\alpha_r^{*}\fa{\mathrm{curv}(\Omega_2)}$. 
This shows \cref{re:pullbackcond:c*}. For \cref{re:pullbackcond:d*} we only have to look at the component $\fc{\mathrm{curv}(\Omega_1)}$, which is according to \cref{lem:pullback} determined by $\alpha_l^{*}\fc{\mathrm{curv}(\Omega_1)} = \mathrm{curv}(\nu) +\alpha_r^{*}\fc{\mathrm{curv}(\Omega_2)}$; this shows the claim.

For \cref{re:pullbackcond:b} we set  $Y_k:= \ob{\inf P_k} \times_M \ob{\inf P_k}$ for $k=1,2$,  consider $Z:= Y_1 \times _M Y_2$ and the projection $\pr_{12}:Z \to Y_1 $. For a given point $(p_1,p_1')\in Y_1$ let $U_1\subset Y_1$ be an open neighborhood with a section $s: U_1 \to W$. We put $(p_1,p_1',p_2,p_2') := s(p_1,p_1')$. We choose a \splitting\ $\sigma_2$ of $\inf P_2$ defined in a neighborhood of $(p_2,p_2')\in Y_2$, as well as \splitting s $\tau$ of $F$ in a neighborhood of $(p_1,p_2)$, with transition function $h$, and $\tau'$ in a neighborhood of $(p_1',p_2')$. We claim that there exists a \splitting\  $\sigma_1$ of $\inf P_1$ with transition function $g_1$ defined on $U_1$ such that over $W$ we have
\begin{equation}
\label{eq:re:inducedconn:5}
\tau' = \rho(\sigma_1^{-1} \circ \tau \circ R(\sigma_2,\id_h),(1,g_1^{-1}))\text{.}
\end{equation}
In order to see this, we choose some \splitting\ of $\inf P_1$ on $U_1$ with some transition function $g_1$. Then one can prove that the right hand side of \cref{eq:re:inducedconn:5} defines a \splitting\ of $F$ along $\pr_{24}:W \to \ob{\inf P_1} \times_M \ob{\inf P_2}$.
\begin{comment}
We define maps $\tilde\tau:W \to F$ and $\tilde h: W \to G$ by
\begin{equation*}
\tilde\tau := \rho(\sigma_1^{-1} \circ \tau \circ R(\sigma_2,\id_h),(1,g_1^{-1}))
\quand
\tilde h := g_2hg_1^{-1}\text{.}
\end{equation*} 
This satisfies
\begin{align*}
\alpha_l(\tilde\tau(x_1,x_1',x_2,x_2'))&=R(\alpha_l(\sigma_1^{-1} \circ \tau \circ R(\sigma_2,\id_h)),g_1^{-1})\\&=R(s(\sigma_1),g_1^{-1})\\&=x_1'
\\
\alpha_r(\tilde\tau(x_1,x_1',x_2,x_2'))&=R(\alpha_r(\sigma_1^{-1} \circ \tau \circ R(\sigma_2,\id_h)),g_1^{-1})\\&=R(s(R(\sigma_2,\id_h)),g_1^{-1})\\&=R(s(\sigma_2),hg_1^{-1})\\&=R(x_2',g_2hg_1^{-1})
\\&=R(x_2',\tilde h)\text{.}
\end{align*}
Thus, $\tilde\tau$ is a \splitting\ of $F$ along the projection $W \to V'$, with transition function $\tilde h$. 
\end{comment}
By \cref{lem:Fsections}, the difference between this \splitting\ and $\tau' \circ \pr_{24}$ defines a smooth map $x:W \to H$. Acting with $x\circ s:U \to H$ on $\sigma_1$ yields a new \splitting\ of $\inf P_1$ with the claimed property.

By eventually passing to smaller open sets, we can assume by regularity of $\nu$ that $\tau^{*}\mathrm{fcurv}(\nu)=\tau'^{*}\mathrm{fcurv}(\nu)=0$, and by regularity of $\Omega_2$ that $\sigma_2^{*}\fb{\mathrm{curv}(\Omega_2)}=0$. Now we rewrite \cref{eq:re:inducedconn:5} as
\begin{equation}
\label{eq:re:inducedconn:6}
\sigma_1^{-1} \circ \tau = \rho(\tau',(1,g_1))\circ R(\sigma_2^{-1},\id_h)\text{.}
\end{equation}
\cref{eq:re:inducedconn:1} gives
\begin{equation*}
\sigma_1^{*}\fb{\mathrm{curv}(\Omega_1)}= -(\sigma_1^{-1} \circ \tau )^{*} \mathrm{fcurv}(\nu) +\tau^{*} \mathrm{fcurv}(\nu)\text{.}
\end{equation*}
Using the expression of \cref{eq:re:inducedconn:6} and then simplifying with \cref{eq:transcurv,eq:conform:a} yields zero. 
\begin{comment}
\begin{align*}
\fb{\mathrm{curv}(\Omega)}_{\sigma_1} &= -\fb{\mathrm{curv}(\Omega)}_{\sigma_1^{-1}}
\\&\eqcref{eq:re:inducedconn:1}  - \fb{\mathrm{fcurv}(\nu)}_{\sigma_1^{-1} \circ \tau } + \fb{\mathrm{fcurv}(\nu)}_{\tau}
\\&= - \fb{\mathrm{fcurv}(\nu)}_{\rho(\tau',(1,g_1))\circ R(\sigma_2^{-1},\id_h)} + \fb{\mathrm{fcurv}(\nu)}_{\tau}
\\&= - \fb{\mathrm{fcurv}(\nu)}_{\rho(\tau',(1,g_1))}- \fb{\mathrm{curv}(\Omega_2)}_{ R(\sigma_2^{-1},\id_h)} + \fb{\mathrm{fcurv}(\nu)}_{\tau}
\\&\eqcref{eq:transcurv,eq:conform:a} - (\alpha_{g_1^{-1}})_{*}( \fb{\mathrm{fcurv}(\nu)}_{\tau'})+ (\alpha_{h^{-1}})_{*}( \fb{\mathrm{curv}(\Omega_2)}_{\sigma_2} )+ \fb{\mathrm{fcurv}(\nu)}_{\tau}
\end{align*}
\end{comment}
\end{proof}

\begin{remark}
Pullbacks on 1-morphisms can be \quot{induced} in the following way: let $F,F':\inf P_1 \to \inf P_2$ be 1-morphisms between principal $\Gamma$-2-bundles equipped with connections $\Omega_1$ and $\Omega_2$, respectively. If $\nu'$ is an $\Omega_2$-pullback on $F'$, and $f:F\Rightarrow F'$ is a 2-morphism, then $\nu:=f^{*}\nu'$ is a pullback of $F$, and each of the properties \quot{connection-preserving}, \quot{connective}, \quot{regular}, and \quot{fake-flat} holds for $\nu$ if and only of it holds for $\nu'$.  
\end{remark}

\begin{theorem}
\label{th:existence}
Let $\inf P$ be a principal $\Gamma$-2-bundle for which the $G$-action on $\ob{\inf P}$ is free and proper. Then, there exists a connection on $\inf P$. 
\end{theorem}

\begin{proof}
We choose an open cover $\mathcal{U}=\{U_{i}\}_{i\in I}$ of $M$, together with local trivializations $\mathcal{T}_i: \inf P|_{U_i} \to \idmorph{(U_i)} \times \Gamma$ and a subordinated partition of unity $\{\phi_i\}_{i\in I}$.
As seen in \cref{ex:trivialconn}, the Maurer-Cartan form $\Theta$ is a connection on the trivial principal $\Gamma$-2-bundle. By \cref{prop:exconnpull} there exist  connective  $\Theta$-pullbacks $\nu_i$ on $\mathcal{T}_i$, and by \cref{re:pullbackcond:a} $(\mathcal{T}_i)^{*}_{\nu_i}\Theta$ is a  connection on $\inf P|_{U_i}$. We set
\begin{equation*}
\Omega := \sum_{i\in I} \pi^{*}\phi_i \cdot  (\mathcal{T}_i)_{\nu_i}^{*}\Theta\text{;}
\end{equation*}
this is a connection on $\inf P$.
\end{proof}

\subsection{Classification by non-abelian differential cohomology}

\label{sec:classdiff}

Let $(g,a,A,B,\varphi)$ be a generalized differential $\Gamma$-cocycle with respect to an open cover $\mathcal{U}=\{U_i\}_{i\in I}$. We define on the reconstructed principal $\Gamma$-2-bundle $\inf P_{(g,a)}$ (see \cref{sec:2bundleclass}) a connection $\Omega$, via
\begin{align*}
&\fa\Omega|_{U_i \times G}\mquad &&:= \mathrm{Ad}_{\pr_G}^{-1}(A_i ) + \pr_G^{*}\theta
\\
&\fb\Omega|_{(U_i \cap U_j) \times H \times G}\mquad\mquad\mquad &&:= (\alpha_{\pr_G^{-1}})_{*}( \mathrm{Ad}_{\pr_H}^{-1}(\varphi_{ij}) + (\tilde\alpha_{\pr_H})_{*}(A_j)  + \pr_H^{*} \theta)
\\
&\fc\Omega|_{U_i \times G}\mquad &&:= -(\alpha_{\pr_G^{-1}})_{*}(B_i)\text{.}
\end{align*}
Cocycle condition
\cref{eq:transconnco} implies $\Delta \fa\Omega=t_{*}(\fb\Omega)$, and cocycle conditions
\cref{eq:transconn2co,eq:transconnco} imply $\Delta\fb\Omega=0$. 
Thus, we have a 1-form $\Omega\in\Omega^1(\inf P_{(g,a)},\gamma)$. In order to show that this 1-form is a connection, we check
\Cref{eq:conform:a,eq:conform:b,eq:conform:c}, which is straightforward. A bit tedious, but straightforward to check is the following.

\begin{lemma}
\label{lem:curvrecon}
The curvature of the connection $\Omega$ has the following components:
\begin{align*}
&\fa{\mathrm{curv}(\Omega)}|_{U_i \times G}  &&=  \mathrm{Ad}_{\pr_G}^{-1} (\mathrm{fcurv}(A_i,B_i))
\\
&\fb{\mathrm{curv}(\Omega)}|_{(U_i \cap U_j) \times H \times G}\mquad &&=  (\alpha_{\pr_G^{-1}})_{*}(\mathrm{Ad}_{\pr_H}^{-1}(\alpha_{*}(A_j\wedge \varphi_{ij})+\mathrm{d}\varphi_{ij}+\frac{1}{2}[\varphi _{ij}\wedge\varphi_{ij}]\\&&&\hspace{5em} +B_j-(\alpha_{\pr_G^{-1}})_{*}(B_i) )+(\tilde\alpha_{\pr_H})_{*}(\mathrm{fcurv}(A_j,B_j)))
\\
&\fc{\mathrm{curv}(\Omega)}|_{U_i \times G} &&= -(\alpha_{\pr_G^{-1}})_{*}(\mathrm{curv}(A_i,B_i))\text{.}  
\end{align*}
\end{lemma}

Next we verify that above \quot{reconstruction} of principal $\Gamma$-2-bundles with connection from local data is well-defined in non-abelian differential cohomology. 

\begin{lemma}
\label{lem:rgen}
The construction of the pair $(\inf P_{(g,a)},\Omega)$ induces a well-defined map 
\begin{equation*}
r^{\gen}: \hat\h^1(M,\Gamma)^{\gen} \to \hc 0 (\zweibuncon\Gamma M)\text{.}
\end{equation*}
\end{lemma}

\begin{proof}
We have to check two things: (a) Invariance under generalized equivalences of cocycles, and (b) Invariance under refinement of open covers.
 For (a), let another generalized differential $\Gamma$-cocycle $(g',a',A',B',\varphi')$ be equivalent to $(g,a,A,B,\varphi)$ via generalized equivalence data $(h,e,\phi)$. Then we have the  functor $\phi: \inf P_{(g,a)} \to \inf P_{(g',a')}$ defined in \cref{sec:2bundleclass}, given by
\begin{equation*}
\phi(i,x,g):= (i,x,h_i(x)g)
\quand
\phi(i,j,x,h,g) := (i,j,x,e_{ij}(x)\alpha(h_j(x),h),h_j(x)g)\text{.}
\end{equation*}
Let $F$ be the corresponding anafunctor, an let $\nu$ be the canonical $\Omega_2$-pullback on $F$. We define $\kappa=(\kappa_0,\kappa_1)$  with $\kappa_0\in \Omega^1(\ob{\inf P_{(g,a)}},\mathfrak{h})$ and $\kappa_1 \in \Omega^2(\ob{\inf P_{(g,a)}},\mathfrak{h})$ by
\begin{equation*}
\kappa_0|_{U_i \times G} := (\alpha_{\pr_G^{-1} \cdot h_i^{-1}})_{*}(\phi_i)
\quand
\kappa_1|{_{U_i \times G}} := (\alpha_{\pr_G^{-1}})_{*}((\alpha_{h_i^{-1}})_{*}(B_i')-B_i) \text{.}
\end{equation*}
The shifted pullback $\nu^{\kappa}$ on $F$ is connective (the $R$-invariance in \cref{re:canpullback:shifted:connective} is easy to check)
\begin{comment}
Indeed, by \cref{re:canpullback:shifted:connective@} we have to show $R^{*}\kappa_i = (\alpha_{\pr_2^{-1}})_{*}(\pr_{1}^{*}\kappa_i)$. For $\kappa_0$, this is
\begin{equation*}
\kappa_0|_{(i,x,gg')} =(\alpha_{g'^{-1}g^{-1}h_i^{-1}})_{*}(\phi_i)=  (\alpha_{g'^{-1}})_{*}(\alpha_{g^{-1}h_i^{-1}})_{*}(\phi_i)= (\alpha_{g'^{-1}})_{*}\kappa_0|_{(i,x,g)}\text{,}
\end{equation*}
and for $\kappa_2$, this is
\begin{equation*}
\kappa_1|_{(i,x,gg')} = (\alpha_{g'^{-1}g^{-1}})_{*}(\alpha_{h_i(x)^{-1}}(B_i')-B_i)= (\alpha_{g'^{-1}})_{*} (\alpha_{g^{-1}})_{*}(\alpha_{h_i(x)^{-1}}(B_i')-B_i)= (\alpha_{g'^{-1}})_{*}\kappa_1|_{(i,x,g)}
\end{equation*}
\end{comment}
 and connection-preserving: using the equivalence condition of \cref{32a}  one can deduce $\fa\Omega = \phi^{*}\fa{\Omega'} + t_{*}\kappa_0$; using the equivalence conditions of \cref{32d,32a,32c} one can deduce $\fb\Omega =\phi^{*}\fb{\Omega'}+ \Delta\kappa_0$, and $\fc\Omega = \phi^{*}\fc{\Omega} + \kappa_1$ holds by definition of $\kappa_1$. 
Hence $F$ and $\nu^{\kappa}$ form a 1-morphism in the  bicategory $\zweibuncon\Gamma M$. 

For (b), let  $\mathcal{U}'$ be a refinement of $\mathcal{U}$, $(g',a',A',B',\varphi')$ be the refined $\Gamma$-cocycle, and $\Omega'$ be the corresponding connection on $\inf P_{(g',a')}$. Then, the evident smooth functor $\phi: \inf P_{(g',a')} \to \inf P_{(g,a)}$ obviously satisfies $\phi^{*}\Omega'=\Omega$.
According to \cref{rem:smoothfunctorconnection:a,rem:smoothfunctorconnection:b} it induces a 1-morphism in the bicategory $\zweibuncon\Gamma M$. 
In both cases, \cref{re:inv1morph} guarantees that these 1-morphisms are automatically 1-isomorphisms.
\end{proof}

\begin{lemma}
The construction of the pair $(\inf P_{(g,a)},\Omega)$  induces a well--defined map
\begin{equation*}
r: \hat\h^1(M,\Gamma) \to \hc 0 (\zweibunconreg\Gamma M)\text{.}
\end{equation*}
\end{lemma}

\begin{proof}
If a differential $\Gamma$-cocycle $(g,a,A,B,\varphi)$ is non-generalized, then the additional cocycle condition of \cref{eq:transcurvco} reduces
the expression for $\fb{\mathrm{curv}(\Omega)}$ of \cref{lem:curvrecon} to
\begin{equation}
\label{eq:curvbmod}
\fb{\mathrm{curv}(\Omega)}|_{(U_i \cap U_j) \times H \times G} = (\alpha_{\pr_G^{-1}})_{*}((\tilde\alpha_{\pr_H})_{*}(\mathrm{fcurv}(A_j,B_j)))\text{.}
\end{equation}
Using the canonical \splitting s $\sigma_{ij}(x)=(i,j,x,1,g_{ij}(x))$ of $\inf P_{(g,a)}$, we see that $\sigma_{ij}^{*}\fb{\mathrm{curv}(\Omega)}=0$. Thus, the connection $\Omega$ is regular. Next we modify the parts (a) (Invariance under equivalences of cocycles) and (b) (Invariance under refinement of open covers) of the proof of \cref{lem:rgen}.
Concerning (a), we suppose another (non-generalized) differential $\Gamma$-cocycle $(g',a',A',B',\varphi')$ is equivalent to $(g,a,A,B,\varphi)$ via \emph{non-generalized} equivalence data $(h,e,\phi)$. Then, the additional equivalence condition of \cref{eq:equivtranscurv}, together with the cocycle condition of \cref{32a},
implies
\begin{equation*}
\kappa_1 =-\alpha_{*}(\phi^{*}\fa{\Omega'} \wedge \kappa_0)- \mathrm{d}\kappa_0- \frac{1}{2}[\kappa_0 \wedge \kappa_0]\text{.} 
\end{equation*}
\begin{comment}
Indeed,
\begin{align*}
\kappa_1|{_{U_i \times G}} &= (\alpha_{\pr_G^{-1}})_{*}((\alpha_{h_i^{-1}})_{*}(B_i')-B_i) \\&\eqcref{eq:equivtranscurv}(\alpha_{\pr_G^{-1}})_{*}  ((\alpha_{h_i^{-1}})_{*}(-\alpha_{*}(A_i' \wedge \phi_i) - \mathrm{d}\phi_i - \frac{1}{2}[\phi_i \wedge \phi_i]))
\\&= -  \alpha_{*}(\mathrm{Ad}^{-1}_{h_i \cdot \pr_G} (A_i') \wedge (\alpha_{\pr_G^{-1}h_i^{-1}})_{*}(\phi_i)) - (\alpha_{\pr_G^{-1}h_i^{-1}})_{*}(\mathrm{d}\phi_i) \\&\qquad- \frac{1}{2}[(\alpha_{\pr_G^{-1}h_i^{-1}})_{*}(\phi_i) \wedge (\alpha_{\pr_G^{-1}h_i^{-1}})_{*}(\phi_i)]
\\&=   -\alpha_{*}(\mathrm{Ad}^{-1}_{h_i \cdot \pr_G} (A_i')-\bar\theta_{\pr_G^{-1}h_i^{-1}} \wedge \kappa_0) - \mathrm{d}\kappa_0- \frac{1}{2}[\kappa_0 \wedge \kappa_0]
\\&=-  \alpha_{*}(\mathrm{Ad}^{-1}_{h_i \cdot \pr_G} (A_i')+\theta_{h_i\cdot \pr_G} \wedge \kappa_0) - \mathrm{d}\kappa_0- \frac{1}{2}[\kappa_0 \wedge \kappa_0]
\\&= -\alpha_{*}(\phi^{*}\fa{\Omega'} \wedge \kappa_0) - \mathrm{d}\kappa_0- \frac{1}{2}[\kappa_0 \wedge \kappa_0]
\end{align*}
\end{comment}
Using the formula of \cref{re:canpullback:shifted:fc}, 
\begin{comment}
namely,
\begin{equation*}
\mathrm{fcurv}(\nu^{\kappa})=\mathrm{fcurv}(\nu) + \pr_1^{*}\big ( \mathrm{d}\kappa_0+ \frac{1}{2}[\kappa_0\wedge \kappa_0]+ \alpha_{*}(\phi^{*}\fa\Omega_2 \wedge \kappa_0)+\kappa_1   \big )\text{,}
\end{equation*}
\end{comment}
this means that
$\mathrm{fcurv}(\nu^{\kappa})=\mathrm{fcurv}(\nu)$, so that \cref{re:canpullback:reg} implies that $\nu^{\kappa}$ is regular. Concerning (b), \cref{re:canpullback:reg} implies directly that the canonical pullback on the obvious smooth functor is regular. \end{proof}

Equipped with the maps $r^{\gen}$ and $r$ we are in position to state the following result.

\begin{theorem}
\label{th:classmain}
Principal $\Gamma$-2-bundles with connections are classified up to 1-isomorphism by differential non-abelian cohomology, in the sense that we have a commutative diagram
\begin{equation*}
\alxydim{}{\hat\h^1(M,\Gamma)^{\ff} \ar[d]^{} \ar@{^(->}[r] & \hat\h^1(M,\Gamma) \ar[d]^{r} \ar[r] & \hat\h^1(M,\Gamma)^{\gen} \ar[d]^{r^{\gen}} \\
\hc 0 (\zweibunconff \Gamma M) \ar@{^(->}[r] & \hc 0 (\zweibunconreg \Gamma M) \ar[r] & \hc 0 (\zweibuncon \Gamma M)\text{,}}
\end{equation*}
in which all vertical arrows are bijections. 
\end{theorem}
\begin{proof}
First of all, we note that the restriction of $r$ to the subset $\hat\h^1(M,\Gamma)^{\ff}$ results in fake-flat connections $\Omega$. Indeed, one can see from \cref{lem:curvrecon} and \cref{eq:curvbmod} that the vanishing of $\mathrm{fcurv}(A_i,B_i)$  implies the vanishing of $\fa{\mathrm{curv}(\Omega)}$ and $\fb{\mathrm{curv}(\Omega)}$.
This shows that the diagram exists and is commutative.

In order to show surjectivity  of its vertical maps we explain how to extract differential $\Gamma$-cocycles. Suppose $\inf P$ is a principal $\Gamma$-2-bundle over $M$ with  connection $\tilde\Omega$. We  proceed relative to the  choices  made in \cref{sec:2bundleclass}, i.e.
a cover $\{U_i\}_{i\in I}$ of $M$ by contractible open sets with contractible double intersections, smooth sections $s_i: U_i \to \inf P$, and \splitting s $\sigma_{ij}:U_i \cap U_j \to \mor{\inf P}$ along $(s_i,s_j):U_i \cap U_j \to \ob{\inf P} \times_M \ob{\inf P}$, together with transitions functions $g_{ij}:U_i \cap U_j \to G$. 
As described in \cref{sec:2bundleclass}, these choices induce smooth maps $a_{ijk}:U_i \cap U_j \cap U_k \to H$ in such a way that $(g,h)$ is a $\Gamma$-cocycle. Now we add the required differential forms:
\begin{enumerate}[(a)]

\item 
over each open set $U_i$, the differential forms
\begin{equation*}
A_i := s_i^{*}\fa{\tilde\Omega} \in \Omega^1(U_i,\mathfrak{g})
\quand
B_i := -s_i^{*}\fc{\tilde\Omega} \in \Omega^2(U_i,\mathfrak{h})\text{.}
\end{equation*}

\item
Over each double intersection $U_i \cap U_j$ the differential form
\begin{equation*}
\varphi_{ij} :=(\alpha_{g_{ij}})_{*}(\sigma_{ij}^{*}\fb{\tilde\Omega}) \in \Omega^1(U_i \cap U_j,\mathfrak{h})\text{.}
\end{equation*}

\end{enumerate}
We show the following:
\begin{enumerate}[1.)]

\item 
The collection $(g,h,A,B,\varphi)$ is a generalized differential $\Gamma$-cocycle.

Over $U_i \cap U_j$, we
pull back \cref{eq:conform:a} along $(s_i,g_{ij}^{-1})$, resulting in $(s_i,g_{ij}^{-1})^{*}R^{*}\fa{\tilde\Omega} = \mathrm{Ad}_{g_{ij}}(A_i) - g_{ij}^{*}\bar \theta$. On the left, we rewrite $R(s_i,g_{ij}^{-1}) =t(R(\sigma_{ij},\id_{g_{ij}^{-1}}))$ and obtain
\begin{multline*}
(\sigma_{ij},\id_{g_{ij}^{-1}})^{*}R^{*}t^{*}\fa{\tilde\Omega} = (\sigma_{ij},\id_{g_{ij}^{-1}})^{*}R^{*}(s^{*}\fa{\tilde\Omega} + t_{*}(\fb{\tilde\Omega}))\\= s_j^{*}\fa{\tilde\Omega} + t_{*}((\sigma_{ij},\id_{g_{ij}^{-1}})^{*}R^{*}\fb{\tilde\Omega})=A_j+t_{*}(\varphi_{ij})
\end{multline*}
where the last step uses \cref{eq:conform:b}. This is the cocycle condition \cref{eq:transconnco}. 
Over $U_i \cap U_j \cap U_k$ we pull back \cref{eq:conform:b} along $(\sigma_{ik},(\alpha(g_{ik}^{-1},a_{ijk}),g_{ik}^{-1}g_{jk}g_{ij}))$. Using the second half of \cref{eq:local:sub1} the left hand side becomes
\begin{equation*}
(\sigma_{ij} \circ R(\sigma_{jk},\id_{g_{ij}}))^{*}\fb{\tilde\Omega} = \sigma_{ij}^{*}\fb{\tilde\Omega} +(\alpha_{g_{ij}^{-1}})_{*}(\sigma_{jk}^{*}\fb{\tilde\Omega})=(\alpha_{g_{ij}^{-1}g_{jk}^{-1}})_{*}((\alpha_{g_{jk}})_{*}(\varphi_{ij}) + \varphi_{jk})\text{.}
\end{equation*}
The right hand side is, using the first half of \cref{eq:local:sub1}  and \cref{eq:conform:a},
\begin{equation*}
(\alpha_{g_{ij}^{-1}g_{jk}^{-1}})_{*}\left (\mathrm{Ad}_{a_{ijk}^{-1}}(\varphi_{ik})+(\tilde \alpha_{a_{ijk}})_{*}(A_k) +a_{ijk}^{*}\theta  \right)\text{.}
\end{equation*}
\begin{comment}
Indeed,
\begin{align*}
\\&\mquad(\alpha_{g_{ij}^{-1}g_{jk}^{-1}g_{ik}})_{*}\left (\mathrm{Ad}_{\alpha(g_{ik}^{-1},a_{ijk})}^{-1}(\sigma_{ik}^{*}\fb\Omega)+(\tilde \alpha_{\alpha(g_{ik}^{-1},a_{ijk})})_{*}((s_k,g_{ik})^{*}R^{*}\fa\Omega) +\alpha(g_{ik}^{-1},a_{ijk})^{*}\theta \right)
\\&=(\alpha_{g_{ij}^{-1}g_{jk}^{-1}})_{*}\left (\mathrm{Ad}_{a_{ijk}^{-1}}(\alpha_{g_{ik}})_{*}(\sigma_{ik}^{*}\fb\Omega)+(\tilde \alpha_{a_{ijk}})_{*}(s_k^{*}\fa\Omega ) +a_{ijk}^{*}\theta  \right)
\\&=(\alpha_{g_{ij}^{-1}g_{jk}^{-1}})_{*}\left (\mathrm{Ad}_{a_{ijk}^{-1}}(\varphi_{ik})+(\tilde \alpha_{a_{ijk}})_{*}(A_k) +a_{ijk}^{*}\theta  \right)
\end{align*}
\end{comment}
This implies the cocycle condition \cref{eq:transconn2co}. This shows that $(g,a,A,B,\varphi)$ is a generalized differential $\Gamma$-cocycle.

\item
If $\tilde\Omega$ is regular, then $(g,a,A,B,\varphi)$ is non-generalized. 

In fact, we have to pass to a smaller open cover constructed as follows. Consider a cover $\mathcal{V}$ of $X:=\ob{\inf P} \times_M \ob{\inf P}$ consisting of open sets $V\subset X$ supporting  transition spans $\sigma$ with $\sigma^{*}\fb{\mathrm{curv}(\tilde\Omega)}=0$; these open sets exist because $\tilde\Omega$ is regular. The inverse image of $\mathcal{V}$ under the map $(s_i,s_j): U_i \cap U_j \to X$ is an open cover $\mathcal{V}_{ij}$ of $U_i \cap U_j$. On paracompact spaces (such as smooth manifolds), every such \quot{hypercover of height 1} can be refined by an ordinary open cover. Restriction of sections $s_i$ to this refinement guarantees the existence of transition spans $\sigma_{ij}$ along $(s_i,s_j)$ with  $\sigma_{ij}^{*}\fb{\mathrm{curv}(\tilde\Omega)}=0$,  for all double intersections\footnote{This is not totally trivial, but drops a bit out of the context. First of all, in our situation  \cite[Lemma 7.2.3.5]{Lurie2009} states the following: there exists an open cover $\mathcal{W}=\{W_i\}_{i\in J}$ of $M$ and  a refinement map $r:J \to I$ with $W_i \subset U_{r(i)}$, such that for $i,j\in J$ with $r(i)\neq r(j)$ there exists an open set $W\in \mathcal{V}_{ij}$ with $W_i \cap W_j \subset W$. Let $\tilde s_i := s_{r(i)}|_{W_i}$ denote the restricted sections. Consider a double intersection $W_i\cap W_j$.
\begin{itemize}

\item 
If $r(i)\neq r(j)$, choose $W\in \mathcal{V}_{ij}$ with $W_i \cap W_j \subset W$. By construction, $W=(s_i,s_j)^{-1}(V)$, where $V\subset X$ is open and supports a transition span $\sigma$ with $\sigma^{*}\fb{\mathrm{curv}(\tilde\Omega)}=0$.  Thus, $\sigma_{ij}:= \sigma \circ (s_i,s_j):W \to \mor{\inf P}$ is a transition span along $(s_i,s_j)$. Now we restrict to $W_i \cap W_j \subset W$, and obtain a transition span $\tilde\sigma_{ij}:W_i \cap W_j \to \mor{\inf P}$  along $(\tilde s_i,\tilde s_j)$.

\item
If $r(i)=r(j)$, then $\tilde s_i = \tilde s_j$. Thus, $\tilde\sigma_{ij} := \id_{\tilde s_i}$  is a transition span with  $\tilde\sigma_{ij}^{*}\fb{\mathrm{curv}(\tilde\Omega)}=\tilde s_i^{*} \id^{*} \fb{\mathrm{curv}(\tilde\Omega)}=0$.
\end{itemize}}. We obtain the generalized $\Gamma$-cocycle $(g,a,A,B,\varphi)$ as described before, but over  $U_i\cap U_j$ we obtain
using \cref{eq:conform:c}:
\begin{equation*}
0=(\alpha_{g_{ij}})_{*}(\sigma_{ij}^{*}\fb{\mathrm{curv}({\Omega})})= -(\alpha_{g_{ij}})_{*}(B_i)+B_j+\mathrm{d}\varphi_{ij}+\frac{1}{2}[\varphi_{ij}\wedge \varphi_{ij}]+ \alpha_{*}(A_j \wedge \varphi_{ij})\text{;}
\end{equation*}
this is exactly the required cocycle condition \cref{eq:transcurvco}. 
\begin{comment}
Indeed we deduce  using \cref{eq:conform:c}:
\begin{align*}
0&=(\alpha_{g_{ij}})_{*}(\sigma_{ij}^{*}\fb{\mathrm{curv}(\Omega)})
\\&= (\alpha_{g_{ij}})_{*}(\sigma_{ij}^{*}t^{*}\fc\Omega)-(\alpha_{g_{ij}})_{*}(\sigma_{ij}^{*}s^{*}\fc\Omega) +(\alpha_{g_{ij}})_{*}(\mathrm{d}\sigma_{ij}^{*}\fb\Omega)+\frac{1}{2}[\varphi_{ij}\wedge \varphi_{ij}]+ \alpha_{*}(\mathrm{Ad}_{g_{ij}}(\sigma_{ij}^{*}s^{*}\fa\Omega) \wedge \varphi_{ij})
\\&= (\alpha_{g_{ij}})_{*}(s_i^{*}\fc\Omega)-(\alpha_{g_{ij}})_{*}(s_j,g_{ij})^{*}R^{*}\fc\Omega +(\alpha_{g_{ij}})_{*}(\mathrm{d}(\alpha_{g_{ij}^{-1}})_{*}(\varphi_{ij}))+\frac{1}{2}[\varphi_{ij}\wedge \varphi_{ij}]+ \alpha_{*}(\mathrm{Ad}_{g_{ij}}((s_j,g_{ij})^{*}R^{*}\fa\Omega) \wedge \varphi_{ij})
\\&= -(\alpha_{g_{ij}})_{*}(B_i)+B_j+\mathrm{d}\varphi_{ij}+\frac{1}{2}[\varphi_{ij}\wedge \varphi_{ij}]+ \alpha_{*}(A_j \wedge \varphi_{ij})
\end{align*}
This is \cref{eq:transcurvco}.
\end{comment}

\item
If $\tilde\Omega$ is fake-flat, then  $(g,a,A,B,\varphi)$ is fake-flat. Indeed,  $\mathrm{fcurv}(A_i,B_i) =s_i^{*}\fa{\mathrm{curv}({\tilde\Omega})}$. 

\item
The cocycle $(g,a,A,B,\varphi)$ is as a preimage of $\inf P$ under $r^{\gen}$. 

We consider the smooth $\Gamma$-equivariant functor  $\phi: \inf P_{(g,a)} \to \inf P$ of  \cref{sec:2bundleclass}, defined by
\begin{equation*}
\phi(i,x,g) := R(s_i(x),g)
\quand
\phi(i,j,x,h,g) := R(\sigma_{ij}(x),(\alpha(g_{ij}(x)^{-1},h),g_{ij}(x)^{-1}g))\text{.}
\end{equation*}
It is straightforward to check that $\phi^{*}{\tilde\Omega}=\Omega$ (where $\Omega$ is the connection reconstructed from $(A,B,\varphi)$ as described at the beginning of this subsection).  
\begin{comment}
We have 
\begin{align*}
\phi^{*}\fa{\tilde\Omega}|_{i,x,g} &= \fa{\tilde\Omega}_{R(s_i,g)}=\mathrm{Ad}_{g}^{-1}(s_i^{*}\fa{\tilde\Omega})+\theta_g=\mathrm{Ad}_g^{-1}(A_i)+ \theta_g=\fa\Omega
\\
\phi^{*}\fb{\tilde\Omega}|_{i,j,x,h,g} &= \fb{\tilde\Omega}_{R(\sigma_{ij},(\alpha(g_{ij}^{-1},h),g_{ij}^{-1}g))}
\\&=(\alpha_{g^{-1}g_{ij}})_{*}\left (\mathrm{Ad}_{\alpha(g_{ij}^{-1},h)}^{-1}(\sigma_{ij}^{*}\fb{\tilde\Omega})+(\tilde \alpha_{\alpha(g_{ij}^{-1},h)})_{*}(s^{*}\sigma_{ij}^{*}\fa{\tilde\Omega}) +\theta_{\alpha(g_{ij}^{-1},h)} \right)
\\&=(\alpha_{g^{-1}})_{*}\left (\mathrm{Ad}_{h}^{-1}(\alpha_{g_{ij}})_{*}(\sigma_{ij}^{*}\fb{\tilde\Omega})+(\tilde \alpha_{h})_{*}(\mathrm{Ad}_{g_{ij}}\fa{\tilde\Omega}_{R(s_j,g_{ij})}+\theta_{g_{ij}^{-1}}) +\theta_{h} \right)
\\&=(\alpha_{g^{-1}})_{*}\left (\mathrm{Ad}_{h}^{-1}(\varphi_{ij})+(\tilde \alpha_{h})_{*}(A_j) +\theta_{h} \right)
\\&=\fb\Omega
\\
\phi^{*}\fc{\tilde\Omega}|_{i,x,g} &= \fc{\tilde\Omega}_{R(s_i,g)}
= (\alpha_{g^{-1}})_{*}(s_i^{*}\fc{\tilde\Omega})=- (\alpha_{g^{-1}})_{*}(B_i)=\fc\Omega 
\end{align*}
\end{comment}
Hence, the canonical $\tilde\Omega$-pullback on the induced anafunctor is connection-preserving and by \cref{rem:smoothfunctorconnection}  connective. Thus, we have a 1-morphism in the bicategory $\zweibuncon \Gamma M$. \Cref{rem:smoothfunctorconnection} also guarantees that the canonical pullback is regular if $\tilde\Omega$ is regular, and fake-flat if $\tilde\Omega$ is fake-flat. Thus, we also obtain 1-morphisms in the bicategories $\zweibunconreg\Gamma M$ and $\zweibunconff\Gamma M$. By \cref{re:inv1morph} these are automatically 1-isomorphisms. This shows surjectivity of all three vertical maps.
 
\end{enumerate}
In order to prove injectivity of the vertical maps in our diagram, we consider two generalized differential $\Gamma$-cocycles $(g,a,A,B,\varphi)$ and $(g',a',A',B',\varphi')$ with respect to the same open cover, and the reconstructed principal $\Gamma$-2-bundles $\inf P:=\inf P_{(g,a)}$ and $\inf P':=\inf P_{(g',a')}$ equipped with the reconstructed connections $\Omega$ and $\Omega'$, respectively. We suppose that there is a 1-morphism $F: \inf P \to \inf P'$ equipped with a connection-preserving, connective pullback $\nu$. We have to show that the two cocycles are generalized equivalent. Following the treatment in \cref{sec:2bundleclass}, we have \splitting s $\sigma_{ij}(x):=(i,j,x,1,g_{ij}(x))$ of $\inf P$ with transition functions $g_{ij}$, and analogous \splitting s $\sigma_{ij}'$ of $\inf P'$ with transition functions $g_{ij}'$. 
 We choose \splitting s $\sigma_i:U_i \to F$ of $F$ along $(s_i,s_i'):U_i \to \ob{\inf P} \times_M \ob{\inf P'}$ with transition functions $h_i: U_i \to G$. In  \cref{sec:2bundleclass} we have then derived  \cref{eq:local:sub2},
\begin{equation}
\label{eq:localglobal1}
\sigma_{ij}\circ \rho(\sigma_j,\id_{g_{ij}}) = \rho(\sigma_i \circ R(\sigma'_{ij},\id_{h_i}),(\alpha(h_i^{-1}g_{ij}'^{-1},e_{ij}),h_i^{-1}g_{ij}'^{-1}h_jg_{ij}))\text{,}
\end{equation}
which is an equality between \splitting s of $F$, and proved that $(h,e)$ satisfies the equivalence conditions  of \cref{32c,32f}. In order to upgrade these results to the differential setting, we set $\phi_i := (\alpha_{h_i})_{*}(\sigma_i^{*}\nu_0)\in\Omega^1(U_i,\mathfrak{h})$ and pull back $\nu_0$ separately along both sides of \erf{eq:localglobal1}. We get, on the left hand side,
\begin{equation*}
(\alpha_{g_{ij}^{-1}h_j^{-1}})_{*}((\alpha_{h_j})_{*}(\varphi_{ij})+\phi_j)\text{,}
\end{equation*}
\begin{comment}
\begin{equation*}
\nu_0|_{\sigma_{ij}\circ \rho(\sigma_j,\id_{g_{ij}})} \eqcref{eq:conana:7} \fb\Omega_{\sigma_{ij}}+\nu_0|_{\rho(\sigma_j,\id_{g_{ij}})}\eqcref{eq:conana:b}(\alpha_{g_{ij}^{-1}})_{*}\left (\varphi_{ij}+\nu_0|_{\sigma_j} \right)=(\alpha_{g_{ij}^{-1}h_j^{-1}})_{*}((\alpha_{h_j})_{*}(\varphi_{ij})+\phi_j)\text{.} \end{equation*} 
\end{comment}
and, on the right hand side,
\begin{equation*}
(\alpha_{g_{ij}^{-1}h_j^{-1}})_{*}\left (\mathrm{Ad}_{e_{ij}}^{-1}((\alpha_{g_{ij}'})_{*}(\phi_i  )+ \varphi'_{ij})+(\tilde\alpha_{e_{ij}})_{*}(A_j')+\theta_{e_{ij}}\right)\text{.}
\end{equation*}
Thus, equality of both formulas gives the  equivalence condition of \cref{32d}.
This shows  that $(h,e,\phi)$ forms a generalized equivalence, which proves the injectivity of $r^{\gen}$. 

For the injectivity of $r$ we need to show that above equivalence is non-generalized, under the assumption that the pullback $\nu$ on $F$ is regular. With this assumption, we can assume in above constructions that the \splitting s $\sigma_i$ satisfy $\sigma_i^{*}\mathrm{fcurv}(\nu)=0$. It is straightforward to check that
\begin{equation*}
0=(\alpha_{h_i})_{*}(\sigma_i^{*}\mathrm{fcurv}(\nu))= B_i' + \alpha_{*}(A_i' \wedge \phi_i) + \mathrm{d}\phi_i + \frac{1}{2}[\phi_i \wedge \phi_i] - (\alpha_{h_i})_{*}(B_i)\text{.}
\end{equation*}
This is the additional equivalence condition of \cref{eq:equivtranscurv} that characterized non-generalized equivalence data.
\begin{comment}
Indeed,
\begin{align*}
0&=(\alpha_{h_i})_{*}(\mathrm{fcurv}(\nu)_{\sigma_i})  \\&=(\alpha_{h_i})_{*}\mathrm{d}\nu_0|_{\sigma_i}+ \frac{1}{2}[(\alpha_{h_i})_{*}\nu_0|_{\sigma_i}\wedge(\alpha_{h_i})_{*} \nu_0|_{\sigma_i}]+(\alpha_{h_i})_{*} \alpha_{*}(\fa{\Omega'}_{\alpha_r(\sigma_i)} \wedge \nu_0|_{\sigma_i})+(\alpha_{h_i})_{*}\nu_1|_{\sigma_i}
 \\&=\mathrm{d}\phi_i-\alpha_{*}(\bar\theta_{h_i} \wedge \phi_i)+ \frac{1}{2}[\phi_i\wedge\phi_i]+ \alpha_{*}(\mathrm{Ad}_{h_i}(\fa{\Omega'}_{R(s_i',h_i)}) \wedge \phi_i)+(\alpha_{h_i})_{*}(\fc\Omega_{s_i}-\fc{\Omega'}_{R(s_i',h_i)})
 \\&=\mathrm{d}\phi_i-\alpha_{*}(\bar\theta_{h_i} \wedge \phi_i)+ \frac{1}{2}[\phi_i\wedge\phi_i]+ \alpha_{*}(A_i'+\mathrm{Ad}_{h_i}(\theta_{h_i}) \wedge \phi_i)+(\alpha_{h_i})_{*}(-B_i)+B_i'
\end{align*}
\end{comment}
\end{proof}

\begin{corollary}
\label{co:pgen}
The map $p^{\gen}: \hat\h^1(M,\Gamma)^{\gen} \to \h^1(M,\Gamma)$ of \cref{sec:nonabdiffcoho} is surjective.
\end{corollary}

\begin{proof}
Let $(g,a)$ be a $\Gamma$-cocycle and $\inf P_{(g,a)}$ be the reconstructed principal $\Gamma$-2-bundle. Since the $G$-action on $\ob{\inf P_{(g,a)}}$ is essentially the multiplication of $G$ on itself, it is free and proper. 
\begin{comment}
Recall that
\begin{equation*}
X:=\ob{\inf P_{(g,a)}} = \coprod_{i\in I} U_i \times G
\quand
X \times G \to X:((i,x,g),g')\mapsto (i,x,gg')\text{.}
\end{equation*}
This action is clearly free. We write $X=Y\times G$, and look at the map
\begin{equation*}
Y \times G \times G \to (Y \times G) \times (Y \times G):(y,g,g')\mapsto ((y,g),(y,gg'))\text{.}
\end{equation*}
It is the composition of the diffeomorphism
\begin{equation*}
Y \times G \times G \to Y \times G \times G:(y,g,g')\mapsto (y,g,gg')
\end{equation*}
and (up to re-ordering factors) the diagonal map
\begin{equation*}
\Delta \times \id \times \id: Y \times G \times G \to Y \times Y \times G \times G\text{,}
\end{equation*}
which is an embedding and hence proper. Hence the action is proper. 
\end{comment}
By \cref{th:existence} there exists a connection on $\inf P_{(g,a)}$. Extracting local data yields a generalized differential $\Gamma$-cocycle with image equivalent to $(g,a)$.
\end{proof}

\begin{corollary}
\label{co:gbunredcon}
Suppose $f:M \to N$ is a smooth map with rank at most one, and $\inf P$ is a principal $\Gamma$-2-bundle with  (any/regular/fake-flat) connection $\Omega$ over $N$. Then, there exists a principal $G$-bundle $P$ with (some/some/flat) connection over $M$ and an  isomorphism $\act PH \cong f^{*}\inf P$ in $\zweibuncon \Gamma M$ / $\zweibunconreg\Gamma M$ / $\zweibunconff\Gamma M$.
\end{corollary}

\begin{proof}
The functors
\begin{equation*}
\buncon GM \to \zweibunconreg\Gamma M
\quand
\buncon GM^{\flat} \to \zweibunconff\Gamma M
\end{equation*}
of \cref{ex:gbunredcon} induce the maps $j$ and $j^{\flat}$ considered in \cref{sec:nonabdiffcoho} under the isomorphism  of \cref{th:classmain}. Then, the claim follows from \Cref{lem:redGbun}.
\begin{comment}
Actually, by looking closely into the proofs, the last bracket could be (regular/regular/flat).
\end{comment}
\end{proof}

\subsection{Local connection forms and gauge transformations}

\label{sec:local}

In this subsection we show that principal $\Gamma$-2-bundles with connection are locally modelled by so-called $\Gamma$-connections and gauge transformations. 
A \emph{$\Gamma$-connection} on $M$ is a pair $(A,B)$ consisting of a 1-form $A\in \Omega^1(M,\mathfrak{g})$ and a 2-form $B \in \Omega^2(M,\mathfrak{h})$. We define curvature and fake-curvature of a $\Gamma$-connection by
\begin{align*}
\mathrm{curv}(A,B) &:= \mathrm{d}B + \alpha_{*}( A\wedge B)\in\Omega^3
(M,\mathfrak{h})
\\
\mathrm{fcurv}(A,B) &:= \mathrm{d}A +\frac{1}{2} [A
\wedge A] - t_{*} ( B)\in\Omega^2(M,\mathfrak{g})
\end{align*}
and call it \emph{fake-flat}, if   $\mathrm{fcurv}(A,B)=0$, and \emph{flat}, if it is fake-flat and $\mathrm{curv}(A,B) =0$. There is a bijection
between the set of $\Gamma$-connections on $M$ and $\Omega^1(\idmorph{M},\gamma)$, under which $(A,B)$ corresponds to $\Psi \in \Omega^1(\idmorph{M},\gamma)$ with $\fa\Psi = A$, $\fb\Psi=0$, and $\fc\Psi=-B$. In particular, every $\Gamma$-connection $(A,B)$ defines a connection $\Omega_{A,B} := \Omega_{\Psi}$ on the trivial principal $\Gamma$-bundle $\inf I$, see \cref{ex:trivialconn}. Explicitly, we have
\begin{equation}
\label{eq:OmegaAB}
\fa\Omega_{A,B} = \mathrm{Ad}_{g}^{-1}(A) + g^{*}\theta  
\;\;\text{,}\quad
\fb\Omega_{A,B} =(\alpha_{g^{-1}})_{*} ((\tilde\alpha_{h})_{*}(A)+h^{*}\theta) \;\;\text{,}\quad
\fc\Omega_{A,B} =-(\alpha_{g^{-1}})_{*}(B)\;\; \end{equation}
where $g$ and $h$ denote the projections. Recasting the results of \cref{ex:trivialconn} in the language of $\Gamma$-connections, we obtain the following.

\begin{lemma}
\label{lem:conntrivbun}
The assignment $(A,B) \mapsto \Omega_{A,B}$ establishes a bijection between the set of $\Gamma$-connections on $M$ and the set of connections on the trivial principal $\Gamma$-bundle  $\inf I$. Moreover, the conditions \quot{fake-flat} and \quot{flat} are invariant under this bijection.
\end{lemma}

We will also use the notation $\inf I_{A,B}$ for the trivial principal $\Gamma$-bundle equipped with the connection $\Omega_{A,B}$. 
A \emph{gauge transformation} between $\Gamma$-connections $(A,B)$ and $(A',B')$ on $M$ is a pair $(g,\varphi)$ of a smooth map $g:M \to G$ and a 1-form $\varphi \in \Omega^1(M,\mathfrak{h})$ such that
\begin{align}
\label{eq:coder:gt1}
\mathrm{Ad}_{g}(A) - g^{*}\bar \theta &=A'+t_{*}(\varphi)
\\
\label{eq:coder:gt2}
B'+\mathrm{d}\varphi+\frac{1}{2}[\varphi\wedge \varphi]+ \alpha_{*}(A' \wedge \varphi)&=(\alpha_{g})_{*}(B)\text{.}
\end{align}  
We define a smooth, $\Gamma$-equivariant functor
$\phi_g:\inf I \to \inf I$  by setting
\begin{equation*}
\phi_g(x,g'):=(x,g(x)g')
\quand
\phi_g(\id_x,h,g'):=(\id_x,\alpha(g(x),h),g(x)g')\text{.}
\end{equation*}
\begin{comment}
Abstractly, this is:
\begin{equation*}
\phi_g := (\pr_M, m \circ (\tilde g \times \id_{\Gamma}))  : \idmorph{M} \times \Gamma \to \idmorph{M} \times \Gamma\text{,}
\end{equation*}
where $\tilde g:\idmorph{M} \to \Gamma$ is the obvious extension of $g$ and $m:\Gamma \times \Gamma \to \Gamma$ is the multiplication.
\end{comment}
The definition of $\phi_g$  implies the composition law  $\phi_{g_2g_1}=\phi_{g_2}\circ \phi_{g_1}$. By \cref{rem:equivsmoothfunctor} $\phi_g$ is a 1-morphism in $\zweibun\Gamma M$.  
We promote $\phi_g$  to a setting with connections in the following way. We denote by $\Omega := \Omega_{A,B}$ and $\Omega' :=\Omega_{A',B'}$
the connections on $\inf I$ corresponding to the $\Gamma$-connections $(A,B)$ and $(A',B')$.  It is easy to see  that in general $\phi_g^{*}\Omega'\neq \Omega$. More precisely, we have
\begin{align*}
\phi_g^{*}\fa{\Omega'} &= \fa\Omega - \mathrm{Ad}_{g\cdot \pr_G}^{-1}(t_{*}(\pr_M^{*}\varphi))
\\
\phi_g^{*}\fb{\Omega'} &= \fb\Omega-(\alpha_{\pr_G^{-1}\cdot t(\pr_H)^{-1}\cdot g^{-1}})_{*}  (\pr_M^{*}\varphi))+(\alpha_{\pr_G^{-1}\cdot g^{-1}})_{*}  (\pr_M^{*}\varphi)
\\
\phi_g^{*}\fc{\Omega'} &= \fc\Omega+(\alpha_{\pr_G^{-1}\cdot g^{-1}})_{*}\pr_M^{*}(\alpha_{*}(A' \wedge \varphi) + \mathrm{d}\varphi + \frac{1}{2}[\varphi \wedge \varphi])\text{.}
\end{align*}
In other words, if $F_g$ is the anafunctor associated to the smooth functor $\phi_g$, then the canonical $\Omega'$-pullback $\nu$ on $F_g$ is \emph{not} connection-preserving.   However, using the affine space structure of \cref{lem:pullexists}, we shift $\nu$ to a new $\Omega'$-pullback $\nu^{\varphi}$ by a pair $(\varphi_0,\varphi_1)$ of forms $\varphi_0\in\Omega^1(M\times G,\mathfrak{h})$ and $\varphi_1\in\Omega^2(M\times G,\mathfrak{h})$, which we obtain from the given 1-form $\varphi$ by the following formulae:
\begin{align}
\label{eq:shiftF:0}
\varphi_0 &:=(\alpha_{\pr_G^{-1}\cdot g^{-1}})_{*}(\pr_M^{*}\varphi)
\\
\label{eq:shiftF:1}
\varphi_1 &:=-(\alpha_{\pr_G^{-1} \cdot g^{-1}})_{*}\pr_M^{*}(\mathrm{d}\varphi+\frac{1}{2}[\varphi\wedge\varphi]+\alpha_{*}(A' \wedge \varphi))\text{.}
\end{align}

\begin{lemma}
\label{lem:gtmorph}
The shifted $\Omega'$-pullback $\nu^{\varphi}$ on $F_g$ has the following properties:
\begin{enumerate}[(a)]

\item 
\label{lem:gtmorph:a}
It is connection-preserving, i.e. $(F_g)^{*}_{\nu^{\varphi}}\Omega'=\Omega$.

\item
\label{lem:gtmorph:b}
It is connective.

\item
\label{lem:gtmorph:c}
It is regular.

\item
\label{lem:gtmorph:d}
It is fake-flat if $(A,B)$ and $(A',B')$ are fake-flat.

\end{enumerate}
\end{lemma}

\begin{proof}
In order to compute $(F_g)^{*}_{\nu^{\varphi}}\Omega'$ we use above calculation of $\phi_g^{*}\Omega'$ and then incorporate the
shift using \cref{re:canpullback:c}; this yields \cref{lem:gtmorph:a*}.
\begin{comment}
Indeed,
\begin{align*}
\fa {((F_g)_{\nu^{\varphi}}^{*}\Omega')} &= \phi_g^{*}\fa{\Omega'} + t_{*}(\varphi_0)
\\&=\fa\Omega-\mathrm{Ad}_{g\pr_G}^{-1}(t_{*}(\varphi))+t_{*}((\alpha_{(g\pr_G)^{-1}})_{*}(\varphi))
\\&=\fa\Omega
\\
\fb {((F_g)_{\nu^{\varphi}}^{*}\Omega')} &= \phi_g^{*}\fb{\Omega'} + \Delta\varphi_0
\\&= \fb\Omega-(\alpha_{\pr_G^{-1}\cdot t(\pr_H)^{-1}\cdot g^{-1}})_{*}  (\pr_M^{*}\varphi))+(\alpha_{\pr_G^{-1}\cdot g^{-1}})_{*}  (\pr_M^{*}\varphi)
\\&\qquad+(\alpha_{\pr_G^{-1}   \cdot t(\pr_H)^{-1}\cdot  g^{-1}})_{*}(\pr_M^{*}\varphi)-(\alpha_{\pr_G^{-1}\cdot g^{-1}})_{*}(\pr_M^{*}\varphi)
\\&=\fb\Omega 
\\
\fc {((F_g)_{\nu^{\varphi}}^{*}\Omega')} &= \phi_g^{*}\fc{\Omega'} +\varphi_1
\\&= \fc\Omega
\end{align*}
\end{comment}
For \cref{lem:gtmorph:b*} we note that the forms $\varphi_0$ and $\varphi_1$ are  $R$-invariant; hence $\nu^{\varphi}$ is connective by \cref{re:canpullback:shifted:connective}. Claims \cref{lem:gtmorph:c*,lem:gtmorph:d*} follow from \cref{re:canpullback:shifted:fc} by verifying \cref{eq:canpullback:shifted},
\begin{equation*}
\mathrm{d}\varphi_0+ \frac{1}{2}[\varphi_0\wedge \varphi_0]+ \alpha_{*}(\phi_g^{*}\fa{\Omega'} \wedge \varphi_0)+\varphi_1=0\text{;}
\end{equation*}
this is straightforward. 
\begin{comment}
Indeed,
\begin{align*}
&\mquad  -\mathrm{d}\varphi_0|_{(x,g')}- \frac{1}{2}[\varphi_0|_{(x,g')}\wedge \varphi_0|_{(x,g')}]- \alpha_{*}(\fa{\Omega'}|_{(x,g(x)g')}\wedge \varphi_0|_{(x,g')})
\\&= -\mathrm{d}(\alpha_{g'^{-1}g(x)^{-1}})_{*}(\varphi_x)- \frac{1}{2}[(\alpha_{g'^{-1}g(x)^{-1}})_{*}(\varphi_x)\wedge(\alpha_{g'^{-1}g(x)^{-1}})_{*}(\varphi_x)]
\\&\qquad- \alpha_{*}((\mathrm{Ad}_{g(x)g'}^{-1}(A'_x) + \theta_{g(x)g'} )\wedge (\alpha_{g'^{-1}g(x)^{-1}})_{*}(\varphi_x))
\\&= -(\alpha_{g'^{-1}g(x)^{-1}})_{*}(\mathrm{d}\varphi+\frac{1}{2}[\varphi_x\wedge\varphi_x])-\alpha_{*}(\bar\theta_{g'^{-1}g(x)^{-1}} \wedge (\alpha_{g'^{-1}g(x)^{-1}})_{*}(\varphi_x))
\\&\qquad- \alpha_{*}((\mathrm{Ad}_{g(x)g'}^{-1}(A_x') + \theta_{g(x)g'} )\wedge (\alpha_{g'^{-1}g(x)^{-1}})_{*}(\varphi_x))
\\&= -(\alpha_{g'^{-1}g(x)^{-1}})_{*}(\mathrm{d}\varphi+\frac{1}{2}[\varphi_x\wedge\varphi_x]+\alpha_{*}(A_x' \wedge \varphi_x))
\\&=\varphi_1
\end{align*}
\end{comment}
 \end{proof}

We use the notation $F_{g,\varphi}$ for the anafunctor associated to $\phi_g$ equipped with the $\Omega'$-pullback $\nu^{\varphi}$. By \cref{lem:gtmorph}, this is a 1-morphism \begin{equation*}
F_{g,\varphi}:\inf I_{A,B} \to \inf I_{A',B'}
\end{equation*}
in $\zweibunconreg\Gamma M$, and in case of fake-flat $\Gamma$-connections  in $\zweibunconff\Gamma M$.
Next we discuss the situation that we have  $\Gamma$-connections $(A,B)$ and $(A',B')$, and \emph{two} gauge transformations $(g_1,\varphi_1)$ and $(g_2,\varphi_2)$. A \emph{gauge 2-transformation} is a smooth map $a:M \to H$ such that 
\begin{equation}
\label{eq:gauge2}
g_2 = t (a)  g_1
\quad\text{ and }\quad
\mathrm{Ad}_a^{-1} (\varphi_2) +(\tilde\alpha_{a})_{*}(A') =  \varphi_1 -a^{*}\theta\text{.}
\end{equation}
We define the smooth map
\begin{equation}
\label{eq:def:etaa}
\eta_a: \ob{\inf I} \to \mor{\inf I}: (x,g) \mapsto (\id_x,(a(x),g_1(x)g))\text{.}
\end{equation}

\begin{lemma}
\label{lem:2gtmorph}
This is a smooth $\Gamma$-equivariant natural transformation $\eta_a:\phi_{g_1}\Rightarrow \phi_{g_2}$, and induces a 2-morphism $\eta_a:F_{g_1,\varphi_1}\Rightarrow F_{g_2,\varphi_2}$ in $\zweibunconreg\Gamma M$.
\end{lemma}

\begin{proof}
Source and target conditions as well as naturality are straightforward to check.
\begin{comment}
We have 
\begin{equation*}
s(\eta_a(x,g))=s(\id_x,(a(x),g_1(x)g))=(x,g_1(x)g)=\phi_{g_1}(x,g)
\end{equation*}
and
\begin{equation*}
t(\eta_a(x,g))=t(\id_x,(a(x),g_1(x)g))=(x,t(a(x))g_1(x)g)=\phi_{g_2}(x,g)
\end{equation*}
For naturality, we consider a morphism $(\id_x,h,g):(x,g) \mapsto (x,t(h)g)$ and check
\begin{align*}
\eta_a(x,t(h)g)\circ \phi_{g_1}(\id_x,h,g)
&= (\id_x,a(x),g_1(x)t(h)g)\circ (\id_x,\alpha(g_1(x),h),g_1(x)g) \\&= (\id_x,a(x)\alpha(g_1(x),h),g_1(x)g)
\\&=  (\id_x, \alpha(t(a(x))g_1(x),h)a(x),g_1(x)g)
\\&= (\id_x, \alpha(g_2(x),h)a(x),g_1(x)g)
\\&= (\id_x,\alpha(g_2(x),h),g_2(x)g)\circ(\id_x,a(x),g_1(x)g) 
\\&= \phi_{g_2}(\id_x,h,g) \circ \eta_a(x,g)\text{.}
\end{align*}
\end{comment}
For $\Gamma$-equivariance, we check according to \cref{rem:equivsmoothfunctor}
\begin{equation*}
\eta_a(x,gg')=(\id_x,(a(x),g_1(x)gg'))=(\id_x,(a(x),g_1(x)g)\cdot(1,g'))=R(\eta_a(x,g),\id_{g'})\text{.}
\end{equation*}
Concerning the connections, we use \cref{re:canpullback:d} and hence have to verify the conditions  $\eta_a^{*}\fb{\Omega'}=\kappa_{1,0} - \kappa_{2,0}$ and $\eta_a^{*}\Delta\fc{\Omega'}=\kappa_{1,1}-\kappa_{2,1}$, where $\kappa_i=(\kappa_{i,0},\kappa_{i,1})$ are the shifts induced from $\varphi_i$ via \cref{eq:shiftF:0,eq:shiftF:1}. The first condition follows immediately from the definitions,
\begin{comment}
Indeed,
\begin{align*}
(\eta_a^{*}\fb\Omega_{A',B'}+\kappa_{2,0})_{x,g}
&=\fb\Omega_{A',B'}|_{\id_x,a(x),g_1(x)g)}+(\alpha_{g^{-1} g_2(x)^{-1}})_{*}(\varphi_2|_x)
\\&=(\alpha_{g^{-1}g_1(x)^{-1}})_{*} ((\tilde\alpha_{a(x)})_{*}(A_x')+\theta_{a(x)}) +(\alpha_{g^{-1} g_2(x)^{-1}})_{*}(\varphi_2|_x)
\\&=(\alpha_{g^{-1}g_1(x)^{-1}})_{*} ((\tilde\alpha_{a(x)})_{*}(A_x')+\theta_{a(x)} +\mathrm{Ad}^{-1}_{a(x)}(\varphi_2|_x))
\\&= (\alpha_{g^{-1} g_1(x)^{-1}})_{*}(\varphi_1|_x)
\\&= \kappa_{1,0}|_{x,g}
\end{align*}
\end{comment}
and the second conditions can be checked using  \cref{eq:coder:gt2}.
\begin{comment}
Indeed,
\begin{align*}
\kappa_{1,1}|_{x,g} 
&= -(\alpha_{g^{-1} \cdot g_1(x)^{-1}})_{*}(\mathrm{d}\varphi_1|_x+\frac{1}{2}[\varphi_1|_x\wedge\varphi_1|_x]+\alpha_{*}(A_x'
 \wedge \varphi_1|_x)) 
\\&\eqcref{eq:coder:gt2} (\alpha_{g^{-1} \cdot g_1(x)^{-1}})_{*}(B'_x-(\alpha_{g_1(x)})(B_x))
\\&= (\alpha_{g^{-1} \cdot g_1(x)^{-1}})_{*}(B'_x)-(\alpha_{g^{-1}})(B_x) 
\\&= - (\alpha_{g^{-1}g_2(x)^{-1}})_{*}(B') +(\alpha_{g^{-1}g_1(x)^{-1}})_{*}(B')+(\alpha_{g^{-1} \cdot g_2(x)^{-1}})_{*}(B'_x)-(\alpha_{g^{-1}})(B_x)
\\&\eqcref{eq:OmegaAB} \fc{\Omega'}_{(x,t(a(x))g_1(x)g)} -\fc{\Omega'}_{(x,g_1(x)g)}+(\alpha_{g^{-1} \cdot g_2(x)^{-1}})_{*}(B'_x-(\alpha_{g_2(x)})(B_x))
\\&\eqcref{eq:coder:gt2} (\Delta\fc{\Omega'})_{(\id_x,(a(x),g_1(x)g))} -(\alpha_{g^{-1} \cdot g_2(x)^{-1}})_{*}(\mathrm{d}\varphi_2|_x+\frac{1}{2}[\varphi_2|_x\wedge\varphi_2|_x]+\alpha_{*}(A_x'
 \wedge \varphi_2|_x))
\\&= (\Delta\fc{\Omega'})_{\eta_a(x,g)} + \kappa_{2,1}|_{x,g}
\end{align*}
\end{comment}
\end{proof}

$\Gamma$-connections, gauge transformations, and gauge 2-transformation form  bicategories $\con\Gamma M$ and $\conff\Gamma M$, see \cite{schreiber5}. The assignments $M \mapsto \con\Gamma M$ and $M\mapsto \conff\Gamma M$ form presheaves of bicategories.

\begin{proposition}
\label{prop:gt}
The assignments 
\begin{equation*}
(A,B) \mapsto \inf I_{A,B}
\quomma
(g,\varphi)\mapsto F_{g,\varphi}
\quand
a\mapsto \eta_a
\end{equation*}
form strict 2-functors
\begin{equation*}
L^{\ff}: \conff \Gamma M \to \zweibunconff\Gamma M
\quand
L:\con \Gamma M \to \zweibunconreg\Gamma M\text{.}
\end{equation*}
Moreover,  $L^{\ff}$ and $L$ are natural in $M$; i.e., they form morphisms of presheaves. 
\end{proposition}

\begin{proof}
Well-definedness is the content of \cref{lem:gtmorph,lem:2gtmorph}. It remains to check that the composition of gauge transformations, and horizontal and vertical composition of gauge 2-transformations are respected. This is straightforward and left as an exercise.
\begin{comment}
Vertical and horizontal composition are produced as indicated in the following diagrams:
\begin{equation*}
\alxydim{@C=1.8cm}{\inf I \ar@/^2pc/[r]^{\phi_{g}}="1" \ar[r]|{\phi_{g'}}="2" \ar@/_2pc/[r]_{\phi_{g''}}="3" \ar@{=>}"1";"2"|-{\eta_{a}}\ar@{=>}"2";"3"|-{\eta_{a'}} & \inf I}
=\alxydim{@C=1.8cm}{\inf I \ar@/^2pc/[r]^{\phi_{g}}="1"  \ar@/_2pc/[r]_{\phi_{g''}}="3" \ar@{=>}"1";"3"|-{\eta_{a'a}} & \inf I}
\quand
\alxydim{@C=1.4cm}{\inf I \ar@/^2pc/[r]^{\phi_{g_1}}="1" \ar@/_2pc/[r]_{\phi_{g_1'}}="2" \ar@{=>}"1";"2"|{\eta_{a_1}} & \inf I \ar@/^2pc/[r]^{\phi_{g_2}}="1" \ar@/_2pc/[r]_{\phi_{g_2'}}="2" \ar@{=>}"1";"2"|{\eta_{a_2}} & \inf I}
=
\alxydim{@C=1.8cm}{\inf I \ar@/^2pc/[r]^{\phi_{g_2g_1}}="1" \ar@/_2pc/[r]_{\phi_{g_2'g_1'}}="2" \ar@{=>}"1";"2"|{\eta_{a_2\alpha(g_2,a_1)}} & \inf I\text{.}}
\end{equation*}
\end{comment}
\end{proof}

\begin{comment}
\begin{remark}
They are compatible with the inclusion of fake-flat connections into regular ones; i.e., the diagram
\begin{equation*}
\alxydim{}{\conff\Gamma - \ar[r]^-{L^{\ff}} \ar@{^(->}[d] & \zweibunconff\Gamma - \ar@{^(->}[d] \\\con\Gamma -\ar[r]_-{L} & \zweibunconreg\Gamma -\text{.}}
\end{equation*}
is commutative.
\end{remark}
\end{comment}

In the remainder of this subsection we prove two statements about the 2-functors $L$ and $L^{\ff}$, see \cref{lem:Locesssurj,prop:locfullyfaithful}. We recall that a \emph{section} in a principal $\Gamma$-2-bundle $\inf P$ is a smooth map $s:M \to \ob{\inf P}$ such that $\pi \circ s = \id_M$. Since $\pi:\ob{\inf P} \to M$ is a surjective submersion, every point $x\in M$ has an open neighborhood $U\subset M$ that supports a  section. Associated to a  section  $s$ is a smooth $\Gamma$-equivariant functor  $T: \inf I \to \inf P$ defined by
\begin{equation*}
T:= R \circ (s\times \id_{\Gamma}): \idmorph{M} \times \Gamma \to \inf P\text{,}
\end{equation*}
called the  \emph{trivialization} associated to $s$. We equip $T$ with the canonical $\Omega$-pullback, which is connective by \cref{rem:smoothfunctorconnection}. Thus, $\Psi := T^{*}\Omega$ is a connection on $\inf I$ by \cref{re:pullbackcond:a}, and thus $\Psi = \Omega_{A,B}$ for a unique $\Gamma$-connection $(A,B)$ on $M$; see \cref{lem:conntrivbun}. By definition of $\Psi$,  $T$ is connection-preserving.
If $\Omega$ is regular, then the canonical $\Omega$-pullback is regular (\cref{re:canpullback:reg}), and if $\Omega$  is fake-flat, then $\Psi$ is fake-flat (\cref{re:pullbackcond:c}), $(A,B)$  is fake-flat (\cref{lem:conntrivbun}), and the canonical $\Omega$-pullback is fake-flat (\cref{re:canpullback:fake}). This shows the following result.

\begin{proposition}
\label{lem:Locesssurj}
Let $\inf P$ be a principal $\Gamma$-2-bundle over $M$ with connection $\Omega$. For every point $x\in M$ there exists  an open neighborhood $U \subset M$ and a
\begin{enumerate}[(a)]

\item 
\label{lem:Locesssurj:a}
$\Gamma$-connection $(A,B)$ on $U$ such that $\inf I_{A,B} \cong \inf P|_U$ in 
$\zweibuncon \Gamma U$.

\item 
\label{lem:Locesssurj:b}
$\Gamma$-connection $(A,B)$ on $U$ such that $\inf I_{A,B} \cong \inf P|_U$ in 
$\zweibunconreg \Gamma U$, if   $\Omega$ is regular.

\item 
\label{lem:Locesssurj:c}
 fake-flat $\Gamma$-connection $(A,B)$ on $U$ such that $\inf I_{A,B} \cong \inf P|_U$ in 
$\zweibunconff \Gamma U$, if  $\Omega$ is fake-flat.

\end{enumerate} 
\end{proposition}

\begin{remark}
\cref{lem:Locesssurj:b,lem:Locesssurj:c} imply that the sheaf morphisms $L$ and $L^{\ff}$ are surjective, i.e., essentially surjective on stalks. As 2-functors, $L$ and $L^{\ff}$ are in general not essentially surjective, regardless to which kind of subset $U \subset M$ one restricts them.
\begin{comment}
Since principal $\Gamma$-2-bundles have a classifying theory, every principal $\Gamma$-2-bundle trivializes over every contractible open set $U \subset M$. Suppose $J:\inf I \to \inf P|_U$ is such a trivialization. It induces an anafunctor $\idmorph{U} \to \inf P|_U$, i.e. a principal $\inf P|_U$-bundle over $U$. There is no reason why this should be trivializable, even if $U$ is contractible. 

More concretely, consider an abelian bundle gerbe for an open cover. The associated principal $\Gamma$-2-bundle has $\ob{\inf P}$ the disjoint union of the open sets. Obviously, not every contractible open set $U \subset M$ will have a smooth section. 
\end{comment}
\end{remark}

Suppose $F: \inf I_{A,B} \to \inf I_{A',B'}$ is a 1-morphism in $\zweibuncon\Gamma M$, and  denote by $\nu=(\nu_0,\nu_1)$ its  $\Omega_{A',B'}$-pullback. Consider the smooth map  
\begin{equation*}
s: M \to \ob{\inf I_{A,B}} \times_M \ob{\inf I_{A',B'}}:x \mapsto ((x,1),(x,1))\text{.}
\end{equation*}
We assume that $F$ admits a transition span $\sigma: M \to F$ along $s$ with $\sigma^{*}\mathrm{fcurv}(\nu)=0$.  Let $g:M \to G$ be a transition functor for $\sigma$, and let $\varphi := (\alpha_g)_{*}(\sigma^{*}\nu_0) \in \Omega^1(M,\mathfrak{g})$.

\begin{lemma}
\label{lem:gt}
The pair $(g,\varphi)$ is a gauge transformation between $(A,B)$ and $(A',B')$. Moreover, there exists a connection-preserving 2-morphism $F_{g,\varphi} \cong F$. 
\end{lemma}

\begin{proof}
We have to verify \cref{eq:coder:gt1,eq:coder:gt2}. The first is a direct calculation using \cref{eq:conana:a}.
\begin{comment}
Indeed,
\begin{align*}
t_{*}(\varphi) &=\mathrm{Ad}_g(t_{*}(\sigma^{*}\nu_0))
\\&\eqcref{eq:conana:a} \mathrm{Ad}_g(\sigma^{*}\alpha_l^{*}\fa\Omega_{A,B}-\sigma^{*}\alpha_r^{*}\fa\Omega_{A',B'})
\\&=\mathrm{Ad}_g((\id,1)^{*}\fa\Omega_{A,B}-(\id,g)^{*}\fa\Omega_{A',B'})
\\&\eqcref{eq:OmegaAB}\mathrm{Ad}_g(A-\mathrm{Ad}_{g}^{-1}(A') - g^{*}\theta)
\\&= \mathrm{Ad}_{g}(A) -A'- g^{*}\bar \theta\text{.} 
\end{align*}
\end{comment}
For the second, it is straightforward to deduce from the definition of $\mathrm{fcurv}(\nu)$ and \cref{eq:conana:a,eq:conana:6} that
\begin{equation*}
(\alpha_{g})_{*}(\sigma^{*}\mathrm{fcurv}({\nu}))= -(\alpha_{g})_{*}(B)+B'+\mathrm{d}\varphi+ \frac{1}{2}[\varphi\wedge \varphi]+\alpha_{*}(A' \wedge \varphi)\text{.}
\end{equation*}
\begin{comment}
Indeed,
\begin{align*}
&\mquad(\alpha_{g})_{*}(\sigma^{*}\mathrm{fcurv}({\nu}))
\\&= (\alpha_g)_{*}\sigma^{*}(\mathrm{d}\nu_0+ \frac{1}{2}[\nu_0\wedge \nu_0]+ \alpha_{*}(\alpha_r^{*}\fa{\Omega}_{A',B'} \wedge \nu_0)+\nu_1)
\\&\eqcref{eq:conana:a,eq:conana:6} (\alpha_g)_{*}(\sigma^{*}\mathrm{d}\nu_0)+ \frac{1}{2}[\varphi\wedge \varphi]+(\alpha_g)_{*} \alpha_{*}(\sigma^{*}(\alpha_l^{*}\fa\Omega_{A,B} - t_{*}(\nu_0)) \wedge\sigma^{*} \nu_0)
\\&\qquad+(\alpha_g)_{*}(\sigma^{*}\alpha_l^{*}\fc\Omega_{A,B}-\sigma^{*}\alpha_r^{*}\fc{\Omega}_{A',B'})
\\&\eqcref{eq:OmegaAB} \mathrm{d}\varphi-\alpha_{*}(g^{*}\bar\theta \wedge \varphi)+ \frac{1}{2}[\varphi\wedge \varphi]+\alpha_{*}(\mathrm{Ad}_g(A)-t_{*}(\varphi)) \wedge\varphi)+(\alpha_g)_{*}(-B+(\alpha_{g^{-1}})_{*}(B'))
\\&= \mathrm{d}\varphi+ \frac{1}{2}[\varphi\wedge \varphi]+\alpha_{*}(A' \wedge\varphi)-(\alpha_g)_{*}(B)+B'
\end{align*}
\end{comment}
Thus, the assumption
$\sigma^{*}\mathrm{fcurv}(\nu)=0$ implies \cref{eq:coder:gt2}. Next consider the smooth map
$\tilde \kappa: M \times G \to F$ defined by $\tilde\kappa(x,g') := \sigma(x) \cdot \id_{g'}$. It is easy to check that it satisfies the conditions of \cref{rem:smoothfunctor:b}, namely
\begin{equation*}
\alpha_l(\tilde\kappa(x,g')) = (x,g')
\quomma
\alpha_r ( \tilde\kappa(x,g')) =  \phi_g(x,g')
\quomma
\alpha\circ \tilde\kappa(x,g')\circ \beta=\tilde\kappa(t(\alpha))\circ \phi_g(\alpha)\circ \beta\text{,}
\end{equation*}
where $x \in M$, $g'\in G$ and $\alpha,\beta\in M \times \mor{\Gamma}$. 
\begin{comment}
\begin{align*}
\alpha_l(\tilde\kappa(x,g')) &=\alpha_l(\sigma(x) \cdot \id_{g'})=R((x,1),g') = (x,g')
\\
\alpha_r ( \tilde\kappa(x,g')) &=\alpha_r (\sigma(x) \cdot \id_{g'}) =R(R((x,1),g(x)),g')=(x,g(x)g')=  \phi_g(x,g')
\\
\alpha\circ \tilde\kappa(x,g')\circ \beta &=R((x,\id_1),(h_{\alpha},g_{\alpha}))\circ (\sigma(x) \cdot \id_{g'})\circ R((x,\id_{g(x)}),(\alpha(g(x)^{-1},h_{\beta}),g(x)^{-1}g_{\beta}))
\\&= ((x,\id_1) \circ\sigma(x)\circ(x,g(x))  ) \cdot ((h_{\alpha},g_{\alpha}) \circ  \id_{g'}\circ(\alpha(g(x)^{-1},h_{\beta}),g(x)^{-1}g_{\beta}) )
\\&= \sigma(x) \cdot(h_{\alpha}\alpha(g(x)^{-1},h_{\beta}),g(x)^{-1}g_{\beta} ) 
\\&=(\sigma(x) \circ(x,\id_{g(x)}) ) \cdot (\id_{t(h_{\alpha})g'} \circ(h_{\alpha}\alpha(g(x)^{-1},h_{\beta}),g(x)^{-1}g_{\beta}) )
 \\&=(\sigma(x) \cdot \id_{t(h_{\alpha})g'})\circ R((x,\id_{g(x)}),(h_{\alpha}\alpha(g(x)^{-1},h_{\beta}),g(x)^{-1}g_{\beta})) \\&=(\sigma(x) \cdot \id_{t(h_{\alpha})g'})\circ (x,\alpha(g(x),h_{\alpha})h_{\beta},g_{\beta}) 
\\&=(\sigma(x) \cdot \id_{t(h_{\alpha})g'})\circ \phi_g(x,h_{\alpha},g_{\alpha})\circ (x,h_{\beta},g_{\beta}) 
\\&=\tilde\kappa(t(\alpha))\circ \phi_g(\alpha)\circ \beta 
\end{align*}
where we have written $\alpha=(x,h_{\alpha},g_{\alpha})$ and $\beta=(x,h_{\beta},g_{\beta})$ with 
\begin{equation*}
(x,t(h_{\beta})g_{\beta})=t(\beta)=\alpha_r ( \tilde\kappa(x,g'))=  \phi_g(x,g')=(x,g(x)g')
\end{equation*}
i.e. $t(h_{\beta})g_{\beta} = g(x)g'$, and
\begin{equation*}
(x,g_{\alpha})=s(\alpha)=\alpha_l(\tilde\kappa(x,g')) = (x,g')
\end{equation*}
i.e. $g'=g_{\alpha}$. 
\end{comment}
Hence it induces a transformation $\kappa: F_{g,\varphi} \Rightarrow F$. According to \cref{rem:equivsmoothfunctor}, $\kappa$ is $\Gamma$-equivariant if $\tilde\kappa$ satisfies
$\tilde\kappa(x,g'g''))=\tilde\kappa(x,g')\cdot \id_{g''}$; this follows immediately from its definition. 
\begin{comment}
Indeed,
\begin{equation*}
\tilde\kappa(x,g'g''))=\sigma(x) \cdot \id_{g'g''}= \sigma(x) \cdot \id_{g'}\cdot \id_{g''}=\tilde\kappa(x,g')\cdot \id_{g''}\text{.}
\end{equation*}
\end{comment}
This shows that $\kappa$ is a 2-morphism $\kappa: F_{g,\varphi} \Rightarrow F$. It remains to check that it is connection-preserving. According to \cref{re:canpullback:d} this is the case if the given pullback $\nu$ on $F$ pulls back along $\tilde\kappa$ to the shift $\varphi$ of the canonical pullback on $F_{g}$; i.e.
\begin{equation*}
\tilde\kappa^{*}\nu_0=\varphi_0
\quand
\tilde\kappa^{*}\nu_1=\varphi_1\text{,}
\end{equation*}
where $\varphi_0$ and $\varphi_1$ are defined in \cref{eq:shiftF:0,eq:shiftF:1}. The first equation follows from just the definitions and the fact that $\nu$ is connective. 
\begin{comment}
\begin{multline*}
\tilde\kappa^{*}\nu_0|_{x,g'}=\nu_0|_{\tilde\kappa(x,g')}=\nu_0|_{\tilde\kappa(x,g')}=\nu_0|_{\sigma(x)\cdot \id_{g'}}\eqcref{eq:conana:b}(\alpha_{g'^{-1}})_{*}(\nu_0|_{\sigma(x)})
\\=(\alpha_{g'^{-1}g(x)^{-1}})_{*}( (\alpha_{g(x)})_{*}(\nu_0|_{\sigma(x)}))=(\alpha_{g'^{-1}g(x)^{-1}})_{*}(\varphi_x)=\varphi_0|_{x,g'}
\end{multline*}
\end{comment}
For the second equation one shows first using connectivity of $\nu$ and the vanishing of $\sigma^{*}\mathrm{fcurv}(\nu)$ that
\begin{equation*}
\tilde\kappa^{*}\nu_1 =-(\alpha_{\pr_G^{-1}})_{*}\left (\sigma^{*}(\mathrm{d}\nu_0+ \frac{1}{2}[\nu_0\wedge \nu_0]+ \alpha_{*}(\mathrm{Ad}_{g}^{-1}(A') \wedge \nu_0)+ \alpha_{*}(g^{*}\theta \wedge \sigma^{*}\nu_0) \right )\text{.}
\end{equation*}
From there it is straightforward to deduce \cref{eq:shiftF:1}. 
\begin{comment}
\begin{align*}
\tilde\kappa^{*}\nu_1|_{x,g'}
&=\nu_1|_{\tilde\kappa(x,g')}
\\&=\nu_1|_{\sigma(x) \cdot \id_{g'}}
\\&\eqcref{eq:conana:mu} (\alpha_{g'^{-1}})_{*}(\nu_1|_{\sigma(x)})
\\&\eqtext{$\sigma^{*}\mathrm{fcurv}(\nu)=0$} -(\alpha_{g'^{-1}})_{*}(\mathrm{d}\nu_0|_{\sigma(x)}+ \frac{1}{2}[\nu_0|_{\sigma(x)}\wedge \nu_0|_{\sigma(x)}]+ \alpha_{*}((\id,g)^{*}\fa{\Omega}_{A',B'}|_x \wedge \nu_0|_{\sigma(x)}))
\\&\eqcref{eq:OmegaAB} -(\alpha_{g'^{-1}})_{*}(\mathrm{d}\nu_0|_{\sigma(x)}+ \frac{1}{2}[\nu_0|_{\sigma(x)}\wedge \nu_0|_{\sigma(x)}]+ \alpha_{*}((\mathrm{Ad}_{g(x)}^{-1}(A'_x)+\theta_{g(x)}) \wedge \nu_0|_{\sigma(x)}))
\\&=-(\alpha_{g'^{-1} \cdot g(x)^{-1}})_{*}(\mathrm{d} (\alpha_{g(x)})_{*}(\nu_0|_{\sigma(x)})+\frac{1}{2}(\alpha_{g(x)})_{*}[\nu_0|_{\sigma(x)}\wedge \nu_0|_{\sigma(x)}]
\\&\qquad+(\alpha_{g(x)})_{*}(\alpha_{*}(\mathrm{Ad}^{-1}_{g(x)}(A'_x) \wedge  \nu_0|_{\sigma(x)}))
\\&=-(\alpha_{g'^{-1} \cdot g(x)^{-1}})_{*}(\mathrm{d} (\alpha_{g(x)})_{*}(\nu_0|_{\sigma(x)})+\frac{1}{2}[(\alpha_{g(x)})_{*}(\nu_0|_{\sigma(x)})\wedge (\alpha_{g(x)})_{*}(\nu_0|_{\sigma(x)})]
\\&\qquad+\alpha_{*}(A' _x\wedge  (\alpha_{g(x)})_{*}(\nu_0|_{\sigma(x)})))
\\&=-(\alpha_{g'^{-1} \cdot g(x)^{-1}})_{*}(\mathrm{d}\varphi_x+\frac{1}{2}[\varphi_x\wedge\varphi_x]+\alpha_{*}(A' _x\wedge \varphi_x))
\\&=\varphi_1|_{x,g'}
\end{align*}
\end{comment}
\end{proof}

\begin{proposition}
\label{prop:locfullyfaithful}
\begin{enumerate}[(a)]

\item 
Every point $x\in M$ has an open neighborhood $U\subset M$ such that $L(U)$ is fully faithful. 
\item
For every contractible open set $U \subset M$ the 2-functor $L^{\ff}(U)$ is fully faithful.
\end{enumerate}
\end{proposition}

\begin{proof}
Suppose $(A,B)$ and $(A',B')$ are $\Gamma$-connections on $M$. We consider the Hom-functor
\begin{equation*}
L: \hom_{\con\Gamma M}((A,B),(A',B')) \to \hom_{\zweibunconreg\Gamma M}(\inf I_{A,B},\inf I_{A',B'})\text{.}
\end{equation*}
That a 2-functor is fully faithful means that the Hom-functor is an equivalence of categories. We want to show using \cref{lem:gt} that it is essentially surjective. By \cref{lem:Fsections} we can first choose any contractible open neighborhood $U \subset M$ of $x$ to guarantee the existence of a transition span $\sigma:U \to F$ along $s$. If $\nu$ is regular, then by definition of regularity we can restrict to a smaller open subset such that $\sigma^{*}\mathrm{fcurv}(\nu)=0$. If $\nu$ is fake-flat, the second restriction is unnecessary; thus  \cref{lem:gt} applies to both cases. 

It remains to show that the Hom-functor is fully faithful; there is no difference between the regular and the fake-flat case. We have to analyze the assignment of the natural transformation $\eta_a$ to a gauge 2-transformation $a: (g_1,\varphi_1) \Rightarrow (g_1,\varphi_1)$. From the definition of $\eta_a$ in \cref{eq:def:etaa} we conclude immediately that the assignment is injective. \begin{comment}
The definition was:
\begin{equation*}
\eta_a: \ob{\inf I} \to \mor{\inf I}: (x,g) \mapsto (\id_x,(a(x),g_1(x)g))\text{.}
\end{equation*}
\end{comment}
For surjectivity, suppose $\eta:F_{g_1,\varphi_1} \Rightarrow F_{g_2,\varphi_2}$ is a connection-preserving 2-morphism. We can assume that it is induced from a smooth $\Gamma$-equivariant natural transformation $\tilde\eta$. Its components form a smooth map $\tilde\eta: M \times G \to M \times H \times G$. Source and target matching, naturality, and $\Gamma$-equivariance imply that  $\tilde\eta(x,g)=(x,a(x),g_1(x)g)$ for a smooth map $a:M  \to H$ satisfying
\begin{equation}
\label{eq:Lfullyfaithful:1}
t(a(x))=g_2(x)g_1(x)^{-1}\text{.}
\end{equation}
\begin{comment}
We have 
\begin{align*}
s(\tilde\eta(x,g)&=\phi_{g_1}(x,g)=(x,g_1(x)g)
\\
t(\tilde\eta(x,g))&=\phi_{g_2}(x,g)=(x,g_2(x)g)
\\
\tilde\eta(x,t(h)g)\circ \phi_{g_1}(\id_x,h,g)&= (x,\tilde\eta_H(x),g_1(x)t(h)g)\circ (\id_x,\alpha(g_1(x),h),g_1(x)g)
\\&=  (x,\tilde\eta_H(x)\alpha(g_1(x),h),g_1(x)g)
\\&=(x,\alpha(g_2(x),h)\tilde\eta_H(x),g_1(x)g) \\&=(x,\alpha(g_2(x),h),g_2(x)g) \circ(x,\tilde\eta_H(x),g_1(x)g) 
\\&=\phi_{g_2}(\id_x,h,g) \circ \tilde\eta(x,g)\text{.}
\\
\tilde\eta(x,gg')&=(x,\tilde\eta_H(x,gg'),g_1(x)gg')\\&=(x,\tilde\eta_H(x,g),g_1(x)gg')
\\&=R(\tilde\eta(x,g),\id_{g'})\text{.}
\end{align*}
\end{comment}
That $\eta$ is connection-preserving implies by \cref{re:canpullback:d} that $\tilde\eta^{*}\fb{\Omega}_{A',B'}=\kappa_{1,0} - \kappa_{2,0}$, where $\kappa_{i,0}=(\alpha_{\pr_G^{-1}\cdot g_i(x)^{-1}})_{*}(\varphi_i|_x)$ according to \cref{eq:shiftF:0}. Then we get
\begin{equation}
\label{eq:Lfullyfaithful:2}
\varphi_1-\mathrm{Ad}_{a}^{-1}(\varphi_2)=(\alpha_{g_1})_{*}(s^{*}\kappa_{1,0} - s^{*}\kappa_{2,0})=(\alpha_{g_1})_{*}s^{*}\tilde\eta^{*}\fb\Omega=(\tilde\alpha_{a})_{*}(A)+a^{*}\theta\text{,}
\end{equation}
where $s:M \to M\times G:x \mapsto (x,1)$. \begin{comment}
Indeed, the full calculation is
\begin{align*}
\varphi_1|_x-\mathrm{Ad}_{a(x)}^{-1}(\varphi_2|_x)&=\varphi_1|_x-(\alpha_{g_1(x)g_2(x)^{-1}})_{*}(\varphi_2|_x)
\\&=(\alpha_{g_1(x)g})_{*}(\kappa_{1,0}|_{x,g} - \kappa_{2,0}|_{x,g})
\\&=(\alpha_{g_1(x)g})_{*}\tilde\eta^{*}\fb\Omega|_{x,g} 
\\&= (\alpha_{g_1(x)g})_{*}\fb\Omega|_{(x,a(x),g_1(x)g)}
\\&=(\tilde\alpha_{a(x)})_{*}(A_x)+\theta_{a(x)}
\end{align*}
\end{comment}
\Cref{eq:Lfullyfaithful:1,eq:Lfullyfaithful:2} show that $a:M \to H$ is a gauge 2-transformation. Obviously, $\eta_a=\eta$. This shows that the Hom-functor is full.
\end{proof}

\begin{appendix}

\section{Formulary for calculations in strict Lie 2-algebras}

The formulas presented here are valid for a strict Lie 2-algebra consisting of Lie algebras $\mathfrak{g}$ and $\mathfrak{h}$, a Lie algebra homomorphism $t_{*}:\mathfrak{h} \to \mathfrak{g}$, and a bilinear map $\alpha_{*}:\mathfrak{g} \times \mathfrak{h} \to \mathfrak{h}$. The axioms are:
\begin{align*}
\alpha_{*}([X_1,X_2],Y)&= \alpha_{*}(X_1,\alpha_{*}(X_2,Y))-\alpha_{*}(X_2,\alpha_{*}(X_1,Y))
\\
\alpha_{*}(X,[Y_1,Y_2])&=[\alpha_{*}(X,Y_1),Y_2]+[Y_1,\alpha_{*}(X,Y_2)]\text{,}
\\
\alpha_{*}(t_{*}(Y_1),Y_2)&=[Y_1,Y_2]
\\
t_{*}\alpha_{*}(X,Y) &= [X,t_{*}(Y)]
\end{align*}
Formulas involving the adjoint action:
\begin{align*}
t_{*}\circ \mathrm{Ad}_h &= \mathrm{Ad}_{t(h)}\circ t_{*}
\\
\mathrm{Ad}_a(\alpha_{*}(X,Y))&= \alpha_{*}(\mathrm{Ad}_{t(a)}(X),\mathrm{Ad}_a(Y))
\end{align*}
Formulas involving the map $\alpha_g:H \to H$ defined by $\alpha_g(h) := \alpha(g,h)$:
\begin{align*}
\mathrm{Ad}_g \circ t_{*} &= t_{*} \circ (\alpha_g)_{*}
\\\mathrm{Ad}_{\alpha(g,h)} \circ (\alpha_{g})_{*}&=(\alpha_g)_{*}\circ \mathrm{Ad}_{h}
\\
(\alpha_{t(h)})_{*}&= \mathrm{Ad}_h
\\
\alpha_{*}(\mathrm{Ad}_g(X),Y) &= (\alpha_g)_{*}(\alpha_{*}(X,(\alpha_{g^{-1}})_{*}(Y)))
\\
(\alpha_{g})_{*}(\alpha_{*}(X, Y)) &= \alpha_{*}(\mathrm{Ad}_{g}(X) , (\alpha_{g})_{*} (Y))
\end{align*}
Formulas involving the map $\tilde\alpha_h:G \to H$ defined by $\tilde\alpha_h(g):= h^{-1}\alpha(g,h)$:
\begin{align*}
(\tilde\alpha_{h_1h_2})_{*} &=  \mathrm{   Ad}_{h_2}^{-1} \circ (\tilde\alpha_{h_1} )_{*}+(\tilde\alpha_{h_2})_{*}
\\
(\tilde\alpha_{\alpha(g,h)})_{*} &=  (\alpha_g)_{*} \circ (\tilde\alpha_{h})_{*}\circ \mathrm{Ad}_g^{-1}
\\
(\tilde\alpha_{h^{-1}})_{*} &=  -\mathrm{Ad}_{h} \circ (\tilde\alpha_{h} )_{*}
\\
t_{*} (  (\tilde\alpha_        h)_{*}(X)) &= \mathrm{Ad}_{t(h)}^{-1}(X)-X
\\
((\tilde\alpha_h)_{*}\circ t_{*})(Y)&=\mathrm{Ad}_h^{-1}(Y)-Y
\\
(\tilde\alpha_h)_{*}([X,Y]) &= [(\tilde\alpha_h)_{*}(X), (\tilde\alpha_h)_{*}(Y)]+ \alpha_{*}(X,(\tilde\alpha_h)_{*}(Y))-\alpha_{*}(Y,(\tilde\alpha_h)_{*}(X))
\end{align*}
Formulas involving the exterior derivative:
\begin{align*}
\mathrm{d}\alpha_{*}(\omega\wedge\eta) &= \alpha_{*}(\mathrm{d}\omega\wedge\eta)+(-1)^{\deg(\omega)}\alpha_{*}(\omega\wedge \mathrm{d}\eta)
\\
\mathrm{d}(\alpha_{g})_{*}(\varphi) &= (\alpha_{g})_{*}(\mathrm{d}\varphi)+\alpha_{*}(g^{*}\bar\theta \wedge (\alpha_{g})_{*}(\varphi))
\\
\mathrm{d}\mathrm{Ad}_{g}^{-1}(\varphi)
&= \mathrm{Ad}_g^{-1}(\mathrm{d}\varphi)-[g^{*}\theta \wedge \mathrm{Ad}_g^{-1}(\varphi)]
\\
\mathrm{d}(\tilde\alpha_h)_{*}(\varphi)&= (\tilde\alpha_h)_{*}(\mathrm{d}\varphi)+(-1)^{\deg(\varphi)}\alpha_{*}(\varphi\wedge h^{*}\theta)  - [h^{*}\theta\wedge (\tilde \alpha_h)_{*}(\varphi) ]
\end{align*}

\end{appendix}

%\nocite{*}

\bibliographystyle{kobib}
\bibliography{kobib}

\end{document}